\documentclass[twoside,leqno,symbols-for-thanks,draft]{rmi} 
 
\usepackage[colorlinks,linkcolor=black,citecolor=black,urlcolor=black, linktocpage=true]{hyperref} 
\usepackage{srcltx} 
\usepackage[english]{babel} 
\numberwithin{equation}{section} 
\setinitialpage{1} 
\renewcommand{\qedsymbol}{$\Box$} 
 
\title[A new strategy in the monomial case]{% 
A new strategy for resolution of singularities 
\linebreak 
in the monomial case in positive characteristic} 
 
\author[H. Kawanoue and K. Matsuki]{% 
Hiraku Kawanoue\thanks{% 
The first author was partially supported by the Inamori Foundation.% 
} and Kenji Matsuki} 
 
\address[kawanoue@kurims.kyoto-u.ac.jp]{{\sc Hiraku Kawanoue}: 
Research Institute for Mathematical Sciences, 
Kyoto University, 
Kitashirakawa-Oiwakecho, Sakyo-ku, 
Kyoto 606-8502, Japan 
} 
\address[kmatsuki@math.purdue.edu]{{\sc Kenji Matsuki}: 
Department of Mathematics, 
Purdue University, 
\\ 
150 N. University Street, West Lafayette, IN 47907-2067, USA 
} 
 
\amsclassification[]{14E15} 
 
 
\usepackage{amsmath, amsfonts, amssymb} 
\numberwithin{equation}{section} 
 
\theoremstyle{plain} 
\newtheorem{art_prop}{Proposition} 
\newtheorem{art_thm}{Theorem} 
\newtheorem{art_cor}{Corollary} 
\newtheorem{art_lem}{Lemma} 
\newtheorem*{claim}{Claim} 
 
\theoremstyle{definition} 
\newtheorem{art_def}{Definition} 
 
\theoremstyle{remark} 
\newtheorem{art_rem}{Remark} 
 
\newcommand{\LPO}[1]{\thicklines \put( 45, 45){\line(1, 1){ 10}} 
\put( 45, 55){\line(1,-1){ 10}}\put( 57, 38){#1}\thinlines} 
\newcommand{\CPO}[1]{\thicklines \put( 95, 45){\line(1, 1){ 10}} 
\put( 95, 55){\line(1,-1){ 10}}\put(107, 38){#1}\thinlines} 
\newcommand{\RPO}[1]{\thicklines \put(145, 45){\line(1, 1){ 10}} 
\put(145, 55){\line(1,-1){10}}\put(157, 38){#1}\thinlines} 
\newcommand{\VLL}[2]{\put( 35, 73){#1}\put( 55, 73){#2} 
\put( 50, 20){\line(0, 1){ 60}}} 
\newcommand{\VCL}[2]{\put( 85, 73){#1}\put(105, 73){#2} 
\put(100, 20){\line(0, 1){ 60}}} 
\newcommand{\VRL}[2]{\put(135, 73){#1}\put(155, 73){#2} 
\put(150, 20){\line(0, 1){ 60}}} 
\newcommand{\RDL}[2]{\put(135, 73){#1}\put(155, 73){#2} 
\put(150, 20){\line(0, 1){ 8}}\put(150,30){\line(0, 1){ 5}} 
\put(150, 37){\line(0, 1){ 5}}\put(150, 44){\line(0, 1){ 5}} 
\put(150, 51){\line(0, 1){ 5}}\put(150, 58){\line(0, 1){ 5}} 
\put(150, 65){\line(0, 1){ 5}}\put(150, 72){\line(0, 1){ 8}}} 
\newcommand{\HOL}[2]{\put(195, 55){#1}\put(190, 40){#2} 
\put(  0, 50){\line( 1,0){ 200}}} 
\newcommand{\HDL}[2]{\put(195, 55){#1}\put(190, 40){#2} 
\put(  0, 50){\line( 1,0){ 8}}\put( 12, 50){\line( 1,0){ 6}} 
\put( 22, 50){\line( 1,0){ 6}}\put( 32, 50){\line( 1,0){ 6}} 
\put( 42, 50){\line( 1,0){ 6}}\put( 52, 50){\line( 1,0){ 6}} 
\put( 62, 50){\line( 1,0){ 6}}\put( 72, 50){\line( 1,0){ 6}} 
\put( 82, 50){\line( 1,0){ 6}}\put( 92, 50){\line( 1,0){ 6}} 
\put(102, 50){\line( 1,0){ 6}}\put(112, 50){\line( 1,0){ 6}} 
\put(122, 50){\line( 1,0){ 6}}\put(132, 50){\line( 1,0){ 6}} 
\put(142, 50){\line( 1,0){ 6}}\put(152, 50){\line( 1,0){ 6}} 
\put(162, 50){\line( 1,0){ 6}}\put(172, 50){\line( 1,0){ 6}} 
\put(182, 50){\line( 1,0){ 6}}\put(192, 50){\line( 1,0){ 8}} 
} 
\newenvironment{ZU}[1]{% 
\begin{center}\begin{picture}(220,78){#1}\end{picture}\end{center} 
\vspace{-20pt} 
} 
 
\begin{document} 
 
%%%%%%%%%%%%%%%%%%%%%%%%%%%%%%%%% 
% 
%  Insert the abstract 
% 
%%%%%%%%%%%%%%%%%%%%%%%%%%%%%%%%% 
 
\begin{abstract} 
According to our approach for resolution of singularities in positive characteristic (called the Idealistic Filtration Program, alias the I.F.P. for short) the algorithm is divided into the following two steps: 
 
\noindent\quad 
Step 1. Reduction of the general case to the monomial case. 
 
\noindent\quad 
Step 2. Solution in the monomial case. 
 
\noindent While we have established Step 1 in arbitrary dimension, 
Step 2 becomes very subtle and difficult in positive characteristic.  This is in clear contrast to the classical setting in characteristic zero, where the solution in the monomial case is quite easy. 
In dimension 3, we provided an invariant, inspired by the work of Benito-Villamayor, which establishes Step 2. 
In this paper, we propose a new strategy to approach Step 2, and provide a different invariant in dimension 3 based upon this strategy. 
The new invariant increases from time to time (the well-known Moh-Hauser jumping 
phenomenon), while it is then shown to eventually decrease.  Since the old invariant 
strictly decreases after each transformation, this may look like a step backward rather than forward.  However, the construction of the new invariant is more faithful to the original philosophy of Villamayor, and we believe that the new strategy has a better fighting chance in higher dimensions. 
\end{abstract} 
 
%%%%%%%%%%%%%%%%%%%%%%%%%%%%%%%%% 
% 
% Body of the article 
% Please insert here the TeX source of your paper 
% (except the bibliography) 
% 
%%%%%%%%%%%%%%%%%%%%%%%%%%%%%%%%% 
 
\setcounter{tocdepth}{1} 
\tableofcontents 
\begin{section}*{Acknowledgement} 
 
%---------------------------------------------------------------------------- 
 
We learned ``the philosophy of Villamayor'' through his papers with 
A.~Benito and A.~Bravo, 
and through the numerous discussions with them 
and S.~Encinas 
during our stays in Madrid.  It is needless to say that, to Moh and Hauser, we owe most of the ideas used in our analysis of the jumping 
phenomenon and eventual decrease.\footnote{The analysis of the jumping 
phenomenon and eventual decrease is done in the monomial case in our setting, while the classical analysis by Moh or Hauser is done in a different setting without any reference to the monomial case.  Therefore, even though we owe most of the ideas to Moh and Hauser, our argument is carried out logically independent of their papers \cite{Hau08}\cite{Hau10}\cite{Moh}.} We thank 
O.~Villamayor, T.T.~Moh, and H.~Hauser 
for their teaching and encouragement both at the professional level and personal one.  The first author thanks 
S.~Perlega 
for the intensive discussions on the subject of the jumping 
phenomenon and eventual decrease during his stay at 
the Erwin Schr\"odinger Institute 
in the winter of 2012.  The main body of this work was done while the second author was staying at Research Institute for Mathematical Sciences at Kyoto University in the summers of 2014 and 2015.  We would like to express our deepest gratitude to 
S.~Mori and S.~Mukai 
for their warm hospitality and consistent support as well as the mathematical vision that has been the guiding light for us.  We would like to thank 
S.~Helmke and N.~Nakayama 
for their critical advice through long hours of conversations in the seminar.  It is the encouragement of 
M.~Kashiwara 
that has kept us going with our project over the period of more than 10 years. 
\end{section} 
\begin{section}{Introduction} 
\begin{subsection}{Reformulation of the problem of resolution of singularities} 
The problem of resolution of singularities, i.e., the problem of finding, given an algebraic variety $X$, a proper birational morphism $X \overset{\pi}\leftarrow \widetilde{X}$ from a nonsingular variety $\widetilde{X}$, is reduced, via the formulation of embedded resolution of singularities, to the following problem of resolution of singularities of a basic object (cf.~\cite{EV}\cite{KM2}\cite{W}): Given a basic object, i.e., the triplet $(W,({\mathcal I},a),E)$ consisting of a nonsingular variety called the ambient space $W$ over a base field $k$, the pair $({\mathcal I},a)$ of a (nonzero) coherent sheaf of ideals ${\mathcal I} \subset {\mathcal O}_W$ and a positive integer $a \in {\mathbb Z}_{> 0}$, and a simple normal crossing divisor $E$ on $W$, construct a sequence of transformations 
\begin{align*} 
\lefteqn{ 
(W,({\mathcal I},a),E) = (W_0,({\mathcal I}_0,a),E_0) \longleftarrow \cdots 
}\\ 
& 
\phantom{ 
(W,({\mathcal I},a),E) 
} 
(W_i,({\mathcal I}_i,a),E_i)  \overset{\pi_{i+1}}\longleftarrow (W_{i+1},({\mathcal I}_{i+1},a),E_{i+1}) \cdots \longleftarrow (W_l,({\mathcal I}_l,a),E_l) 
\end{align*} 
such that $\mathrm{Sing}({\mathcal I}_l,a) = \emptyset$.  Note that the singular locus of the pair of a sheaf of ideals and a positive integer  is defined to be $\mathrm{Sing}({\mathcal I},a) := \{P \in W \mid \mathrm{ord}_P({\mathcal I}) \geq a\}$, where $\mathrm{ord}_P({\mathcal I}) 
= \sup\left\{n \in {\mathbb Z}_{\geq 0} \mid {\mathcal I}_P \subset {\mathfrak m}_P^n\right\}$. 
For the precise definition of a transformation $\pi_{i+1}$, we refer the reader to \cite{EV}\cite{KM2}\cite{W}. 
\end{subsection} 
\begin{subsection}{Solution in characteristic zero}\label{1.2} 
In characteristic zero $\mathrm{char}(k) = 0$, the above problem is solved as follows: in year $i$ of the resolution sequence, we compute the long chain of invariants, defined as a function on $\mathrm{Sing}({\mathcal I}_i,a)$, consisting of the units of the form $(\dim,\text{w-ord},s)$ 
\begin{align*} 
\left({\mathrm{inv}_{\mathrm{classic}}}\right)_i =& (\dim H_i^0,\text{w-ord}_i^0,s_i^0)(\dim H_i^1,\text{w-ord}_i^1,s_i^1) \cdots \\ 
& (\dim H_i^j, \text{w-ord}_i^j,s_i^j) \cdots (\dim H_i^{m-1}, \text{w-ord}_i^{m-1},s_i^{m-1}) \\ 
& 
\begin{cases} 
(\dim H_i^m, \mathrm{w\text{-}ord}_i^m = \infty) \quad\text{or} 
\\ 
(\dim H_i^m, \mathrm{w\text{-}ord}_i^m = 0, \Gamma). 
\end{cases} 
\end{align*} 
Note that, while the subscript ``$i$'' refers to the year proceeding vertically in the resolution sequence, the superscript ``$j$'' refers to the stage in a fixed year proceeding horizontally.  Note also that there is no third factor ``$s$'' in the last unit.  Along with the computation of the long chain of invariants, we simultaneously construct the consecutive modifications 
\begin{align*} 
(H_i^0,({\mathcal J}_i^0,b_i^0),F_i^0), 
%(H_i^1,({\mathcal J}_i^1,b_i^1),F_i^1), 
\dotsc, (H_i^j,({\mathcal J}_i^j,b_i^j),F_i^{j}), \dotsc, 
%\\ \hskip1in 
%(H_i^{m-1},({\mathcal J}_i^{m-1},b_i^{m-1}),F_i^{m-1}), 
(H_i^m,({\mathcal J}_i^m,b_i^m),F_i^m), 
\end{align*} 
where the $H_i^j$'s are the so-called hypersurfaces of maximal contact, starting with $H_i^0 = W_i$.  We choose the center $C_i$ of blow up for the transformation $\pi_{i+1}$ to be the maximum locus of the long chain of invariants $\left({\mathrm{inv}_{\mathrm{classic}}}\right)_i$.  Depending upon the form of the last unit (at a point on the center), this choice of the center has the following local description. 
 
\smallskip 
 
(i) the $\infty$ case, i.e., when the last unit is of the form $(\dim H_i^m, \mathrm{w\text{-}ord}_i^m = \infty)$: In this case, the center $C_i$ is actually the last hypersurface of maximal contact, i.e., $C_i = H_i^m$.  After the transformation, the long chain of invariants strictly decreases, i.e., $\left({\mathrm{inv}_{\mathrm{classic}}}\right)_i > \left({\mathrm{inv}_{\mathrm{classic}}}\right)_{i+1}$. 
 
(ii) {the monomial case}, i.e., when the last unit is of the form $(\dim H_i^m, \mathrm{w\text{-}ord}_i^m = 0, \Gamma)$: In this case, what we have to do is to construct the resolution sequence for the basic object $(H_i^m,({\mathcal J}_i^m,b_i^m),F_i^m)$, which is in the monomial case.  As the name indicates, the ideal ${\mathcal J}_i^m$ on the last hypersurface of maximal contact $H_i^m$ is generated by a monomial of the defining ideals of the components of the boundary divisor $F_i^m$.  We compute the invariant 
$\Gamma$ associated to the monomial.  The center $C_i$ is the maximum locus of the invariant $\Gamma$ (on $H_i^m$), which is the intersection of some components of $F_i^m$.  After the transformation, the long chain of invariants strictly decreases, i.e., $\left({\mathrm{inv}_{\mathrm{classic}}}\right)_i > \left({\mathrm{inv}_{\mathrm{classic}}}\right)_{i+1}$. 
 
\smallskip 
 
Since the long chain of invariants cannot decrease infinitely many times, this process must come to an end after finitely many years, achieving resolution of singularities for the given basic object $(W,({\mathcal I},a),E)$ in characteristic zero. 
\end{subsection} 
\begin{subsection}{A further reformulation in the framework of the I.F.P.}\label{1.3} 
The problem of resolution of singularities of a basic object is further reformulated into the problem of resolution of singularities of an idealistic filtration according to the I.F.P. (See \cite{K}\cite{K2}\cite{KM1}\cite{KM2} for the detail.): Given an triplet $(W,{\mathcal R},E)$, where the pair $({\mathcal I},a)$ in a basic object $(W,({\mathcal I},a),E)$ in the classical setting is replaced by an idealistic filtration ${\mathcal R}$, construct a sequence of transformations 
\begin{align*} 
(W,{\mathcal R},E) = &(W_0,{\mathcal R}_0,E_0) \longleftarrow \cdots  \\ 
&(W_i,{\mathcal R}_i,E_i)  \overset{\pi_{i+1}}\longleftarrow (W_{i+1},{\mathcal R}_{i+1},E_{i+1}) \cdots \longleftarrow (W_l,{\mathcal R}_l,E_l) 
\end{align*} 
such that $\mathrm{Sing}({\mathcal R}_l) = \emptyset$.  Note that an idealistic filtration (of i.f.g. type) is a finitely generated graded ${\mathcal O}_W$-algebra ${\mathcal R} = \bigoplus_{n \in {\mathbb Z}_{\geq 0}}({\mathcal I}_n,n)$, satisfying the condition ${\mathcal O}_W = {\mathcal I}_0 \supset {\mathcal I}_1 \supset {\mathcal I}_2 \cdots  \supset {\mathcal I}_n \supset \cdots$, where ``$n$'' in the second factor specifies the ``level'' of the ideal ${\mathcal I}_n$ in the first factor, and that the singular locus of an idealistic filtration is defined to be $\mathrm{Sing}({\mathcal R}) := \{P \in W \mid \mathrm{ord}_P({\mathcal I}_n) 
\geq n,\ \forall n \in {\mathbb Z}_{\geq 0}\}$.  For the precise definition of a transformation $\pi_{i+1}$, we refer the reader to \cite{KM2}. 
 
We remark that the problem of resolution of singularities of a basic object $(W,({\mathcal I},a),E)$ is reduced to the problem of resolution of singularities of an idealistic filtration $(W,{\mathcal R},E)$ if we set ${\mathcal R} = \bigoplus_{n \in {\mathbb Z}_{\geq 0}}({\mathcal I}^{\lceil \frac{n}{a} \rceil},n)$. 
 
We also remark that what we actually discuss in this paper is the following local version of the above problem (as we discussed only the local version in our previous paper \cite{KM2}): Starting from a closed point $P \in \mathrm{Sing}({\mathcal R}) \subset W$ and its neighborhood, we have a sequence of closed  points and their neighborhoods 
\begin{align*} 
&P\phantom{\hspace{0pt}_0} \in \mathrm{Sing}({\mathcal R}) 
\phantom{\hspace{0pt}_0} \subset W \\ 
&\phantom{P_0 \in \mathrm{Si}}\| 
\\ 
&P_0 \in \mathrm{Sing}({\mathcal R}_0) \subset W_0 \longleftarrow P_1 \in \mathrm{Sing}({\mathcal R}_1) \subset W_1 \longleftarrow \cdots \longleftarrow P_i \in \mathrm{Sing}({\mathcal R}_i) \subset W_i 
\end{align*} 
in the resolution sequence.  After we choose the center $P_i \in C_i \subset \mathrm{Sing}({\mathcal R}_i) \subset W_i$ and take the corresponding transformation $W_i \overset{\pi_{i+1}}\longleftarrow W_{i+1}$ to extend the resolution sequence, the ``devil" tries to choose a closed point $P_{i+1} \in \pi_i^{-1}(P_i) \cap  \mathrm{Sing}({\mathcal R}_{i+1}) \subset W_i$.  If $\pi_i^{-1}(P_i) \cap  \mathrm{Sing}({\mathcal R}_{i+1}) = \emptyset$, then the devil cannot choose a closed point and he loses the game.  If $\pi_i^{-1}(P_i) \cap  \mathrm{Sing}({\mathcal R}_{i+1}) \neq \emptyset$, then the devil chooses a closed point $P_{i+1} \in \pi_i^{-1}(P_i) \cap  \mathrm{Sing}({\mathcal R}_{i+1}) \subset W_i$ and the game continues.  Our task is to provide a prescription on how to choose the center each year so that, no matter how the devil makes his choice, he will end up losing.  That is to say, the prescription should guarantee that we ultimately reach year $i = l-1$ so that, with the choice of the center $C_{l-1}$, we have $\pi_l^{-1}(P_{l-1}) \cap \mathrm{Sing}({\mathcal R}_l) = \emptyset$ after the blow up. 
\end{subsection} 
\begin{subsection}{Our approach in positive characteristic}\label{1.4} 
In positive characteristic $\mathrm{char}(k) = p > 0$, our approach tries to provide a solution to the above problem as follows (Note that our approach is also valid in characteristic zero, where it is indeed complete yielding a slightly different algorithm from the classical one 
discussed in \ref{1.2}.): 
in year $i$ of the resolution sequence (locally constructed as described in the local version of the problem), we compute the long chain of invariants at the closed point $P_i \in \mathrm{Sing}({\mathcal R}_i) \subset W_i$, consisting of the units of the form $(\sigma, \widetilde{\mu},s)$ 
\begin{align*} 
\left({\mathrm{inv}_{\mathrm{new}}}\right)_i =& (\sigma_i^0,\widetilde{\mu}_i^0,s_i^0)(\sigma_i^1,\widetilde{\mu}_i^1,s_i^1) \cdots 
(\sigma_i^j, \widetilde{\mu}_i^j,s_i^j) \cdots (\sigma_i^{m-1}, \widetilde{\mu}_i^{m-1},s_i^{m-1}) \\ 
& (\sigma_i^m, \widetilde{\mu}_i^m, s_i^m) = 
\begin{cases} 
(\sigma_i^m,\widetilde{\mu}_i^m = \infty,0) \quad\text{or} 
\\ 
(\sigma_i^m,\widetilde{\mu}_i^m = 0,0). 
\end{cases} 
\end{align*} 
Note that there is the third factor $s_i^m = 0$ in the last unit, in contrast to the classical long chain of invariants 
$\left(\mathrm{inv}_{\mathrm{classic}}\right)_i$. 
Along with the computation of the long chain of invariants, we simultaneously construct the consecutive modifications 
$$ 
(W_i^0,{\mathcal R}_i^0,E_i^0), (W_i^1,{\mathcal R}_i^1,E_i^1), \dotsc, (W_i^j,{\mathcal R}_i^j,E_i^j), \dotsc 
%(W_i^{m-1},{\mathcal R}_i^{m-1},E_i^{m-1}) 
, 
(W_i^m,{\mathcal R}_i^m,E_i^m), 
$$ 
where actually all the ambient spaces remain the same, i.e., we have 
$$ 
W_i = W_i^0 = W_i^1 = \cdots = W_i^j = W_i^{j+1} = \cdots = W_i^{m-1} = W_i^m. 
$$ 
 
Depending upon the form of the last unit, there are two cases to consider. 
\begin{itemize} 
\setlength{\parskip}{0pt} 
\setlength{\itemsep}{0pt} 
\item[(i)] 
the $\infty$ case, i.e., when the last unit is of the form $(\sigma_i^m,\widetilde{\mu}_i^m = \infty,0)$: In this case, we take the center $C_i$ to be the singular locus of the last modified idealistic filtration ${\mathcal R}_i^m$, i.e., $C_i = \mathrm{Sing}({\mathcal R}_i^m)$.  After the transformation, the long chain of invariants strictly decreases, i.e., $\left({\mathrm{inv}_{\mathrm{new}}}\right)_i > \left({\mathrm{inv}_{\mathrm{new}}}\right)_{i+1}$.  The analysis of this case is complete. 
\item[(ii)] 
the monomial case, i.e., when the last unit is of the form $(\sigma_i^m,\widetilde{\mu}_i^m = 0,0)$: In this case, what we have to do is to construct the resolution sequence for the last modification $(W_i^m,{\mathcal R}_i^m,E_i^m)$, which is in the monomial case in our setting.  Roughly speaking, the idealistic filtration ${\mathcal R}_i^m$ in the monomial case is generated by the elements of an L.G.S. 
(cf.~\ref{1.5} below) 
and a monomial, at a certain level, of the defining ideals of the components of the boundary divisor $E_i^m$.  Unlike the monomial case in the classical setting, however, we \emph{cannot} 
determine the center $C_i$ simply by looking at the monomial.  The analysis of the monomial case in our setting in positive characteristic is subtle and difficult.  \emph{How to choose the center of blow up and how to detect  effectively what is improved after each transformation in the monomial case is the central issue of this paper.}  Once the resolution sequence for $(W_i^m,{\mathcal R}_i^m,E_i^m)$ is complete (or in the middle of the process of constructing the resolution sequence), the long chain of invariants strictly decreases, i.e., $\left({\mathrm{inv}_{\mathrm{new}}}\right)_i > \left({\mathrm{inv}_{\mathrm{new}}}\right)_{i'}$ for some $i' > i$. 
\end{itemize} 
Since the long chain of invariants cannot decrease infinitely many times, this process must come to an end after finitely many years, and should achieve resolution of singularities for the given triplet $(W,{\mathcal R},E)$ in positive characteristic. 
 
\smallskip 
 
The remaining task in order for us to complete our approach is to establish resolution of singularities in the monomial case. 
\end{subsection} 
\begin{subsection}{Strategy in the monomial case}\label{1.5} 
Here we outline the strategy to establish resolution of singularities for the triplet $(W,{\mathcal R},E)$ in the monomial case.  (Strictly speaking, the triplet in the monomial case only appears as the last modification $(W_i^m,{\mathcal R}_i^m,E_i^m)$ in the process of computing the long chain of invariants 
as described in \ref{1.4}. 
But we suppress the subscript ``$i$'' and the superscript ``$m$'' for the sake of simplicity.) 
 
\smallskip 
 
The strategy is divided into the following five steps: 
 
\smallskip 
 
\noindent \textbf{5 steps of the strategy in the monomial case.} 
\begin{itemize} 
\setlength{\parskip}{0pt} 
\setlength{\itemsep}{0pt} 
\item[(1)] Description of the precise setting for the monomial case. 
\item[(2)] Inductive scheme in terms of the invariant $\tau$. 
\item[(3)] Analysis of the tight monomial case (when $\tau = 1$, in arbitrary dimension). 
\item[(4)] Introduction of the new invariant and study of its behavior under 
transformations ($\tau = 1$, in dimension 3). 
\item[(5)] Analysis of the jumping 
phenomenon and eventual decrease ($\tau = 1$, 
in dimension 3). 
\end{itemize} 
 
Now we explain each step of the above strategy more in detail. 
 
\noindent (1) \textbf{Setting}: In section 2, we give the precise description of the setting of the monomial case.  Note that we discuss only the local version of the problem, and hence that our analysis is carried out locally in a neighborhood of a closed point $P \in \mathrm{Sing}({\mathcal R}) \subset W$.  The global version will be discussed elsewhere. 
 
\noindent (2) \textbf{Invariant $\tau$}: The first invariant we compute is the invariant $\tau$, which is just the number of the elements in the Leading Generator System, which plays the role of a collective substitute in the I.F.P. for the notion of a hypersurface of maximal contact in the classical setting.  We explain the inductive scheme in terms of the invariant $\tau$, which reduces our analysis to the case $\tau = 1$, i.e., to the case where there is only one element in the L.G.S.  This confirms the folklore, in our setting, that the essential case to consider in the resolution problem is the case of a hypersurface singularity, i.e., a singularity defined by one equation. 
 
After this point, we concentrate ourselves on the case where $\tau = 1$. 
 
Our idealistic filtration ${\mathcal R}$ is, roughly speaking, generated by the L.G.S. ${\mathbb H} = \{(h,p^e)\}$ consisting of a unique element, and a monomial $(M,a)$ at level ``$a$''.  Via the Weierstrass preparation theorem, the element $h$ is in the following form with respect to a regular system of parameters $(x_1, x_2, \ldots, x_d)$ 
$$h = x_1^{p^e} + a_1x_1^{p^e-1} + \cdots + a_{p^e-1}x_1 + a_{p^e}$$ 
with $a_i \in k[[x_2, \ldots, x_d]]$ and $\mathrm{ord}_P(a_i) > i$ for $i = 1, \ldots, p^e - 1, p^e$.  The central issue turns out to be how to control the last coefficient $a_{p^e}$, as the other coefficients $a_i\ (i = 1, \ldots, p^e - 1)$ are {\it well-controlled} 
(See \ref{4.1} for the precise meaning.). 
 
\textbf{Cleaning}: For the purpose of controlling $a_{p^e}$, we look at its order $\mathrm{ord}_P(a_{p^e})$. 
However, this number depends on the choice of a regular system of parameters. 
For example, if the initial form $\mathrm{In}(a_{p^e})$ is a $p^e$-th power, we may replace the original $x_1$ with $x_1 + \{\mathrm{In}(a_{p^e})\}^{1/p^e}$ and the order increases.  The process of eliminating this ambiguity and making the order well-defined is called 
\emph{cleaning}. 
 
\textbf{Invariant ${\mathrm H}$}: After cleaning and taking the information on the monomial also into consideration, we define the invariant 
${\mathrm H}$ (after Hironaka) 
$$\mathrm{H}(P) := \min\left\{\mathrm{ord}_P(a_{p^e})/p^e, \mathrm{ord}_P(M_{\mathrm{usual}})\right\}.$$ 
(3) \textbf{Tight Monomial Case}: Using the invariant $\mathrm{H}$, we can also define the \it tight \rm monomial $M_{\mathrm{tight}}$.  In general, we have $\mathrm{H}(P) \geq \mathrm{ord}_P(M_{\mathrm{tight}})$.  When the equality $\mathrm{H}(P) = \mathrm{ord}_P(M_{\mathrm{tight}})$ holds, we say that we are in the tight monomial case, following the terminology initiated by Villamayor.  We show in section 4 that, if we are in the tight monomial case, then we can easily accomplish resolution of singularities by computing the invariant  $\Gamma_{\mathrm{tight}}$ and by following the classical procedure. 
 
\smallskip 
 
\noindent \fbox{\textbf{Villamayor's philosophy}} 
 
\smallskip 
 
The strategy for resolution of singularities can therefore be summarized symbolically as follows: 
%$$\begin{array}{c} 
%\bold{general\ case} \\ 
%\downarrow \\ 
%\bold{monomial\ case} \\ 
%\downarrow \\ 
%\bold{monomial\ case\ with\ \tau = 1} \\ 
%\downarrow \\ 
%\bold{tight\ monomial\ case\ (with\ \tau = 1)} \\ 
%\downarrow \\ 
%\bold{finish}.\\ 
%\end{array}$$ 
 
\medskip 
 
\begin{tabular}{rl} 
\textbf{{general case}} 
$\to$& 
\textbf{{monomial case}} 
$\to$ 
\textbf{{monomial case with $\tau=1$}} 
\\ 
$\to$& 
\textbf{{tight monomial case (with $\tau=1$)}} 
$\to$ 
\textbf{{finish.}} 
\end{tabular} 
 
\medskip 
 
\noindent 
The above argument establishes all the procedures 
\emph{except} for the third arrow. 
 
The final remaining task, therefore, is to establish the procedure for the third arrow.  That is to say, we have to establish the procedure to reach the tight monomial case, after reaching the monomial case with $\tau = 1$. 
 
\smallskip 
 
\noindent (4) \textbf{The new invariant $\mathrm{inv}_{\mathrm{MON,\spadesuit}}$}: 
In order to accomplish the final task, we introduce the new invariant 
$\mathrm{inv}_{\mathrm{MON,\spadesuit}}$, which should measure how far we are from the tight monomial case.  A more naive version, the invariant $\mathrm{inv}_{\mathrm{MON,\heartsuit}}(P)$ is defined to be $\mathrm{inv}_{\mathrm{MON,\heartsuit}}(P) = \mathrm{H}(P) - \mathrm{ord}_P(M_{\mathrm{tight}})$, 
which is only natural when we consider the definitions of the invariant ${\mathrm H}$ and the tight monomial case.  We are in the tight monomial case if and only if $\mathrm{inv}_{\mathrm{MON,\heartsuit}}(P) =0$.  Our new invariant $\mathrm{inv}_{\mathrm{MON,\spadesuit}}(P)$ incorporates the information extracted from the Newton polygon associated to the last coefficient $a_{p^e}$.  It is more sensitive to and indicative of what is improved under some transformations than the naive version 
$\mathrm{inv}_{\mathrm{MON,\heartsuit}}$. 
Therefore, it is better suited for the purpose of showing the termination of the algorithm effectively. 
 
Our ultimate goal is to bring this new invariant down to zero, where our algorithm terminates and we are in the tight monomial case.  (We have the inequality $\mathrm{inv}_{\mathrm{MON,\spadesuit}}(P) \geq \mathrm{inv}_{\mathrm{MON,\heartsuit}}(P) \geq 0$ in general, and hence $\mathrm{inv}_{\mathrm{MON,\spadesuit}}(P) = 0$ implies $\mathrm{inv}_{\mathrm{MON,\heartsuit}}(P) = 0$.) 
 
In dimension 3, we prescribe an algorithm to reach the tight monomial case, by analyzing the singular locus of the idealistic filtration. 
 
Our expectation is that the new invariant 
$\mathrm{inv}_{\mathrm{MON,\spadesuit}}$ 
strictly decreases under each transformation prescribed by the algorithm, and hence that it effectively shows the termination of the algorithm.  When we actually analyze the behavior of the invariant 
$\mathrm{inv}_{\mathrm{MON,\spadesuit}}$ in section 5, 
we find that this expectation is met in most of the cases.  However, as our study in section 6 shows, the invariant 
$\mathrm{inv}_{\mathrm{MON,\spadesuit}}$ 
strictly increases from time to time in some special cases !  That is to say, we rediscover the famous ``Moh-Hauser jumping 
phenomenon'' in our setting. 
 
\noindent (5) \textbf{Moh-Hauser jumping 
phenomenon and Eventual Decrease}: In section 6, we analyze the ``Moh-Hauser jumping 
phenomenon'' more in detail.  Then as a consequence of this analysis, we show that, when the invariant 
$\mathrm{inv}_{\mathrm{MON,\spadesuit}}$ strictly increases, 
it eventually decreases after some more transformations 
to a value lower than the original one.  This \emph{eventual decrease} is enough to guarantee that our algorithm terminates after finitely many times.  Our argument is an extension of the one given by Moh and Hauser, but it seems to be new in the sense that we remove the restriction on the power of the leading term of the monic polynomial $h$ used by Hauser, and we carry out the entire argument in the monomial case to show the eventual decrease, which can be considered as the ``Moh's stability theorem'' in our setting. 
 
In the last section 7 of our paper, we make a brief comparison of the new invariant 
$\mathrm{inv}_{\mathrm{MON,\spadesuit}}$ with the old one 
$\mathrm{inv}_{\text{MON}}$ in our previous paper \cite{KM2}. 
We also mention why we think that the new invariant 
$\mathrm{inv}_{\mathrm{MON,\spadesuit}}$ has 
a better fighting chance in higher dimensions than the previous one 
$\mathrm{inv}_{\text{MON}}$. 
 
\smallskip 
 
This finishes the description of the outline of our paper. 
\end{subsection} 
\begin{subsection}{Assumption on the base field} 
In this paper, we always assume that the base filed $k$ is an algebraically closed field of characteristic zero $\mathrm{char}(k) = 0$ or of positive characteristic $\mathrm{char}(k) = p > 0$ for simplicity, even though our algorithm can be shown to be valid over any perfect field by the argument of Galois descent. 
\end{subsection} 
\end{section} 
\begin{section}{Description of the precise setting for the monomial case}\label{2.1} 
We give the precise description of the setting for the triplet $(W,{\mathcal R},E)$ in the monomial case at a closed point $P \in \mathrm{Sing}({\mathcal R}) \subset W$.  Note that, while the description is given at the analytic level (i.e. at the level of completion) and the invariants are computed also at the analytic level, the center is chosen at the algebraic level and hence the procedure of the algorithm is carried out at the algebraic level. 
 
\smallskip 
 
\noindent \fbox{\textbf{Setting}} 
 
\smallskip 
 
There exists a regular system of parameters $X = (x_1, \ldots, x_t, x_{t+1}, \ldots, x_d)$ taken from $\widehat{{\mathfrak m}_P} \subset \widehat{{\mathcal O}_{W,P}}$ satisfying the following conditions. 
\begin{itemize} 
\item[1.] 
The elements of the L.G.S. ${\mathbb H} = \{(h_{\alpha},p^{e_{\alpha}})\}_{\alpha = 1}^t \subset \widehat{{\mathcal R}_P} := {\mathcal R}_P \otimes_{{\mathcal O}_{W,P}} \widehat{{\mathcal O}_{W,P}}$, with $0 \leq e_1 \leq \cdots \leq e_t$, satisfy the equations $h_{\alpha} = x_{\alpha}^{p^{\alpha}} \hskip.05in \text{mod} \hskip.05in \widehat{{\mathfrak m}_P}^{p^{e_{\alpha}} + 1}$ for $\alpha = 1, \ldots, t$. 
\item[2.] 
There exist a monomial $M = \prod_{D \in E_{\mathrm{young}}}x_D^{m_D}$, 
where $E_{\mathrm{young}}$ is a subset of $E$ (See \cite{KM2} for the definition of $E_{\mathrm{young}}$.) and where the defining equation $x_D$ of an irreducible component $D \in E_{\mathrm{young}}$ coincides with one of $(x_{t+1}, \ldots, x_d)$, and a positive integer $a \in {\mathbb Z}_{> 0}$ such that $(M,a) \in \widehat{{\mathcal R}_P}$ with $\sum m_D > a$. 
\item[3.] 
For any element $(f,\lambda) \in \widehat{{\mathcal R}_P}$ with $f = \sum c_{f,B}H^B$ being the power series expansion of $f$ with respect to ${\mathbb H}$ and $X$ (cf.~\cite{KM1}\cite{KM2}), we have $(M^{1/a})^{\lambda} \mid  c_{f,{\mathbb O}}$.  That is to say, we have $c_{f,{\mathbb O}}/x_D^{\lceil m_D \cdot \lambda/a \rceil} \in \widehat{{\mathcal O}_{W,P}}$ for any $D \in E_{\mathrm{young}}$.  Since $(c_{f,B}, \max\{0,\lambda - |[B]|\}) \in \widehat{{\mathcal R}_P}$ by the formal coefficient lemma (cf.~\cite{KM1}\cite{KM2}) and since $(c_{f,B})_{\mathbb O} = c_{f,B}$, we have in general $(M^{1/a})^{\max\{0,\lambda - |[B]|\}} \mid  c_{f,B}$. 
\item[4.] 
The idealistic filtration $\widehat{{\mathcal R}_P}$ is closed under the partial differentiations with respect to $(x_1, \ldots, x_t)$.  That is to say, we have the implication $(f,\lambda) \in \widehat{{\mathcal R}_P} \Longrightarrow 
\left(\frac{\partial^n}{\partial x_{\alpha}^n}f, \max\{0,\lambda - n\}\right) \in \widehat{{\mathcal R}_P}$ 
for any $\alpha = 1, \ldots, t$ and any $n \in {\mathbb Z}_{\geq 0}$.  Note that, when we write $\frac{\partial^n}{\partial x_{\alpha}^n}$, we include the partial derivatives $\frac{\partial^{p^e}}{\partial x_{\alpha}^{p^e}}$ of Hasse-Weil type with $\frac{\partial^{p^e}}{\partial x_{\alpha}^{p^e}}(x_{\alpha}^{p^e}) = 1$. 
\end{itemize} 
\begin{art_rem} 
\item[(1)] 
When the idealistic filtration $(W,{\mathcal R}, E)$ in the monomial case appears as the last modification $(W_i^m,{\mathcal R}_i^m,E_i^m)$ in the process of computing the long chain of invariants, the corresponding last unit is of the form $(\sigma_i^m,\widetilde{\mu}_i^m,s_i^m) = (\sigma_i^m,0,0)$.  Since $s = s_i^m$ is the number of the components of $E_{\text{aged}} = E \setminus E_{\mathrm{young}}$ passing through the point $P$, there are actually no components of $E_{\text{aged}}$ but only those of $E_{\mathrm{young}}$ locally around $P$. 
 
\item[(2)] 
We remark that condition 3 of the setting above is the exact meaning of the statement 
(cf.~\ref{1.4} (ii)) 
``Roughly speaking,  in our monomial case, the idealistic filtration ${\mathcal R}_P$ is generated by the elements of an L.G.S. ${\mathbb H}$ and a monomial $M$.'' 
 
\item[(3)] 
In our algorithm, one cannot expect in general to have the idealistic filtration $\widehat{{\mathcal R}_P}$ in the monomial case to be ${\mathcal D}$-saturated, ${\mathcal D}_E$-saturated, or ${\mathcal D}_{E_{\mathrm{young}}}$-saturated, since we only take the pull-back of the idealistic filtration under transformation, without taking any further saturation, when the invariant $\sigma$ stays the same (cf.~\cite{KM2}).  It seems that the partial saturation as described in condition 4 in the setting is the best we can hope for in our current algorithm.  Nevertheless, we remark that condition 4 puts very strong restrictions on the idealistic filtration $\widehat{{\mathcal R}_P}$ in the monomial case. 
\end{art_rem} 
\end{section} 
\begin{section}{Inductive scheme in terms of the invariant $\tau$} 
\begin{subsection}{Invariant $\tau$} 
The first invariant we look at is the invariant $\tau$, which is just the number of the elements in the L.G.S., i.e., $0 \leq \tau = \# {\mathbb H} = t \leq d = \dim W$. 
\end{subsection} 
\begin{subsection}{Inductive scheme in terms of $\tau$} 
We summarize, in the table below, the procedure for resolution of singularities in the monomial case, according to the value of the invariant $\tau$ (cf.~\cite{KM2}). 
 
\smallskip 
 
\noindent \textbf{Table of the procedures according to the value of the invariant $\tau$.} 
 
\smallskip 
 
$\boxed{\tau = 0}$ In this case, the problem of resolution of singularities for the triplet $(W,{\mathcal R},E)$ is reduced to the problem of resolution of singularities for the basic object $(W, ((M),a), E)$, which is in the monomial case in the classical sense.  For the latter, we can just compute the invariant $\Gamma$, and carry out the procedure for resolution of singularities accordingly.  (In the middle of the procedure, the value of the invariant $\sigma$ may decrease for the transformation of the triplet $(W,{\mathcal R},E)$.  Then we go back to the reduction step 
``\textbf{general case $\to$ monomial case}'' with the decreased value of the invariant $\sigma$.) 
 
$\boxed{\tau = 1}$ This is the \textit{most difficult} case. 
 
$\boxed{\tau = j}$ $(j = 2, \ldots, d-1)$ The analysis for this case is similar to the one for the case where $\dim = d - 1 \hskip.03in \& \hskip.03in \tau = j - 1$.  Therefore, we can use the induction on dimension to analyze the case. 
 
$\boxed{\tau = d}$ This case does \emph{not} happen. 
 
\smallskip 
 
Therefore, according to the inductive scheme on dimension associated to the value of the invariant $\tau$ described as above, what remains is the task of figuring out the resolution procedure in the case $\tau = 1$, i.e., the case where there is a unique element in the L.G.S. 
 
\smallskip 
 
The ``similarity'' between the analysis for the case $\dim = d \hskip.03in \& \hskip.03in \tau = j$ where $j = 2, \ldots, d-1$ and the one for the case $\dim = d - 1 \hskip.03in \& \hskip.03in \tau = j - 1$ is straightforward, but not trivial.  It requires some argument.  We publish the detail of the argument for the similarity, e.g., between the analysis for the case $\dim = 4 \hskip.03in \& \hskip.03in j = 2$ and the one for the case $\dim = 4 -1 = 3 \hskip.03in \& \hskip.03in \tau = 2 - 1 = 1$ elsewhere. 
 
\begin{art_rem} 
\item[(1)] 
Our analysis of the algorithm is, after all, reduced to the case where there is only one element in the L.G.S. discussed as above.  This seems to confirm, in our setting, the folklore, which says: in order to solve the problem of resolution of singularities, the ``essential'' case is the case of a hypersurface singularity, i.e., a singularity defined by only one equation. 
\item[(2)] 
In positive characteristic, the case where $\dim = d \hskip.03in \& \hskip.03in \tau = 1$ is \it not \rm reduced to the case where $\dim = d - 1 \hskip.03in \& \hskip.03in \tau = 1 -1$.  Even after we reach the monomial case, the analysis of the case $\dim = d \hskip.03in \& \hskip.03in \tau = 1$ remains as a challenge, while the other cases with $\tau > 1$ are reduced to the lower dimensional ones.  In characteristic zero, however, the case where $\dim = d \hskip.03in \& \hskip.03in \tau = 1$ is \it indeed \rm reduced to the case where $\dim = d-1 \hskip.03in \& \hskip.03in \tau = 1 -1 = 0$.  When the value of the invariant $\tau = 0$, we achieve resolution of singularities easily by using the invariant $\Gamma$ as in the classical setting.  Therefore, our algorithm is complete in characteristic zero by induction on dimension. 
\item[(3)] 
We refer the reader to 5.1 in \cite{KM2} for a more detailed case analysis according to the value of the invariant $\tau$ in dimension $d = 3$. 
\end{art_rem} 
 
\textit{In the rest of this paper, we concentrate ourselves on the case where the invariant $\tau = 1$.} 
\end{subsection} 
\end{section} 
\begin{section}{Analysis of the tight monomial case} 
\begin{subsection}{A further analysis of the monomial case with $\tau = 1$}\label{4.1} 
Note first that, since we are in the monomial case with $\tau = 1$, there is a unique element in the L.G.S. ${\mathbb H} = \{(h,p^e)\}$. 
 
\smallskip 
 
\noindent (\textbf{Weierstrass Form}) Via the Weierstrass Preparation Theorem, we may further assume that $h$ is of the following form 
$$h = x_1^{p^e} + a_1x_1^{p^e - 1} + a_2x_1^{p^e - 2} + \cdots + a_{p^e - 1}x_1 + a_{p^e},$$ 
with $a_i \in k[[x_2, \ldots, x_d]]$ and $\mathrm{ord}_P(a_i) > i$ for  $i = 1, \ldots, p^e - 1, p^e$. 
 
\smallskip 
 
\noindent (\textbf{Control over the coefficients $a_i$ for $i = 1, \ldots, p^e - 1$}) 
We observe that the coefficients $a_i \hskip.03in (i = 1, \ldots, p^e - 1)$ are well-controlled in the sense that $(M^{1/a})^i$ divides $a_i$ for $i = 1, \ldots, p^e - 1$, where $(M,a) \in \widehat{{\mathcal R}_P}$ is the monomial 
$M = \prod_{D \in E_{\mathrm{young}}}x_D^{m_D}$ at level $a > 0$ described in condition  2 of the setting.  That is to say, we have 
$x_D^{\lceil m_D \cdot i/a\rceil} \mid  a_i\ \forall D \in E_{\mathrm{young}}$ 
 for $i = 1, \ldots, p^e - 1$.  This can be seen easily from condition 3 of the setting, if one considers the fact that $(\frac{\partial^i}{\partial x_1^i}h, p^e - i) \in \widehat{{\mathcal R}_P}$ for $i = 1, \ldots, p^e - 1$, which follows from condition 4 of the setting, and that the constant term $c_{g_i,{\mathbb O}}$ of $g_i = \frac{\partial^i}{\partial x_1^i}h$, in the power series expansion with respect to $X$ and ${\mathbb H}$, is $g_i$ itself, i.e., $c_{g_i,{\mathbb O}} = g_i$. 
 
\smallskip 
 
\noindent (\textbf{Control over the last coefficient $a_{p^e}$}) \textit{ The central issue of the analysis of the monomial case with $\tau = 1$ is how to control the last coefficient $a_{p^e}$.} 
 
\smallskip 
 
\begin{art_rem} 
The control over the coefficients $a_i$ for $i = 1, \ldots, p^e - 1$ mentioned above allows us to ``ignore'' these coefficients in our analysis of the monomial case with $\tau = 1$, as if we were dealing with the hypersurface defined by the equation $x_1^{p^e} + a_{p^e} = 0$, which would provide a purely inseparable extension. 
\end{art_rem} 
\end{subsection} 
\begin{subsection}{Invariant $\mu$} 
We define the invariant $\mu$ by the formula 
$$\mu(P) := \mathrm{ord}_P(M)/a = \sum\nolimits_{D \in E_{\mathrm{young}}}m_D/a = \mathrm{ord}_P(M_{\mathrm{usual}}),$$ 
where $(M,a) \in \widehat{{\mathcal R}_P}$ is the monomial $M = \prod_{D \in E_{\mathrm{young}}}x_D^{m_D}$ at level $a > 0$ described in condition 2 of the setting.  (See Definition 2 for the definition of $M_{\mathrm{usual}}$.) 
For the generic point $\xi_D$ of a component $D \in E_{\mathrm{young}}$, we also define 
$$\mu(\xi_D) := \mathrm{ord}_{\xi_D}(M)/a = m_D/a = \mathrm{ord}_{\xi_D}(M_{\mathrm{usual}}).$$ 
\end{subsection} 
\begin{subsection}{Invariant $\mathrm{H}$}\label{4.3} 
We define the invariant $\mathrm{H}$, which sits at the heart of our analysis of the monomial case with $\tau = 1$, 
through the process of ``cleaning''. 
 
\begin{art_def} Let the setting be as described 
in \ref{2.1} and \ref{4.1} above. 
 
\smallskip 
 
\noindent(\textbf{Slope})\quad We define the slope of $h$ with respect to the regular system of parameters $X = (x_1, x_2, \ldots, x_d)$ by the formula 
$$ 
\mathrm{Slope}_{h,X}(P) 
:= \min\left\{\mathrm{ord}_P(a_{p^e})/p^e,\mu(P)\right\}. 
$$ 
Note that we have $\mathrm{Slope}_{h,X}(P) \leq\mathrm{ord}_P(a_i)/i$ 
for $i = 1, \ldots, p^e - 1$ 
because of the control over the coefficients $a_i$, which implies 
$\mathrm{ord}_P(a_i)/i \geq \mathrm{ord}_P(M_{\mathrm{usual}}) = \mu(P)$. 
%\begin{align*} 
%\mathrm{Slope}_{h,X}(P) &:= \min\left\{\mathrm{ord}_P(a_{p^e})/p^e,\mu(P)\right\} \\ 
%&\phantom{:}= \min\left(\left\{\mathrm{ord}_P(a_i)/i 
%\mid i = 1, \ldots, p^e\right\}\cup \{\mu(P)\}\right). 
%\end{align*} 
%Note that the second equality holds 
%because of the control over the coefficients $a_i$ for 
%$i = 1, \ldots, p^e - 1$, which implies 
%$\mathrm{ord}_P(a_i)/i \geq \mathrm{ord}_P(M_{\mathrm{usual}}) = \mu(P)$ 
%for $i = 1, \ldots, p^e - 1$. 
 
\smallskip 
 
\noindent(\textbf{Well-adaptedness})\quad We say $h$ is well-adapted at $P$ with respect to $X$ if one of the following two conditions holds: 
\begin{itemize} 
\setlength{\parskip}{0pt} 
\setlength{\itemsep}{0pt} 
\item[A.] 
$\mathrm{Slope}_{h,X}(P) = \mu(P)$, or 
\item[B.] 
$\mathrm{Slope}_{h,X}(P) = \mathrm{ord}_P(a_{p^e})/p^e <\mu(P)$ and the initial form $\mathrm{In}_P(a_{p^e})$ is not a $p^e$-th power. 
\end{itemize} 
 
Similarly, we say $h$ is well-adapted at $\xi_{H_x}$ with respect to $X$, where $\xi_{H_x}$ is the generic point of the hypersurface $H_x = \{x = 0\}$ in $E_{\text{young}}$, if one of the following two conditions holds: 
\begin{itemize} 
\setlength{\parskip}{0pt} 
\setlength{\itemsep}{0pt} 
\item[A.] 
$\mathrm{Slope}_{h,X}(\xi_{H_x}) = \mu(\xi_{H_x}) = m_{H_x}/a$, or 
 \item[B.] 
$\mathrm{Slope}_{h,X}(\xi_{H_x}) = \mathrm{ord}_{\xi_{H_x}}(a_{p^e})/p^e <\mu(\xi_{H_x})$ and the initial form $\mathrm{In}_{\xi_{H_x}}(a_{p^e})$ is not a $p^e$-th power. 
\end{itemize} 
 
Note that if 
$$a_{p^e} = \sum\nolimits_{|I| \geq m}c_IX^I = x^r\{g_x + x \cdot \omega_x\},$$ 
where 
$$ 
c_I \in k, \quad 
g_x \in k[[x_2, \ldots, \overset{\vee}x, \ldots, x_d]]\setminus\{0\}, \quad 
\omega_x \in k[[x_2, \ldots, x_d]] 
$$ 
with $\overset{\vee}{x}$ indicating the omission of $x$, 
and where 
$$ 
m = \mathrm{ord}_P(a_{p^e}), \quad 
r = \mathrm{ord}_{\xi_{H_x}}(a_{p^e}), 
$$ 
then we have 
$$ 
\mathrm{In}_P(a_{p^e}) = \sum\nolimits_{|I| = m} c_IX^I, \quad 
\mathrm{In}_{\xi_{H_x}}(a_{p^e}) = x^r \cdot g_x. 
$$ 
\end{art_def} 
\begin{art_prop} 
\item[{\rm (1)}] 
\emph{(Cleaning makes the system well-adapted)}\quad 
Given the unique element $h$ (in the Weierstrass form) in the L.G.S. and a regular system of parameters $X = (x_1, x_2, \ldots, x_d)$ as described 
in \ref{4.1}, 
satisfying the conditions of the setting 
\ref{2.1}, 
we can apply the process of ``cleaning'' and find another regular system of parameters $X' = (x_1', x_2, \ldots, x_d)$ such that $h$ is well-adapted at $P$ with respect to $X'$.  (Note that the only difference between $X$ and $X'$ lies in the first coordinates, i.e., $x_1$ and $x_1'$.)  Moreover, by applying the process of cleaning further, we can find yet another regular system of parameters $X'' = (x_1'', x_2, \ldots, x_d)$ such that $h$ is well-adapted simultaneously at $P$ and at the generic points of all the components of $E_{young}$ passing through $P$ with respect to $X''$.  (Note that again the only difference lies in the first coordinates.)  We remark that the process of cleaning is carried out in such a way that $h \hskip.03in \& \hskip.03in X'$ (and also $h \hskip.03in \& \hskip.03in X''$) satisfy all the conditions of the setting described in 
\ref{2.1}. 
\item[{\rm (2)}] 
\emph{(Independence of the slope when well-adapted)}\quad 
If $h$ is well-adapted at $P$ (resp. at $\xi_{H_x}$) with respect to $X$, then $\mathrm{Slope}_{h,X}(P)$ (resp. $\mathrm{Slope}_{h,X}(\xi_{H_x})$) is independent of the choice of $h$ and $X$. 
\item[{\rm (3)}] 
\emph{(Independence of the residual order along a bad divisor when well-adapted)}\quad 
If $h$ is well-adapted at $\xi_{H_x}$ with respect to $X$ and $\mathrm{Slope}_{h,X}(\xi_{H_x}) < \mu(\xi_{H_x}) = m_{H_x}/a$, then 
$\mathrm{res\text{-}ord}^{(p^e)}_P(x^r \cdot g_x)$ 
is independent of the choice of $h \hskip.03in \& \hskip.03in X$.  Moreover, by applying the process of cleaning further, we may assume 
$\mathrm{res\text{-}ord}^{(p^e)}_P\left(x^r \cdot g_x\right) = \mathrm{ord}_P\left(x^r \cdot g_x\right)$. 
Note that 
$\mathrm{res\text{-}ord}^{(p^e)}_P$ 
is the lowest degree of the nonzero and non-$p^e$-th power terms appearing in the Taylor expansion. 
\end{art_prop} 
\begin{proof} We refer the reader to Propositions 5 and 7 in \cite{KM2} for the detail of a proof.  Note that the proof in \cite{KM2} is given in dimension $d = 3$, but the same proof works in arbitrary dimension.  Here we only make a quick remark about the process of ``cleaning'':  Suppose $h$ is not well-adapted at $P$ with respect to $X = (x_1, x_2, \ldots, x_d)$, i.e., we are not in Case A or Case B.  This is equivalent to saying that $\mathrm{Slope}_{h,X}(P) = \mathrm{ord}_P(a_{p^e})/p^e <\mu(P)$ and the initial form $\mathrm{In}_P(a_{p^e})$ is actually a $p^e$-th power.  We set $X^{\ast} = (x_1^{\ast}, x_2, \ldots, x_d)$ with $x_1^{\ast} = x_1 + \{\mathrm{In}_P(a_{p^e})\}^{1/p^e}$.  Then, using the control over the coefficients $(M^{1/a})^i \mid  a_i$ for $i = 1, \ldots, p^e - 1$,  we observe $\mathrm{ord}_P(a_{p^e})/p^e < \mathrm{ord}_P(a_{p^e}^{\ast})/p^e$.  Since the latter is bounded from above by $\mu(P)$ if $h$ is not well-adapted with respect to $X^{\ast}$, this process has to come to an end after finitely many repetitions with $h$ being well-adapted at $P$ with respect to $X' = (x_1', x_2, \ldots, x_d)$.  This is the process of cleaning. 
\end{proof} 
\begin{art_def} We define the invariant $\mathrm{H}$ by the formula 
$$\mathrm{H}(P) = \mathrm{Slope}_{h,X}(P)$$ 
where $h$ is well-adapted at $P$ with respect to $X$.  This is independent of the choice of $h$ and $X$ by Proposition 1 (2).  The invariant $\mathrm{H}(\xi_{H_x})$ is defined similarly. 
\end{art_def} 
\end{subsection} 
\begin{subsection}{Analysis of the tight monomial case} 
\begin{art_def} We define the tight monomial $M_{\mathrm{tight}}$ and the usual monomial $M_{\mathrm{usual}}$ by the formulas 
$$ 
M_{\mathrm{tight}} = \prod\nolimits_{D \in E_{\mathrm{young}}}x_D^{\mathrm{H}(\xi_D)},\quad 
M_{\mathrm{usual}} = M^{1/a} = \prod\nolimits_{D \in E_{\mathrm{young}}}x_D^{m_D/a}.\\ 
$$ 
Note that the powers $\mathrm{H}(\xi_D)$ in $M_{\mathrm{tight}}$ and $m_D/a$ in $M_{\mathrm{usual}}$ are rational numbers.  We say $M_{\mathrm{tight}}$ divides $M_{\mathrm{usual}}$, i.e., $M_{\mathrm{tight}} \mid  M_{\mathrm{usual}}$, since we have $\mathrm{H}(\xi_D) \leq \mu(\xi_D) = m_D/a \hskip.1in \forall D \in E_{\mathrm{young}}$ by definition. 
\end{art_def} 
 
It is convenient in our analysis to introduce the notion of a point (or a divisor) being 
``\textit{good\,/\,bad}'' 
according to its behavior in terms of the invariant $\mathrm{H}$. 
\begin{art_def} We say $P$ is a 
good (resp.~bad) point if $\mathrm{H}(P) = \mu(P)$ 
(resp.~$\mathrm{H}(P) < \mu(P)$). 
Similarly, we say 
$H_x$ is a good (resp.~bad) divisor if 
$\mathrm{H}(\xi_{H_x}) = \mu(\xi_{H_x})$ 
(resp.~$\mathrm{H}(\xi_{H_x}) < \mu(\xi_{H_x})$). 
\end{art_def} 
\begin{art_rem} 
It is straightforward to see that if we blow up a good (resp. bad) point, then the exceptional divisor is accordingly a good (resp. bad) divisor (cf.~Lemma 4 in \cite{KM2}). 
\end{art_rem} 
\begin{art_def} Let the setting be as described 
in \ref{2.1}, \ref{4.1} and \ref{4.3}. 
Suppose that the divisor $D = H_x = \{x = 0\}$ in $E_{\mathrm{young}}$ is bad.  Then we define the invariants $\rho_D(P)$ and $w\text{-}\rho_D(P)$ by the formulas 
$$\begin{cases} 
\phantom{w\text{-}}\rho_D(P) &=\ \mathrm{res\text{-}ord}^{(p^e)}_P\left(x^r \cdot g_x\right)/p^e - r/p^e,\\ 
w\text{-}\rho_D(P) &=\ \mathrm{res\text{-}ord}^{(p^e)}_P\left(x^r \cdot g_x\right)/p^e - \mathrm{ord}_P(M_{\mathrm{tight}}). 
\end{cases}$$ 
The invariants $\rho_D(P)$ and $w\text{-}\rho_D(P)$ are independent of the choice of $h$ and $X$ by Proposition 1 (2) and (3). 
\end{art_def} 
 
Now we give the definition of our new invariant, which plays the central role in our new strategy for resolution of singularities in the monomial case. 
\begin{art_def}\label{invs-def} 
\item[(1)] 
Firstly we define the invariant $\mathrm{inv}_{\mathrm{MON,\heartsuit}}$ 
by the following formula 
\begin{align*} 
\lefteqn{ 
\mathrm{inv}_{\mathrm{MON,\heartsuit}}(P) 
= \mathrm{H}(P) - \mathrm{ord}_P(M_{\mathrm{tight}}) = \mathrm{Slope}_{h,X}(P) - \mathrm{ord}_P(M_{\mathrm{tight}}) 
}\\ 
&\quad= \min\left\{\mathrm{ord}_P(a_{p^e})/p^e,\ \mu(P)\right\} - \mathrm{ord}_P(M_{\mathrm{tight}}) \\ 
&\quad= \min\left\{\mathrm{ord}_P(a_{p^e})/p^e - \mathrm{ord}_P(M_{\mathrm{tight}}),\ \mathrm{ord}_P(M_{\mathrm{usual}}) - \mathrm{ord}_P(M_{\mathrm{tight}})\right\}. 
\end{align*} 
\item[(2)] 
Secondly we define the invariant 
$\mathrm{inv}_{\mathrm{MON,\spadesuit}}$ 
by the following formula, which is more involved than the naive version above denoted by $\mathrm{inv}_{\mathrm{MON,\heartsuit}}$ 
$$ 
\mathrm{inv}_{\mathrm{MON,\spadesuit}}(P) = \mathrm{lex}\left\{A,B\right\}, 
$$ 
which needs the following explanations: 
 
\item[(i)] 
What we denote by $A$ is the ``word'' 
$$A = \mathrm{lex}\left(w\text{-}\rho_{x_D}(P)\mid D \text{ ranges over all 
\textit{bad} divisors }D \in E_{\mathrm{young}}\right)$$ 
consisting of the letters (numbers) $w\text{-}\rho_{x_D}(P)$'s with $D$ varying over all bad divisors  in $E_{\mathrm{young}}$ with the letters lined up from the smallest to the largest going from left to right.  When one of the letters $w\text{-}\rho_{x_D}(P)$'s is equal to $0$ (actually when this happens, all of the letters $w\text{-}\rho_{x_D}(P)$'s necessarily become $0$), we set $A = 0$.  When there is no bad divisor passing through the point $P$, we also set $A = 0$.  We give the lexicographical order to the set of words, where ``\ $0$'' is the smallest word in the lexicographical order. 
\item[(ii)] 
We set 
$$B = \mathrm{ord}_P(M_{\mathrm{usual}}) - \mathrm{ord}_P(M_{\mathrm{tight}}).$$ 
\item[(iii)] 
Now we have the two words $A$ and $B$.  By writing 
$\mathrm{lex}\left\{A,B\right\}$, 
we mean the pair of these two words lined up from the smallest to the largest going from left to right.  We give the lexicographical order to the set of pairs. 
 
When either one of the words is equal to $0$, i.e. either $A = 0$ or $B = 0$, we set $\mathrm{inv}_{\mathrm{MON,\spadesuit}}(P) = \mathrm{lex}\left\{A,B\right\} = 0$, where the ``\ $0$'' on the right hand side of the above equation is the smallest ``pair'' in the lexicographical order (and where the pair ``\ $0$'' is identified with the word ``\ $0$'' from time to time by abuse of notation). 
 
We also extend the lexicographical order to the union of the set of the pairs $\mathrm{lex}\{A,B\}$ and the set of words $C$: 
$$\mathrm{lex}\{A,B\} 
\begin{cases} 
> C &\text{ if } \min\{A,B\} > C \ \text{or}\ \min\{A,B\} = C \neq 0 \\ 
= C = 0 &\text{ if } \min\{A,B\} = C = 0 \\ 
< C &\text{ if } \min\{A,B\} < C. 
\end{cases} 
$$ 
Note that a pair of words and a word cannot be equal unless both are $0$. 
\end{art_def} 
 
We give the definition of the ``tight'' monomial case, following Villamayor.  (Note that we are already in the monomial case with $\tau = 1$.  So the tight monomial case is a special case of the monomial case with $\tau = 1$.) 
\begin{art_def}[{Tight Monomial Case}] We say we are in the tight monomial case if $\mathrm{inv}_{\mathrm{MON,\heartsuit}}(P) = 0$ (while it is presumed that we are in the monomial case with $\tau = 1$).  It is straightforward to see that this is equivalent to the condition $\mathrm{inv}_{\mathrm{MON,\spadesuit}}(P) = 0$.  That is to say, 
$$\text{Tight Monomial Case} \overset{\text{def}}\Longleftrightarrow \mathrm{inv}_{\mathrm{MON,\heartsuit}}(P) = 0 \Longleftrightarrow \mathrm{inv}_{\mathrm{MON,\spadesuit}}(P) = 0.$$ 
\end{art_def} 
\begin{art_rem} 
\item[(1)] 
If we mix all the letters (numbers) appearing in the word ``A'' and the word ``B'' and if we simply defined the new invariant as the one with these letters lined up from the smallest to the largest going from left to right, then the invariant would not behave very well under blow ups.  (For example, it may strictly increase going from configuration \textcircled{\footnotesize 5} to configuration \textcircled{\footnotesize 3} (cf.~Proposition \ref{invbehavior-prop}).) 
\item[(2)] 
We have the inequalities, for any bad divisor $D \in E_{\mathrm{young}}$, 
$$\mathrm{res\text{-}ord}^{(p^e)}_P(x_D^r \cdot g_{x_D})/p^e \geq \mathrm{ord}_P(a_{p^e})/p^e \geq \mathrm{ord}_P(M_{\mathrm{tight}})$$ 
and hence $w\text{-}\rho_D(P) \geq \mathrm{ord}_P(a_{p^e})/p^e - \mathrm{ord}_P(M_{\mathrm{tight}}) \geq 0$.  Therefore, we also have the inequalities 
$$\mathrm{inv}_{\mathrm{MON,\spadesuit}}(P) \geq \mathrm{inv}_{\mathrm{MON,\heartsuit}}(P) \geq 0,$$ 
where the inequalities are considered according to the lexicographical order described in Definition \ref{invs-def}.  Note that if the second inequality becomes an equality, the first inequality also automatically becomes an equality. 
\end{art_rem} 
\begin{art_prop} 
\item[{\rm (1)}] 
\emph{(Characterization of the tight monomial case)}\quad 
We are in the tight monomial case if and only if one of the following holds: 
 
\smallskip 
 
\noindent \underline{\rm Type I.} $\mathrm{ord}_P(a_{p^e})/p^e - \mathrm{ord}_P(M_{\mathrm{tight}}) = 0$. 
 
\smallskip 
 
We are in the tight monomial case of Type I if and only if we have 
$a_{p^e} = u \cdot (M_{\mathrm{tight}})^{p^e}$, where 
\begin{itemize} 
\setlength{\parskip}{0pt} 
\setlength{\itemsep}{0pt} 
\item 
$a_{p^e}$ is the last coefficient of the unique element $h$ (of the L.G.S., in the Weierstrass form), which is well-adapted at $P$ and at the generic points of all the components of $E_{\mathrm{young}}$ passing through $P$ simultaneously with respect to $X$, 
\item 
$H(\xi_D) \cdot p^e = \mathrm{ord}_{\xi_D}(a_{p^e})$ is an integer for any component $D$ of $E_{\mathrm{young}}$ passing through $P$, and hence $(M_{\mathrm{tight}})^{p^e} = \prod_{D \in E_{\mathrm{young}}}x_D^{H(\xi_D) \cdot p^e}$ has all the integer powers and can be regarded as an element in $\widehat{{\mathcal O}_{W,P}}$, and 
\item 
$u$ is a unit in $\widehat{{\mathcal O}_{W,P}}$. 
\end{itemize} 
\underline{\rm Type II.} $\mathrm{ord}_P(M_{\mathrm{usual}}) - \mathrm{ord}_P(M_{\mathrm{tight}}) = 0$. 
 
\smallskip 
 
We are in the tight monomial case of Type II if and only if we have 
$(M_{\mathrm{tight}})^a = (M_{\mathrm{usual}})^a = M$, where 
\begin{itemize} 
\setlength{\parskip}{0pt} 
\setlength{\itemsep}{0pt} 
\item 
$M$ is the monomial with $(M,a) \in \widehat{{\mathcal R}_P}$ appearing in condition 2 of the setting for the monomial case, 
\item 
$H(\xi_D) \cdot a = \mu(\xi_D) \cdot a = \mathrm{ord}_{\xi_D}(M)$ is an integer for any component $D$ of $E_{\mathrm{young}}$ passing through $P$, and hence $(M_{\mathrm{tight}})^a = (M_{\mathrm{usual}})^a = \prod_{D \in E_{\mathrm{young}}}x_D^{\mu(\xi_D) \cdot a}$ has all the integer powers and can be identified with the element $M \in \widehat{{\mathcal O}_{W,P}}$. 
\end{itemize} 
\item[\rm (2)] 
\emph{(Resolution procedure in the tight monomial case)}\quad 
We can achieve resolution of singularities for $(W,{\mathcal R},E)$ in the tight monomial case by using the invariant $\Gamma_{\mathrm{tight}}$.  More precisely, we have the following description of the resolution procedure: We choose the center $C$ of blow up for the transformation $(W,{\mathcal R},E) \overset{\pi}\longleftarrow (\widetilde{W},\widetilde{{\mathcal R}},\widetilde{E})$ to be 
$$ 
C = \{x_1 = 0\} \cap \mathrm{MaxLocus}(\Gamma_{\mathrm{tight}}) 
= \mathrm{Sing}({\mathcal R}) \cap \mathrm{MaxLocus}(\Gamma_{\mathrm{tight}}) \subset \mathrm{Sing}({\mathcal R}). 
$$ 
Take a closed point $\widetilde{P} \in \pi^{-1}(P) \cap \mathrm{Sing}(\widetilde{\mathcal R}) \subset \widetilde{W}$ (where we assume that the invariant $\sigma$ stays the same).  Then depending on whether $P \in \mathrm{Sing}({\mathcal R}) \subset W$ is in the tight monomial case of Type I or Type II, the point $\widetilde{P} \in \pi^{-1}(P) \cap \mathrm{Sing}(\widetilde{\mathcal R}) \subset \widetilde{W}$ is again in the tight monomial case of Type I or Type II, respectively.  Moreover, the tight monomial $\widetilde{M}_{\mathrm{tight}}$ of the transformation coincides with the transformation $(M_{\mathrm{tight}})^{\widetilde{}}$ of the tight monomial, i.e., $\widetilde{M}_{\mathrm{tight}} = (M_{\mathrm{tight}})^{\widetilde{}}$ (up to the multiplication by a unit).  Therefore, we have the strict decrease of the invariant $\Gamma_{\mathrm{tight}}$ 
$$\widetilde{\Gamma}_{\mathrm{tight}} = \Gamma(\widetilde{M}_{\mathrm{tight}}) = \Gamma((M_{\mathrm{tight}})^{\widetilde{}}) < \Gamma(M_{\mathrm{tight}}) = \Gamma_{\mathrm{tight}}.$$ 
Since the invariant $\Gamma_{\mathrm{tight}}$ cannot decrease infinitely many times, this procedure must come to an end after finitely many repetitions, achieving resolution of singularities for $(W,{\mathcal R},E)$ (or the invariant $\sigma$ strictly decreases). 
\end{art_prop} 
\begin{proof} 
\item[(1)] 
%(\textbf{Characterization of the tight monomial case}) 
The characterization of the tight monomial case follows immediately from the definition, and is left to the reader as an exercise. 
\item[(2)] 
%(\textbf{Resolution procedure in the tight monomial case}) 
Recall that the tight monomial is given by the formula 
$M_{\mathrm{tight}} \!= 
\prod_{D \in E_{\mathrm{young}}}\!\!x_D^{\mathrm{H}(\xi_D)}$ 
and that the invariant $\Gamma_{\mathrm{tight}}$ is computed, for a point $Q$ in a neighborhood of the point of our concern $P \in \mathrm{Sing}({\mathcal R}) \subset W$, by the formula 
$$\Gamma_{\mathrm{tight}}(Q) = 
\begin{cases} 
(\Gamma_{\mathrm{tight},1}(Q), \Gamma_{\mathrm{tight},2}(Q), \Gamma_{\mathrm{tight},3}(Q)) 
&\text{ if }\sum_{Q \in D \in E_{\mathrm{young}}}H(\xi_D) \geq 1, \\ 
(\Gamma_{\mathrm{tight},1}(Q)) =(-\infty) 
&\text{ if }\sum_{Q \in D \in E_{\mathrm{young}}}H(\xi_D) < 1, 
\end{cases} 
$$ 
where the components of $\Gamma_{\mathrm{tight}}(Q)$ 
in the former case are defined by 
\begin{align*} 
\Gamma_{\mathrm{tight},1}(Q) &= 
\max\left\{-n \mid \lambda_1, \ldots, \lambda_n \text{ s.t. } 
\sum\nolimits_{i=1}^nH(\xi_{D_{\lambda_i}})\geq 1,\ 
Q \in \bigcap\nolimits_{i=1}^n D_{\lambda_i} 
\right\}, \\ 
\Gamma_{\mathrm{tight},2}(Q) &= \max\left\{ 
\sum\nolimits_{i=1}^nH(\xi_{D_{\lambda_i}}) 
\mid 
Q \in \bigcap\nolimits_{i=1}^nD_{\lambda_i},\ 
- n = \Gamma_{\mathrm{tight},1}(Q) 
\right\}, \\ 
\Gamma_{\mathrm{tight},3}(Q) &= \max\left\{(\lambda_1, \cdots, \lambda_n) \mid 
\!\!\begin{array}{l} 
\lambda_1 < \cdots < \lambda_n,\ 
- n = \Gamma_{\mathrm{tight},1}(Q), 
\\ 
Q \in \bigcap\nolimits_{i=1}^n D_{\lambda_i},\ 
\sum\nolimits_{i=1}^nH(\xi_{D_{\lambda_i}}) 
= \Gamma_{\mathrm{tight},2}(Q) 
\end{array} 
\right\}. 
\end{align*} 
Note that $\Gamma_{\mathrm{tight},2}(Q)$ and $\Gamma_{\mathrm{tight},3}(Q)$ are computed only when $\Gamma_{\mathrm{tight},1}(Q) \neq - \infty$. 
Note also that 
the components $D_{\lambda}$ of the boundary divisor $(E_{\mathrm{young}} \subset) E$ are indexed by the subscripts $\lambda$ in the totally ordered set $\Lambda$.  We remark that, in this discussion of the local version of the resolution problem 
(cf.~\ref{1.3}), 
even when the point $Q$ varies in a neighborhood of the point $P$, we are computing the invariant $\Gamma_{\mathrm{tight}}(Q)$ with respect to the fixed monomial $M_{\mathrm{tight}}$ determined solely by the invariant $\mathrm{H}$ computed at the point $P$ of our concern.  The discussion of the global version of the problem will be given elsewhere. 
 
\medskip 
 
We start the proof of the resolution procedure. 
 
\medskip 
 
$\circ$ Firstly we prove the equality $C_1=C_2$ for the center (in an analytic 
neighborhood of $P$), where 
$$C_1 = 
\{x_1 = 0\} \cap \mathrm{MaxLocus}(\Gamma_{\mathrm{tight}}) 
, \quad C_2= 
\mathrm{Sing}({\mathcal R}) \cap \mathrm{MaxLocus}(\Gamma_{\mathrm{tight}}) 
.$$ 
For a pair $(f,\lambda) \in {\mathcal R}_P$, we introduce the notation $\mathrm{Sing}(f,\lambda) = \{Q \in W \mid \mathrm{ord}_Q(f) \geqq \lambda\}$.  Since $h = x_1^{p^e} + a_1x_1^{p^e - 1} + a_2x_1^{p^e - 2} + \cdots + a_{p^e - 1}x_1 + a_{p^e}$ and since we have $(M_{\mathrm{tight}})^i \mid  (M_{\mathrm{usual}})^i \mid  a_i$ for $i = 1, \ldots, p^e - 1$ 
by Definition 3 and \ref{4.1} 
(Control over the coefficients) and have $(M_{\mathrm{tight}})^{p^e} \mid  a_{p^e}$ by definition, we conclude 
$$ 
C_2 \subset 
\mathrm{Sing}(h,p^e) \cap \mathrm{MaxLocus}(\Gamma_{\mathrm{tight}}) 
= C_1. 
$$ 
On the other hand, by condition 3 of the setting \ref{2.1}, 
for any element $(f,\lambda) \in {\mathcal R}_P$ with 
$f = \sum c_{f,b}h^b$ being the power series expansion of $f$ 
with respect to ${\mathbb H} = \{(h,p^e)\}$ and $X$, we have by the formal coefficient lemma $(M_{\mathrm{usual}})^{\max\{0,\lambda - b \cdot p^e\}} \mid  c_{f,b}$.  This, together with $M_{\mathrm{tight}} | M_{\mathrm{usual}}$, implies 
$$ 
C_1 = \mathrm{Sing}(h,p^e) \cap \mathrm{MaxLocus}(\Gamma_{\mathrm{tight}}) 
\subset \mathrm{Sing}(f,\lambda). 
$$ 
Since $(f,\lambda) \in {\mathcal R}_P$ is arbitrary, we conclude $C_1 \subset C_2$.  Therefore, we have the desired equality $C_1 = C_2$. 
 
\medskip 
 
$\circ$ Secondly we discuss the resolution procedure.  We only present the argument for Type I.  The argument for Type II is easy and left to the reader as an exercise. 
 
\smallskip 
 
Firstly observe that the condition of being in the tight monomial case of \text{Type I} with $h$ being well-adapted at $P$ and at the generic points of all the components of $E_{\mathrm{young}}$ passing through $P$ is equivalent to the following set $(\bigstar)$ of conditions (i), (ii), (iii): 
$$(\bigstar)\ 
\begin{cases} 
\ \mathrm{(i)} & \mathrm{ord}_{\xi_D}(a_{p^e}) \leq \mu(\xi_D) \cdot p^e, 
\quad\forall D \in E_{\mathrm{young}}, \\ 
\,\mathrm{(ii)} & a_{p^e} = u \cdot \left(\prod_{D \in E_{\mathrm{young}}}x_D^{\mathrm{ord}_{\xi_D}(a_{p^e})}\right) \text{ where }u \in \widehat{{\mathcal O}_{W,P}} \text{ is a unit,} \\ 
\mathrm{(iii)} & \exists D \in E_{\mathrm{young}} \text{ s.t. }\mathrm{ord}_{\xi_D}(a_{p^e}) \not\equiv 0 \text{ mod }p^e. 
\end{cases} 
$$ 
When $(\bigstar)$ is satisfied, we have $\mathrm{ord}_{\xi_D}(a_{p^e}) = \mathrm{H}(\xi_D) \cdot p^e \hskip.03in \forall D \in E_{\mathrm{young}}$ and hence $M_T := (M_{\mathrm{tight}})^{p^e} = \prod_{D \in E_{\mathrm{young}}}x_D^{\mathrm{ord}_{\xi_D}(a_{p^e})}$. 
 
\medskip 
 
Secondly suppose that $\mathrm{MaxLocus}(\Gamma_{\mathrm{tight}}) = 
 V(x_{D}\mid D\in\Omega_0)$ 
where 
$$\Omega_0=\{D_{\lambda_i}\mid i=1,\ldots,n\}\subset 
\Omega:=\{D\mid D\in E_{\mathrm{young}}\}.$$ 
Note that 
$$C = \{x_1 = 0\} \cap \mathrm{MaxLocus}(\Gamma_{\mathrm{tight}}) = V( 
\{x_1\}\cup\{ x_{D}\mid D\in\Omega_0\}).$$ 
 
\smallskip 
 
\noindent \fbox{Case: $n = 1$, i.e., $\mathrm{codim}_WC = 2$.} 
 
Since the singular locus $\mathrm{Sing}(\widetilde{\mathcal R})$ is empty over the $x_1$-chart of the blow up, we may assume that $\widetilde{P} \in \mathrm{Sing}(\widetilde{\mathcal R})$ is in the 
$x_{G}$-chart for the unique element $G\in\Omega_0$ with the regular system of parameters $\widetilde{X} = (\widetilde{x_1}, \widetilde{x_2}, \ldots, \widetilde{x_d})$, where 
$$ 
\widetilde{x_i} = 
\begin{cases} 
x_1/x_{G} 
& 
\text{ if }x_i = x_1, 
\\ 
x_{G} 
& 
\text{ if }x_i = x_{G}, 
\\ 
x_i 
& 
\text{ if }x_i \neq x_1, x_{G}. 
\end{cases} 
$$ 
Note that the components of $\widetilde{E}_{\mathrm{young}}$ passing through $\widetilde{P}$ are 
\begin{itemize} 
\setlength{\parskip}{0pt} 
\setlength{\itemsep}{0pt} 
\item 
the new exceptional divisor $\widetilde{F}$ defined by $\{x_{G} = x_{\widetilde{F}} = 0\}$, and 
\item 
the strict transforms $\widetilde{D}$ of $D \in\Omega\setminus\{G\}$. 
\end{itemize} 
We compute the transformation $(\widetilde{h},p^e) = \left(\pi^*(h)/x_{G}^{p^e}, p^e\right)$ of $(h,p^e)$ 
$$\widetilde{h} = \pi^*(h)/x_{G}^{p^e} = \widetilde{x_1}^{p^e} + \widetilde{a_1}\widetilde{x_1}^{p^e - 1} + \cdots + \widetilde{a_{p^e - 1}}\widetilde{x_1} + \widetilde{a_{p^e}},$$ 
where $\widetilde{a_i} = \pi^*(a_i)/x_{G}^i$ for $i = 1, \ldots, p^e - 1$ and where $\widetilde{a_{p^e}} = \pi^*(a_{p^e})/x_{G}^{p^e} = \pi^*(u) \cdot \pi^*(M_T)/x_{G}^{p^e}=  \widetilde{u} \cdot \widetilde{M_T}$ with $\widetilde{u} = \pi^*(u) \in \widehat{{\mathcal O}_{\widetilde{W},\widetilde{P}}}$ being a unit, and 
$$\widetilde{M_T} = \pi^*(M_T)/x_{G}^{p^e} = x_{\widetilde{F}}^{\mathrm{H}(\xi_{G}) \cdot p^e - p^e} \cdot 
\left(\prod\nolimits_{D \in\Omega\setminus\{G\}} 
x_{\widetilde{D}}^{\mathrm{H}(\xi_D) \cdot p^e}\right).$$ 
Observe that, if $\mathrm{ord}_{\widetilde{P}}(\widetilde{a_{p^e}}) = \mathrm{ord}_{\widetilde{P}}(\widetilde{M_T}) \leq p^e$, then $\widetilde{P} \not\in \mathrm{Sing}(\widetilde{{\mathcal R}})$ or the invariant $\sigma$ strictly decreases, i.e., $\sigma(P) > \sigma(\widetilde{P})$.  Therefore, we may assume $\mathrm{ord}_{\widetilde{P}}(\widetilde{a_{p^e}}) = \mathrm{ord}_{\widetilde{P}}(\widetilde{M_T}) > p^e$, which is equivalent to  the condition $\mathrm{ord}_{\widetilde{P}}(\pi^*\left(M_{\mathrm{tight}}/x_{G}\right)) > 1$.  Since we have $M_{\mathrm{tight}}^i \mid  M_{\mathrm{usual}}^i = M^{i/a} \mid  a_i$ for $i = 1, \ldots, p^e - 1$ 
by Definition 3 and \ref{4.1} (Control over the coefficients $a_i$ for $i = 1, \ldots, p^e - 1$), we conclude 
\begin{align*} 
\mathrm{ord}_{\widetilde{P}}(\widetilde{a_i}) &= \mathrm{ord}_{\widetilde{P}}\left(\pi^*(a_i)/x_{G}^i\right) \geq \mathrm{ord}_{\widetilde{P}}\left(\pi^*(M_{\mathrm{tight}})^i/x_{G}^i\right) 
\\ 
&= \mathrm{ord}_{\widetilde{P}}\left(\pi^*(M_{\mathrm{tight}}/x_{G})^i\right)> i \quad\text{ for }i = 1, \ldots, p^e - 1. 
\end{align*} 
On the other hand, we compute the usual monomial $\widetilde{M}_{\mathrm{usual}}$ at $\widetilde{P}$ to be 
$$\widetilde{M}_{\mathrm{usual}} = x_{\widetilde{F}}^{m_{G}/a - 1} \cdot 
\left(\prod\nolimits_{D \in\Omega\setminus\{G\}}x_{\widetilde{D}}^{m_D/a}\right).$$ 
Since $(\bigstar)$ is satisfied at $P$, we have $\mathrm{H}(\xi_D) \leq m_D/a 
\quad \forall D \in\Omega$, which implies 
 
(i) $\mathrm{H}(\xi_{G}) -1 \leq m_{G}/a - 1$ and $\mathrm{H}(\xi_D) \leq m_D/a \hskip.1in \forall D \in\Omega\setminus\{G\}$. 
 
We obviously have $\mathrm{(ii)} \hskip.1in \widetilde{a_{p^e}} = \widetilde{u} \cdot \widetilde{M_T} \text{ where }\widetilde{u} \in \widehat{{\mathcal O}_{\widetilde{W},\widetilde{P}}} \text{ is a unit}$. 
 
Moreover, since $\exists D \in\Omega$ s.t. $\mathrm{H}(\xi_D) \cdot p^e \not\equiv 0 \text{ mod }p^e$, we also conclude that either 
\begin{itemize} 
\setlength{\parskip}{0pt} 
\setlength{\itemsep}{0pt} 
\item 
$\mathrm{H}(\xi_{G}) \cdot p^e - p^e \not\equiv 0 \text{ mod }p^e$, or 
\item 
$\exists \widetilde{D}$ the strict transform of 
$D \in\Omega\setminus\{G\}$ s.t. $\mathrm{H}(\xi_D) \cdot p^e \not\equiv 0 \text{ mod }p^e$. 
\end{itemize} 
This implies that condition (iii) is also satisfied at $\widetilde{P}$. 
 
Therefore, we conclude that $(\bigstar)$ is satisfied at the point $\widetilde{P}$, which implies that $\widetilde{P} \in \mathrm{Sing}(\widetilde{\mathcal R}) \subset \widetilde{W}$ is in the tight monomial case of Type I and that $\widetilde{h}$ is well-adapted at $\widetilde{P}$ and at the generic points of all the components of $\widetilde{E}_{\mathrm{young}}$.  We also conclude that $\widetilde{M}_{\mathrm{tight}} = (\widetilde{M_T})^{1/p^e} = \left(\widetilde{(M_{\mathrm{tight}})^{p^e}}\right)^{1/p^e} = (M_{\mathrm{tight}})^{\widetilde{}}$. 
 
\smallskip 
 
\noindent \fbox{Case: $n > 1$, i.e., $\mathrm{codim}_WC > 2$.} 
 
Since the singular locus $\mathrm{Sing}(\widetilde{\mathcal R})$ is empty over the $x_1$-chart of the blow up, we may assume that $\widetilde{P} \in \mathrm{Sing}(\widetilde{\mathcal R})$ is in the $x_{G}$-chart 
for some $G\in\Omega_0$ 
with the regular system of parameters $\widetilde{X} = (\widetilde{x_1}, \widetilde{x_2}, \ldots, \widetilde{x_d})$, where 
$$ 
\widetilde{x_i} = 
\begin{cases} 
x_1/x_{G}, 
&\text{ if }x_i = x_1, \\ 
x_{G} 
&\text{ if }x_i = x_{G}, \\ 
\left(x_D + c_{D} \cdot x_{G}\right)/x_{G} 
&\text{ if }x_i=x_D,\  D\in\Omega_0\setminus{G}, \\ 
\hfill\text{for some } c_{D} \in k  &\\ 
x_i 
&\text{ if }x_i \not\in\{x_1\}\cup\{x_D\mid D\in\Omega_0\}. 
\end{cases} 
$$ 
Note that the components of $\widetilde{E}_{\mathrm{young}}$ passing through $\widetilde{P}$ are 
\begin{itemize} 
\setlength{\parskip}{0pt} 
\setlength{\itemsep}{0pt} 
\item 
the new exceptional divisor $\widetilde{F}$ defined by $\{x_{G} = x_{\widetilde{F}} = 0\}$, 
\item 
the strict transforms $\widetilde{D}$ of 
$D \in\{G\in\Omega_0\setminus\{G\}\mid c_G=0\}$, and 
\item 
the strict transforms $\widetilde{D}$ of $D \in E_{\mathrm{young}}$ with 
$D \in\Omega\setminus\Omega_0$. 
\end{itemize} 
We compute the transformation $(\widetilde{h},p^e) = \left(\pi^*(h)/x_{G}^{p^e}, p^e\right)$ of $(h,p^e)$ 
$$\widetilde{h} = \pi^*(h)/x_{G}^{p^e} = \widetilde{x_1}^{p^e} + \widetilde{a_1}\widetilde{x_1}^{p^e - 1} + \cdots + \widetilde{a_{p^e - 1}}\widetilde{x_1} + \widetilde{a_{p^e}}$$ 
where $\widetilde{a_i} = \pi^*(a_i)/x_{G}^i$ for $i = 1, \ldots, p^e - 1$ and where $\widetilde{a_{p^e}} = \pi^*(a_{p^e})/x_{G}^{p^e} =  \pi^*(u) \cdot \pi^*(M_T)/x_{G}^{p^e} = \widetilde{u} \cdot \widetilde{M_T}$ with $\widetilde{u} = \pi^*(u) \in \widehat{{\mathcal O}_{\widetilde{W},\widetilde{P}}}$ being a unit, and 
\begin{align*} 
\widetilde{M_T} =& \pi^*(M_T)/x_{G}^{p^e} 
= \pi^*\left(M_{\mathrm{tight}}/x_{G}\right)^{p^e} \\ 
=& 
{x_{\widetilde{F}}}^{\left(\sum_{D \in\Omega_0}\mathrm{H}(\xi_{D})\right) \cdot p^e - p^e} 
\cdot \prod\nolimits_{D \in\Omega_0\setminus\{G\},\ c_{D} = 0} 
\widetilde{x_D}^{\mathrm{H}(\xi_D) \cdot p^e} 
\\ 
&\cdot 
\prod\nolimits_{D \in\Omega\setminus\Omega_0}\widetilde{x_D}^{\mathrm{H}(\xi_D) \cdot p^e} 
\cdot 
\prod\nolimits_{D \in\Omega_0\setminus\{G\},\ c_{D} \neq 0} 
\left[\widetilde{x_D} - c_{D}\right]^{\mathrm{H}(\xi_D) \cdot p^e} 
. 
\end{align*} 
Noting that the last factor above is a unit, we set 
$$ 
\widetilde{M_T}' 
= {x_{\widetilde{F}}}^{\left(\sum_{D \in\Omega_0}\mathrm{H}(\xi_{D})\right) \cdot p^e - p^e} 
\cdot 
\prod\nolimits_{D \in\Omega_0\setminus\{G\}, c_{D} = 0}\widetilde{x_D}^{\mathrm{H}(\xi_D) \cdot p^e} 
\cdot \prod\nolimits_{D \in\Omega\setminus\Omega_0}\widetilde{x_D}^{\mathrm{H}(\xi_D) \cdot p^e}. 
$$ 
Observe that, if $\mathrm{ord}_{\widetilde{P}}(\widetilde{a_{p^e}}) = \mathrm{ord}_{\widetilde{P}}(\widetilde{M_T}) \leq p^e$, then $\widetilde{P} \not\in \mathrm{Sing}(\widetilde{{\mathcal R}})$ or the invariant $\sigma$ 
strictly decreases, i.e., $\sigma(P) > \sigma(\widetilde{P})$. 
Therefore, we may assume $\mathrm{ord}_{\widetilde{P}}(\widetilde{a_{p^e}}) = \mathrm{ord}_{\widetilde{P}}(\widetilde{M_T}) > p^e$, which is equivalent to  the condition $\mathrm{ord}_{\widetilde{P}}\left(\pi^*\left(M_{\mathrm{tight}}/x_{G}\right)\right) > 1$.  Since we have $M_{\mathrm{tight}}^i \mid  M_{\mathrm{usual}}^i = M^{i/a} \mid  a_i$ for $i = 1, \ldots, p^e - 1$ 
by Definition 3 and \ref{4.1} (Control over the coefficients $a_i$ for $i = 1, \ldots, p^e - 1$), we conclude 
\begin{align*} 
\mathrm{ord}_{\widetilde{P}}(\widetilde{a_i}) &= \mathrm{ord}_{\widetilde{P}}\left(\pi^*(a_i)/x_{G}^i\right) \geq \mathrm{ord}_{\widetilde{P}}\left(\pi^*(M_{\mathrm{tight}})^i/x_{G}^i\right) \\ 
&= \mathrm{ord}_{\widetilde{P}}\left(\pi^*(M_{\mathrm{tight}}/x_{G})^i\right)> i 
\qquad 
\text{ for }i = 1, \ldots, p^e - 1. 
\end{align*} 
On the other hand, we compute the transformation of the usual monomial $M_{\mathrm{usual}}$ 
\begin{align*} 
\pi^*\left((M_{\mathrm{usual}})^a\right)/{x_{G}}^a 
=& \pi^*\left(\prod\nolimits_{D \in\Omega}x_D^{m_D}\right)/{x_{G}}^a \\ 
=& {x_{\widetilde{F}}}^{(\sum_{D \in\Omega_0}m_D) - a} 
\cdot \prod\nolimits_{D \in\Omega_0\setminus\{G\}, c_{D} = 0} 
\widetilde{x_D}^{m_D} \\ 
&\ \cdot \prod\nolimits_{D \in \Omega_0\setminus\{G\}, c_{D} \neq 0} 
\left[\widetilde{x_D} - c_{D}\right]^{m_D} 
\cdot \prod\nolimits_{D \in\Omega\setminus\Omega_0}\widetilde{x_D}^{m_D}. 
\end{align*} 
Therefore, we conclude that the usual monomial $\widetilde{M}_{\mathrm{usual}}$ at the point $\widetilde{P}$ is given by the formula 
$$ 
\widetilde{M}_{\mathrm{usual}} = {x_{\widetilde{F}}}^{(\sum_{D\in\Omega_0}m_D/a) - 1} 
\cdot \prod\nolimits_{ 
D \in\Omega_0\setminus\{G\},\ c_{D} = 0} 
{\widetilde{x_D}^{m_D}/a} 
\cdot \prod\nolimits_{D \in\Omega\setminus\Omega_0}\widetilde{x_D}^{m_D/a}. 
$$ 
We want to show that $(\bigstar)$ is satisfied at the point $\widetilde{P}$. 
 
Since $(\bigstar)$ is satisfied at $P$, we have $\mathrm{H}(\xi_D) 
\leq m_D/a,\ \forall D \in\Omega$, which implies 
$$ 
\sum\nolimits_{D \in\Omega_0}\mathrm{H}(\xi_D) -1 \leq 
\sum\nolimits_{D\in\Omega_0}m_D/a - 1, 
\quad\text{and}\quad 
\mathrm{H}(\xi_D) \leq m_D/a \ \text{for}\ D \in\Omega\setminus\{G\}. 
$$ 
This checks condition (i) at $\widetilde{P}$. 
 
\smallskip 
 
We obviously have $\mathrm{(ii)} \hskip.1in \widetilde{a_{p^e}} = \widetilde{u} \cdot \widetilde{M_T} = \widetilde{u}' \cdot \widetilde{M_T}' \text{ where }\widetilde{u}' \in \widehat{{\mathcal O}_{\widetilde{W},\widetilde{P}}} \text{ is a unit.}$ 
 
\smallskip 
 
Condition (iii) follows immediately from the following claim. 
\begin{claim} 
We have either 
$$ 
\exists D \in\Omega\setminus\Omega_0 
\ \text{s.t.}\ 
\mathrm{H}(\xi_D) \cdot p^e \not\equiv 0 \text{ mod }p^e 
\text{ or } 
\left(\sum\nolimits_{D \in\Omega_0}\mathrm{H}(\xi_D) -1\right) \cdot p^e \not\equiv 0 \text{ mod }p^e. 
$$ 
%\begin{itemize} 
%\setlength{\parskip}{0pt} 
%\setlength{\itemsep}{0pt} 
%\item 
%$\exists D \in\Omega\setminus\Omega_0$, s.t. $\mathrm{H}(\xi_D) \cdot p^e \not\equiv 0 \text{ mod }p^e$, or 
%\item 
%$\left(\sum_{D \in\Omega_0}\mathrm{H}(\xi_D) -1\right) \cdot p^e \not\equiv 0 \text{ mod }p^e$. 
%\end{itemize} 
\end{claim} 
\begin{proof}[Proof of the Claim.] 
First we remark that, in this case of $n > 1$, we have $0 \leq \mathrm{H}(\xi_D) < 1 \quad \forall D \in\Omega$ when we look at how the invariant $\Gamma_{\mathrm{tight}}$ dictates the algorithm. 
 
\smallskip 
 
\noindent \underline{Subcase:} 
$\exists D \in\Omega\setminus\Omega_0$ s.t. $0 \neq \mathrm{H}(\xi_D)$. 
\quad 
In this subcase, we have $0 < \mathrm{H}(\xi_D) \cdot p^e < p^e$ and hence $\mathrm{H}(\xi_D) \cdot p^e \not\equiv 0 \text{ mod }p^e$. 
 
\smallskip 
 
\noindent \underline{Subcase:} Otherwise, i.e., 
$\mathrm{H}(\xi_D) = 0,\ \forall D\in\Omega\setminus\Omega_0$. 
\quad 
In this subcase, we claim to have $1 < \sum_{D\in\Omega_0} \mathrm{H}(\xi_{D}) < 2$.  In fact, on one hand, if 
$\sum_{D\in\Omega_0} \mathrm{H}(\xi_{D}) \geq 2$, then, since $0 \leq \mathrm{H}(\xi_{D_{\lambda_n}}) < 1$, we would have 
$\sum_{D\in\Omega_0\setminus\{D_{\lambda_n}\}} \mathrm{H}(\xi_{D}) \geq 1$. 
But this is against the maximality of the value of $\Gamma_{\mathrm{tight,1}}(P)$.  On the other hand, if $\sum_{D\in\Omega_0} \mathrm{H}(\xi_{D}) \leq 1$, then we would have 
\begin{align*} 
\mathrm{ord}_P(a_{p^e}) &= \mathrm{ord}_P\left((M_{\mathrm{tight}})^{p^e}\right) = \mathrm{ord}_P\left(\left(\prod\nolimits_{D \in\Omega}x_D^{\mathrm{H}(\xi_D)}\right)^{p^e}\right) \\ 
&= \mathrm{ord}_P\left(\left(\prod\nolimits_{D\in\Omega_0}x_{D}^{\mathrm{H}(\xi_{D})}\right)^{p^e}\right) 
%(\text{by the subcase assumption}) 
= \left(\sum\nolimits_{D\in\Omega_0} \mathrm{H}(\xi_{D})\right) \cdot p^e \leq p^e, 
\end{align*} 
where the third equation follows from the subcase assumption. 
But this is against the description of the Weierstrass form for $h$ 
in \ref{4.1}. 
 
Now we conclude that $0 < \left(\sum_{D \in\Omega_0}\mathrm{H}(\xi_D) -1\right) \cdot p^e < p^e$ and hence that $\left(\sum_{D \in\Omega_0}\mathrm{H}(\xi_D) -1\right) \cdot p^e \not\equiv 0 \text{ mod }p^e$. 
 
This finishes the proof of the claim. 
\end{proof} 
 
Therefore, we conclude that $(\bigstar)$ is satisfied at the point $\widetilde{P}$, which implies that $\widetilde{P} \in \mathrm{Sing}(\widetilde{\mathcal R}) \subset \widetilde{W}$ is in the tight monomial case of Type I and that $\widetilde{h}$ is well-adapted at $\widetilde{P}$ and at the generic points of all the components of $\widetilde{E}_{\mathrm{young}}$.  We also conclude that $\widetilde{M}_{\mathrm{tight}} = (\widetilde{M_T}')^{1/p^e} = (\widetilde{M_T})^{1/p^e} = \left(\widetilde{(M_{\mathrm{tight}})^{p^e}}\right)^{1/p^e} = (M_{\mathrm{tight}})^{\widetilde{}}$, where the second equality holds up to the multiplication by a unit. 
 
%\noindent \underline{Type II.} Left to the reader as an exercise. 
 
This completes the proof of Proposition 2. 
\end{proof} 
\begin{art_rem} 
\item[(1)] 
In the resolution procedure in the tight monomial case, no matter whether it is of Type I or Type II, the center $C = \mathrm{Sing}({\mathcal R}) \cap \mathrm{MaxLocus}(\Gamma_{\mathrm{tight}})$ is algebraic, being the intersection of the singular locus $\mathrm{Sing}({\mathcal R})$ and some components of the boundary divisor $(E_{\mathrm{young}} \subset) E$.  Therefore, the whole resolution process is algebraic, even though our analysis is carried out at the analytic level. 
\item[(2)] 
The proof of Proposition 2 above is easy, and may look innocuous.  However, it is worthwhile to note that, if we take the blow up with an arbitrary permissible center, i.e., a center which is nonsingular, contained in the singular locus, and transversal to the boundary divisor, we may 
\emph{not} 
stay in the tight monomial case any longer after blow up (even though we will stay in the monomial case).  Therefore, it is essential to choose the center dictated by the invariant $\Gamma_{\mathrm{tight}}$.  In several of the existing approaches for resolution of (hypersurface) singularities in positive characteristic, people may regard an equation of the form $x_1^{p^e} + a_{p^e} = 0$ a ``good'' equation when $a_{p^e} = u \cdot M$ with $u$ being a unit and $M$ being a monomial in terms of $(x_2, \ldots, x_d)$.  However, if we take the blow up with an arbitrary permissible center, the equation may 
\emph{not} 
stay being ``good'' after blow up and after 
\emph{cleaning}. 
This is considered to be one of many pathologies and/or obstacles toward resolution of singularities in positive characteristic by some people. It was Villamayor who first realized that this obstacle can be overcome with the introduction of the notion of the tight monomial case and by the use of the invariant $\Gamma_{\mathrm{tight}}$. 
\end{art_rem} 
\end{subsection} 
\end{section} 
\begin{section}{The new invariant and its behavior under transformation} 
In order to quote the results of our previous paper \cite{KM2} and compare them to the results of this paper directly, we try to use the common notations and symbols: in the description of the setting for the monomial case (with $\tau = t = 1$ and $d = 3$ 
(cf.~\ref{2.1} and \ref{4.1})), 
we set $(x_1, \ldots, x_t, x_{t+1}, \ldots, x_d) = (x_1 = x_t, x_2 = x_{t+1}, x_3 = x_d) = (z, x, y)$ so that the unique element $h$ in the L.G.S. ${\mathbb H} = \{(h,p^e)\}$ is in the Weierstrass form $h = z^{p^e} + a_1z^{p^e-1} + a_2z^{p^e-2} + \cdots + a_{p^e - 1}z + a_{p^e}$ with $a_i \in k[[x,y]]$ and $\mathrm{ord}_P(a_i) > i$ for $i = 1, \ldots, p^e$. 
\begin{subsection}{Description of the algorithm}\label{5.1} 
The description of our algorithm to reach the tight monomial case in dimension 3 depends upon the analysis of the singular locus, which is determined by looking at the invariant $\mathrm{H}$ as below. 
\begin{art_prop}[{Description of the singular locus} (cf.~Proposition 6 in \cite{KM2})] 
We have the following description of the singular locus $\mathrm{Sing}({\mathcal R})$ at $P$, denoted by $\mathrm{Sing}({\mathcal R})_P$, 
according to the values of $h_x = \mathrm{H}(\xi_{H_x})$ 
and $h_y = \mathrm{H}(\xi_{H_y})$: 
$$\mathrm{Sing}({\mathcal R})_P = 
\begin{cases} 
V(z,x) \cup V(z,y)&\text{ if }\quad h_x \geq 1 \text{ and } h_y \geq 1\\ 
V(z,x) &\text{ if }\quad h_x \geq 1 \text{ and } h_y < 1\\ 
V(z, y) &\text{ if }\quad h_x < 1 \text{ and } h_y \geq 1\\ 
V(z,x,y) = P &\text{ if }\quad h_x < 1 \text{ and } h_y < 1, 
\end{cases} 
$$ 
where ``V'' denotes the vanishing locus and where $(z,x,y)$ is a regular system of parameters at $P$ with respect to which $h$ is well-adapted simultaneously at $P$, $\xi_{H_x}$, and $\xi_{H_y}$. 
\end{art_prop} 
 
\noindent \fbox{\textbf{Algorithm}} 
 
\smallskip 
 
\noindent Step 1. \quad We ask the question: $\mathrm{inv}_{\mathrm{MON,\spadesuit}} = 0\ ?$ 
 
If the answer is YES, then we are in the tight monomial case and we are done. 
 
If the answer is NO, then we go to Step 2. 
 
\smallskip 
 
\noindent Step 2. \quad We ask the question: $\dim \mathrm{Sing}({\mathcal R})_P = 1\ ?$ 
 
If the answer is YES, then we take the transformation with center $V(z,x)$ (or $V(z,y)$), and we go back to Step 1.  In case $\mathrm{Sing}({\mathcal R}_P) = V(z,x) \cup V(z,y)$, we choose $V(z,x)$ as the center if $\mathrm{H}(\xi_{H_x}) > \mathrm{H}(\xi_{H_y})$, and choose $V(z,y)$ if $\mathrm{H}(\xi_{H_y}) > \mathrm{H}(\xi_{H_x})$.  If $\mathrm{H}(\xi_{H_x}) = \mathrm{H}(\xi_{H_y})$, then we choose $V(z,x)$ if $H_x = E_{\alpha}$ has the bigger index $\alpha$ than the index $\beta$ for $H_y = E_{\beta}$, and vice versa.  (Note that locally in a neighborhood of $P$ the components of the boundary divisor $(E_{\mathrm{young}} \subset$) $E$ have distinct indices and hence that $\alpha$ cannot be equal to $\beta$.) 
 
If the answer is NO, then we take the transformation with center $P$ and we go back to Step 1. 
 
\smallskip 
 
We repeat this procedure. 
 
\smallskip 
 
The only issue is: \emph{Does this algorithm terminate after finitely many procedures?}  In order to settle this issue, we analyze the behavior of $\mathrm{inv}_{\mathrm{MON,\spadesuit}}$ under transformations. 
\end{subsection} 
\begin{subsection}{Behavior of the new invariant under transformations} 
For the purpose of analyzing the behavior of the new invariant 
$\mathrm{inv}_{\mathrm{MON,\spadesuit}}$, 
we introduce the classification of the \emph{configurations} as in \cite{KM2}. 
 
\smallskip 
 
\noindent \textbf{Configurations} 
 
\smallskip 
 
Looking at the boundary divisor(s) in $E_{\mathrm{young}}$ at the point 
$P \in \mathrm{Sing}({\mathcal R})$ 
and seeing whether they are good or bad, we come up with the following classification of the configurations.  Note that the pictures depict the configurations in a 
2-dimensional manner, 
taking the intersection with the hypersurface $Z = \{z = 0\}$. 
\begin{itemize} 
\setlength{\parskip}{0pt} 
\setlength{\itemsep}{0pt} 
\item[\textcircled{\footnotesize 1}] 
The point $P$ is only on one boundary divisor (in $E_{\mathrm{young}}$), say $H_x$, which is good. 
 \ZU{\VCL{$H_x$}{good} 
 \CPO{$P$}} 
\item[\textcircled{\footnotesize 2}] 
The point $P$ is at the intersection of two boundary divisors (in $E_{\mathrm{young}}$), both of which are good. 
 \ZU{\VCL{$H_x$}{good} 
 \HOL{$H_y$}{good} 
 \CPO{$P$}} 
\item[\textcircled{\footnotesize 3}] 
The point $P$ is only on one boundary divisor (in $E_{\mathrm{young}}$), say $H_x$, which is bad. 
 \ZU{\VCL{$H_x$}{bad} 
 \CPO{$P$}} 
\item[\textcircled{\footnotesize 4}] 
The point $P$ is at the intersection of two boundary divisors (in $E_{\mathrm{young}}$), one of which, say, $H_x$, is bad, while the other, say $H_y$, is good. 
 \ZU{\VCL{$H_x$}{bad} 
 \HOL{$H_y$}{good} 
 \CPO{$P$}} 
\item[\textcircled{\footnotesize 5}] 
The point $P$ is at the intersection of two boundary divisors (in $E_{\mathrm{young}}$), say $H_x$ and $H_y$, both of which are bad. 
 \ZU{\VCL{$H_x$}{bad} 
 \HOL{$H_y$}{bad} 
 \CPO{$P$}} 
\end{itemize} 
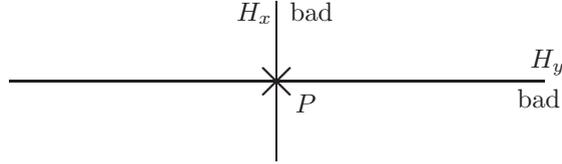
\begin{art_prop}[{Behavior of 
${\mathrm{inv}_{\mathrm{MON,\spadesuit}}}$ under transformations}]% 
\label{invbehavior-prop} 
\item[\rm 1.] 
Suppose that the point $P$ is in configuration \textcircled{\footnotesize 1} or \textcircled{\footnotesize 2}.  Then we have 
$\mathrm{inv}_{\mathrm{MON,\spadesuit}}(P) \allowbreak = 0$, 
and hence we are in the tight monomial case. 
\item[\rm 2.] 
Suppose that the point $P$ is in configuration \textcircled{\footnotesize 3}, \textcircled{\footnotesize 4}, or \textcircled{\footnotesize 5}, and that 
$\mathrm{inv}_{\mathrm{MON,\spadesuit}}(P) \allowbreak \neq 0$, i.e., we are 
\emph{not} in the tight monomial case. 
 
Let $(W,{\mathcal R},E) \overset{\pi}\longleftarrow (\widetilde{W},\widetilde{\mathcal R},\widetilde{E})$ be the transformation with center $C \subset \mathrm{Sing}({\mathcal R})$ as specified in the algorithm.  Take a closed point $\widetilde{P} \in \pi^{-1}(P) \cap \mathrm{Sing}(\widetilde{\mathcal R}) \subset \widetilde{W}$.  (If $\pi^{-1}(P) \cap \mathrm{Sing}(\widetilde{\mathcal R}) = \emptyset$, then the local version of the resolution problem has already 
been solved (cf.~\ref{1.3}). 
Therefore, we assume that $\pi^{-1}(P) \cap \mathrm{Sing}(\widetilde{\mathcal R}) \neq \emptyset$.  If $\sigma(P) > \sigma(\widetilde{P})$, then we go back to the reduction step 
{\rm \textbf{general case $\to$ monomial case}} 
with the decreased value of 
the invariant $\sigma$. 
Therefore, we also assume that the invariant $\sigma$ stays the same, i.e., 
$\sigma(P) = \sigma(\widetilde{P})$, 
and hence so does the invariant $\tau$, i.e., $\tau(P) = \tau(\widetilde{P}) = 1$.)  The behavior of the new invariant 
${\mathrm{inv}_{\mathrm{MON,\spadesuit}}}$ is summarized as follows. 
 
\medskip 
 
\noindent \underline{{\rm Case (a):} 
$\pi$ is the blow up with a 1-dimensional center, say, $C = V(z,x)$. 
} 
 
\smallskip 
 
In this case, we have ${\mathrm{inv}_{\mathrm{MON,\spadesuit}}}(P) = {\mathrm{inv}_{\mathrm{MON,\spadesuit}}}(\widetilde{P})$ and $\mu(\xi_{H_x}) > \mu(\xi_{H_x}) - 1 = \mu(\xi_{H_{\widetilde{x}}})$. 
 
\smallskip 
 
\noindent \underline{{\rm Case (b):} 
$\pi$ is the blow up with a point center $C = P = V(z,x,y)$. 
} 
 
\smallskip 
 
In this case, we have ${\mathrm{inv}_{\mathrm{MON,\spadesuit}}}(P) > {\mathrm{inv}_{\mathrm{MON,\spadesuit}}}(\widetilde{P})$ 
\emph{except} 
for the case where the transformation $P \in \mathrm{Sing}({\mathcal R}) \subset W \overset{\pi}\longleftarrow \widetilde{P} \in \pi^{-1}(P) \cap \mathrm{Sing}(\widetilde{\mathcal R}) \subset \widetilde{W}$ is ``esoteric''.  The definition of the transformation being esoteric (or standard) is given below. 
 
It turns out (cf.~Lemma 1) that an esoteric transformation occurs \emph{only} 
when the point $P$ is in configuration \textcircled{\footnotesize 4} or \textcircled{\footnotesize 5} and when the point $\widetilde{P}$ is in 
configuration \textcircled{\footnotesize 3}. 
\end{art_prop} 
\begin{art_def} Let $P \in \mathrm{Sing}({\mathcal R}) \subset W \overset{\pi}\longleftarrow \widetilde{P} \in \pi^{-1}(P) \cap \mathrm{Sing}(\widetilde{\mathcal R}) \subset \widetilde{W}$ be a transformation where the point $P$ is in configuration \textcircled{\footnotesize 3}, \textcircled{\footnotesize 4}, or \textcircled{\footnotesize 5} and where the point $\widetilde{P}$ is also in configuration \textcircled{\footnotesize 3}, \textcircled{\footnotesize 4}, or \textcircled{\footnotesize 5}.  We say that the transformation $\pi$ is 
$$ 
\text{esoteric} \Longleftrightarrow\  A < \widetilde{A}, 
\qquad 
\text{standard} \Longleftrightarrow\  A \geq \widetilde{A}, 
$$ 
where the ``words'' $A$ and $\widetilde{A}$ are defined as in Definition 6 
$$ 
\begin{cases} 
A &=\ \mathrm{lex}\left(w\text{-}\rho_{x_D}(P)\mid D \text{ ranges over all 
\emph{bad} 
divisors }D \in E_{\mathrm{young}}\right), \\ 
\widetilde{A} &=\ \mathrm{lex}\left(w\text{-}\rho_{x_{\widetilde{D}}}(\widetilde{P})\mid \widetilde{D} \text{ ranges over all 
\emph{bad} 
divisors }\widetilde{D} \subset \widetilde{E}_{\text{young}}\right), 
\end{cases} 
$$ 
and where the inequalities are given with respect to the lexicographical order. 
\end{art_def} 
\begin{proof}[Proof of Proposition \ref{invbehavior-prop}] 
\item[1.] 
The assertions in 1 are obvious from the definitions, and hence their proof is left to the reader as an exercise. 
\item[2.] 
We present a proof for the assertions in 2 below. 
 
\smallskip 
 
\noindent \underline{{\rm Case (a):}}\quad 
We note that, in this case, there is possibly only one point $\widetilde{P} \in \pi^{-1}(P) \cap \mathrm{Sing}(\widetilde{R}) \subset \widetilde{W}$, lying in the $x$-chart, with the regular system of parameters $\widetilde{X} = (\widetilde{z},\widetilde{x},\widetilde{y}) = (z/x,x,y)$. It is then straightforward to see that $\widetilde{h} = \pi^*(h)/x^{p^e}$ is well-adapted with respect to $\widetilde{X}$, and the rest of the assertions follow easily. 
 
\smallskip 
 
\noindent \underline{{\rm Case (b):}}\quad 
For the verification of the assertions in this case, we consider the following lemma, which is similar to Claim 2 in \cite{KM2}. 
 
\begin{art_lem} Let $(W,{\mathcal R},E) \overset{\pi}\longleftarrow (\widetilde{W},\widetilde{\mathcal R},\widetilde{E})$ be the transformation with a point center $C = P = V(z,x,y) \in \mathrm{Sing}({\mathcal R})$, where $(W,{\mathcal R},E)$ with $\dim W = 3$ is in the monomial case with $\tau = 1$ at $P$ 
as described in the setting \ref{2.1} and \ref{4.1} 
(cf.~the remark about the notations and symbols at the beginning of Section 5), and where the unique element $h$ in the L.G.S. ${\mathbb H} = \{(h,p^e)\}$ is well-adapted simultaneously at $P$ and at the the generic points of all the components of $E_{\mathrm{young}}$ passing through $P$ with respect to the regular system of parameters $(z,x,y)$ (cf.~Proposition 1 (1)).  Set $Z = \{z = 0\}$.  Then a point $\widetilde{P} \in \pi^{-1}(P) \cap \mathrm{Sing}(\widetilde{\mathcal R}) \subset \widetilde{W}$ must be on the strict transform $Z'$ of $Z$, lying either in the $x$-chart or in the $y$-chart.  Assume that the invariant $\sigma$ stays the same, i.e., $\sigma(P) = \sigma(\widetilde{P})$ (and hence that we stay in the monomial case with $\tau = 1$ at $\widetilde{P}$).  Assume further that $\mathrm{inv}_{\mathrm{MON,\spadesuit}}(P) \neq 0$, i.e., we are 
\emph{not} 
in the tight monomial case at $P$.  We make the following observations regarding the behavior of the invariants under blow up.  We denote the strict transforms of $H_x = \{x = 0\}$ and $H_y = \{y = 0\}$ by $H_x'$ and $H_y'$, and the exceptional divisor by $E_P$.  Note that the pictures depict the configuration in a 2-dimensional manner, taking the intersection with the hypersurface $Z$ before blow up and with its strict transform $Z'$ after blow up. 
 
\smallskip 
 
\item[\rm (1)] 
The point $P$ is in configuration \textcircled{\footnotesize 3}, \textcircled{\footnotesize 4}, or \textcircled{\footnotesize 5}. 
 
{\rm (1.1)} Look at the point $\widetilde{P} = E_P \cap H'_x \cap Z'$ in the $y$-chart with the regular system of parameters $(\widetilde{z},\widetilde{x},\widetilde{y}) = (z/y,x/y,y)$.  Then the hypersurface $H'_x = H_{\widetilde{x}}$ is bad, and we have $w\text{-}\rho_x(P) > w\text{-}\rho_{\widetilde{x}}(\widetilde{P})$. 
 
{\rm (1.2)} Suppose $P$ is bad, and hence $E_P$ is 
also bad (cf.~Remark 4).  Look at the point $\widetilde{P} \in E_P \cap H'_y \cap Z'$ in the $x$-chart with the regular system of parameters 
$(\widetilde{z},\widetilde{x},\widetilde{y}) 
= (z/x,x,y/x)$.  (Note that, if the point $P$ is in configuration \textcircled{\footnotesize 3}, then the point $\widetilde{P}$ is 
\emph{any} point on $(E_P \setminus H_x') \cap Z'$.)  Then we have $w\text{-}\rho_x(P) \geq w\text{-}\rho_{\widetilde{x}}(\widetilde{P})$. 
 
\ZU{ 
 \VCL{$H_x$}{bad} 
 \HDL{$H_y$}{} 
 \CPO{$P$} 
\hskip94pt $\uparrow$} 
\ZU{ 
 \VLL{$H_x'$}{bad} 
 \RDL{$H_y'$}{} 
 \HOL{$E_P$}{$\mathit{bad}_{(1.2)}$} 
 \LPO{$\widetilde{P}_{(1.1)}$} 
 %\CPO{$\widetilde{P}_{(1.2)}$ \text{ or }} 
\RPO{$\widetilde{P}_{(1.2)}$} 
} 
 
{\rm (1.3)} Suppose $P$ is bad, and hence $E_P$ is 
also bad (cf.~Remark 4).  Look at the point $\widetilde{P} \in (E_P \setminus H_x') \cap Z'$ with a regular system of parameters $(\widetilde{z},\widetilde{x},\widetilde{y}) = (z/x,x,y/x - c)$ for some $c \in k$.  We further assume, in case the point $P$ is in configuration \textcircled{\footnotesize 5}, that $c \neq 0$, i.e., $\widetilde{P} \not\in H_y'$. 
Then we have $\mathrm{ord}_P(M_{\mathrm{usual}}) - \mathrm{ord}_P(M_{\mathrm{tight}}) > \mathrm{ord}_{\widetilde{P}}(\widetilde{M}_{\mathrm{usual}}) - \mathrm{ord}_{\widetilde{P}}(\widetilde{M}_{\mathrm{tight}})$. 
 
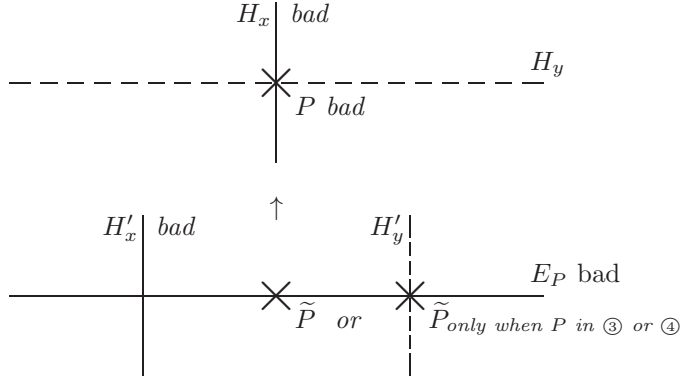
\ZU{ 
 \VCL{$H_x$}{bad} 
 \HDL{$H_y$}{} 
 \CPO{$P$ \text{bad}} 
\hskip94pt $\uparrow$} 
\ZU{ 
 \VLL{$H_x'$}{bad} 
 \RDL{$H_y'$}{} 
 \HOL{$E_P$$\mathrm{\ bad}$}{} 
 %\LPO{$\widetilde{P}_{(1.1)}$} 
 \CPO{$\widetilde{P}$ \text{ or }} 
\RPO{$\widetilde{P}_{\text{only when }P\text{ in 
\textcircled{\tiny 3} or \textcircled{\tiny 4}}}$} 
} 
 
\item[\rm (2)] 
The point $P$ is in configuration \textcircled{\footnotesize 3}, \textcircled{\footnotesize 4}, or \textcircled{\footnotesize 5}.  Suppose $P$ is bad, and hence $E_P$ is also bad (cf.~Remark 4).  Look at the point $\widetilde{P} \in E_P \cap H'_x \cap Z'$ with the regular system of parameters $(\widetilde{z},\widetilde{x},\widetilde{y}) = (z/y,x/y,y)$.  Then we have $\mathrm{ord}_P(M_{\mathrm{usual}}) - \mathrm{ord}_P(M_{\mathrm{tight}}) > w\text{-}\rho_{\widetilde{y}}(\widetilde{P})$. 
 
\ZU{ 
 \VCL{$H_x$}{bad} 
 \HDL{$H_y$}{} 
 \CPO{$P$ \text{bad}} 
\hskip94pt $\uparrow$} 
\ZU{ 
 \VLL{$H_x'$}{bad} 
 \RDL{$H_y'$}{} 
 \HOL{$E_P$}{bad} 
 \LPO{$\widetilde{P}$} 
 %\CPO{$\widetilde{P}$} 
} 
\item[\rm (3)] 
The point $P$ is in configuration \textcircled{\footnotesize 5}.  Suppose $P$ is good.  Then we have 
$w\text{-}\rho_x(P) > w\text{-}\mu(\xi_{H_y})$ and $w\text{-}\rho_y(P) > w\text{-}\mu(\xi_{H_x})$.  Note that the invariant $w\text{-}\mu(\xi_{H_x})$ is defined by the formula $w\text{-}\mu(\xi_{H_x}) = \mu(\xi_{H_x}) - \mathrm{H}(\xi_{H_x})$.  Look at the point $\widetilde{P} = E_P \cap H'_x \cap Z'$ in the $y$-chart 
with the regular system of parameters 
$(\widetilde{z},\widetilde{x},\widetilde{y}) = (z/y,x/y,y)$.  Since 
$w\text{-}\mu(\xi_{H_{\widetilde{x}}}) = w\text{-}\mu(\xi_{H_x})$, we have as a consequence $w\text{-}\rho_y(P) > w\text{-}\mu(\xi_{H_{\widetilde{x}}}) = \mathrm{ord}_{\widetilde{P}}(\widetilde{M}_{\mathrm{usual}}) - \mathrm{ord}_{\widetilde{P}}(\widetilde{M}_{\mathrm{tight}})$.  We draw a similar conclusion looking 
at the point $E_P \cap H'_y \cap Z'$ in the $x$-chart. 
 
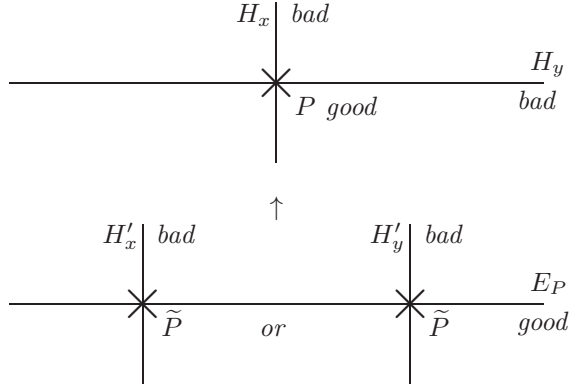
\ZU{ 
 \VCL{$H_x$}{bad} 
 \HOL{$H_y$}{bad} 
 \CPO{$P$ \text{good}} 
\hskip94pt $\uparrow$} 
\ZU{ 
 \VLL{$H_x'$}{bad} 
 \VRL{$H_y'$}{bad} 
 \HOL{$E_P$}{good} 
 \LPO{$\widetilde{P}$ \hskip.3in \text{ or }} 
\RPO{$\widetilde{P}$} 
} 
\end{art_lem} 
\begin{proof}  Our lemma here is different from Claim 2 in \cite{KM2} in the following two aspects: 
\begin{itemize} 
\setlength{\parskip}{0pt} 
\setlength{\itemsep}{0pt} 
\item 
we study the behavior of the invariant $\mathrm{ord}(M_{\mathrm{usual}}) - \mathrm{ord}(M_{\mathrm{tight}})$ more closely, 
\item 
the proof does \emph{not} use the condition that, according to the algorithm, we blow up a point center only when $\mathrm{H}(D) < 1$ for any bad divisor  $D$ passing through $P$.  (Note that the proof of Claim 2 in \cite{KM2} heavily uses this condition.) 
\end{itemize} 
Despite these differences and some subtleties to be taken care of, however, the proof of Lemma 1 goes almost identical to that of Claim 2 in \cite{KM2}.  Therefore, we omit the proof here.  For the detailed calculation and proof, we refer the reader to \cite{KM3}. 
\end{proof} 
 
Now the verification of the assertions in 2. {\rm Case (b)} of Proposition \ref{invbehavior-prop} 
follows directly from Lemma 1 and the assumption that we exclude the case where the transformation $\pi$ is esoteric. 
We only provide a proof for the case where $P$ is in configuration \textcircled{\footnotesize 5}.  The argument for the other cases where $P$ is in configuration \textcircled{\footnotesize 3} or \textcircled{\footnotesize 4} is similar, and left to the reader as an exercise. 
 
\smallskip 
 
\noindent \underline{Case: $P$ is in configuration \textcircled{\footnotesize 5}} 
 
\smallskip 
 
\noindent \underline{Subcase:} 
$\widetilde{P}$ is in configuration \textcircled{\footnotesize 3}.  The transformation $\pi$ is standard. 
 
In this subcase, by the definition of $\pi$ being standard we have the inequality $\mathrm{lex}\left(w\text{-}\rho_x(P), w\text{-}\rho_y(P)\right) \geq w\text{-}\rho_{\widetilde{x}}(\widetilde{P})$.  By Lemma 1 (1) (1.3) we also have the inequality $\mathrm{ord}_P(M_{\mathrm{usual}}) - \mathrm{ord}_P(M_{\mathrm{tight}}) > \mathrm{ord}_{\widetilde{P}}(\widetilde{M}_{\mathrm{usual}}) - \mathrm{ord}_{\widetilde{P}}(\widetilde{M}_{\mathrm{tight}})$.  Therefore, we conclude 
\begin{align*} 
\mathrm{inv}_{\mathrm{MON,\spadesuit}}(P) &= 
\mathrm{lex}\left\{\mathrm{lex}\left(w\text{-}\rho_x(P), w\text{-}\rho_y(P)\right), \mathrm{ord}_P(M_{\mathrm{usual}}) - \mathrm{ord}_P(M_{\mathrm{tight}})\right\} \\ 
&> \mathrm{lex}\left\{w\text{-}\rho_{\widetilde{x}}, \mathrm{ord}_{\widetilde{P}}(\widetilde{M}_{\mathrm{usual}}) - \mathrm{ord}_{\widetilde{P}}(\widetilde{M}_{\mathrm{tight}})\right\} 
= \mathrm{inv}_{\mathrm{MON,\spadesuit}}(\widetilde{P}). 
\end{align*} 
\noindent \underline{Subcase:} 
$\widetilde{P}$ is in configuration \textcircled{\footnotesize 4} with $\widetilde{P} \in H_{\widetilde{x}}$. 
 
(The verification of the subcase where $\widetilde{P}$ is in configuration \textcircled{\footnotesize 4} with $\widetilde{P} \in H_{\widetilde{y}}$ is identical, and hence omitted.) 
 
In this subcase, by Lemma 1 (1) (1.1) we have the inequality $w\text{-}\rho_x(P) > w\text{-}\rho_{\widetilde{x}}(\widetilde{P})$.  By Lemma 1 (3) we also have the inequality $w\text{-}\rho_y(P) > \mathrm{ord}_{\widetilde{P}}(\widetilde{M}_{\mathrm{usual}}) - \mathrm{ord}_{\widetilde{P}}(\widetilde{M}_{\mathrm{tight}})$.  Moreover, we compute 
\begin{align*} 
\lefteqn{ 
\mathrm{ord}_P(M_{\mathrm{usual}}) - \mathrm{ord}_P(M_{\mathrm{tight}}) 
%= \left\{\mu(\xi_{H_x}) + \mu(\xi_{H_y})\right\} - 
%\left\{\mathrm{H}(\xi_{H_x}) + \mathrm{H}(\xi_{H_y})\right\} 
= \left\{ 
\mu(\xi_{H_x}) - \mathrm{H}(\xi_{H_x}) 
\right\} + \left\{ 
\mu(\xi_{H_y}) - \mathrm{H}(\xi_{H_y}) 
\right\} 
} 
\\ & 
\quad 
> 
\phantom{\bigl\{} 
\mu(\xi_{H_x}) - \mathrm{H}(\xi_{H_x}) 
\phantom{\bigl\}} 
\! = \! 
\phantom{\bigl\{} 
\mu(\xi_{H_{\widetilde{x}}}) - \mathrm{H}(\xi_{H_{\widetilde{x}}}) 
\phantom{\bigl\}} 
&\qquad\qquad 
{(\text{since }H_y \text{ is bad})\phantom{g}} 
\\ 
& 
\quad 
= \left\{ 
\mu(\xi_{H_{\widetilde{x}}}) - \mathrm{H}(\xi_{H_{\widetilde{x}}}) 
\right\} 
+ \left\{ 
\mu(\xi_{H_{\widetilde{y}}}) - \mathrm{H}(\xi_{H_{\widetilde{y}}}) 
\right\} 
&\qquad\qquad 
(\text{since $H_{\widetilde{y}}$ is good}) 
\\ 
& 
\quad 
= \mathrm{ord}_{\widetilde{P}}(\widetilde{M}_{\mathrm{usual}}) - \mathrm{ord}_{\widetilde{P}}(\widetilde{M}_{\mathrm{tight}}). 
\end{align*} 
Therefore, we conclude 
\begin{align*} 
\mathrm{inv}_{\mathrm{MON,\spadesuit}}(P) &= \mathrm{lex}\left\{\mathrm{lex}\left(w\text{-}\rho_x(P), w\text{-}\rho_y(P)\right), \mathrm{ord}_P(M_{\mathrm{usual}}) - \mathrm{ord}_P(M_{\mathrm{tight}})\right\} \\ 
&> \mathrm{lex}\left\{w\text{-}\rho_{\widetilde{x}}(\widetilde{P}), \mathrm{ord}_{\widetilde{P}}(\widetilde{M}_{\mathrm{usual}}) - \mathrm{ord}_{\widetilde{P}}(\widetilde{M}_{\mathrm{tight}})\right\} 
%\\ & 
= \mathrm{inv}_{\mathrm{MON,\spadesuit}}(\widetilde{P}). 
\end{align*} 
\noindent \underline{Subcase:} 
$\widetilde{P}$ is in configuration \textcircled{\footnotesize 5} with $\widetilde{P} \in H_{\widetilde{x}}$. 
 
(The verification of the subcase where $\widetilde{P}$ is in configuration \textcircled{\footnotesize 5} with $\widetilde{P} \in H_{\widetilde{y}}$ is identical, and hence omitted.) 
 
In this subcase, by Lemma 1 (1) (1.1) we have the inequality $w\text{-}\rho_x(P) > w\text{-}\rho_{\widetilde{x}}(\widetilde{P})$.  By Lemma 1 (1) (1.2) we have the inequality $w\text{-}\rho_y(P) \geq w\text{-}\rho_{\widetilde{y}}(\widetilde{P})$.  By Lemma 1 (2) we also have the inequality $\mathrm{ord}_P(M_{\mathrm{usual}}) - \mathrm{ord}_P(M_{\mathrm{tight}}) > w\text{-}\rho_{\widetilde{y}}(\widetilde{P})$.  Therefore, we conclude 
\begin{align*} 
\lefteqn{ 
\mathrm{inv}_{\mathrm{MON,\spadesuit}}(P) 
 = \mathrm{lex}\left\{\mathrm{lex}\left(w\text{-}\rho_x(P), w\text{-}\rho_y(P)\right), \mathrm{ord}_P(M_{\mathrm{usual}}) - \mathrm{ord}_P(M_{\mathrm{tight}})\right\} 
}\\ 
&\ > \mathrm{lex}\left\{\mathrm{lex}\left(w\text{-}\rho_{\widetilde{x}}(\widetilde{P}), w\text{-}\rho_{\widetilde{y}}(\widetilde{P})\right), \mathrm{ord}_{\widetilde{P}}(\widetilde{M}_{\mathrm{usual}}) - \mathrm{ord}_{\widetilde{P}}(\widetilde{M}_{\mathrm{tight}})\right\} 
= \mathrm{inv}_{\mathrm{MON,\spadesuit}}(\widetilde{P}). 
\end{align*} 
This completes the proof of Proposition \ref{invbehavior-prop}. 
\end{proof} 
 
Now we give the statement of the eventual decrease, whose proof, together with a more detailed analysis of the esoteric transformation, will be given in the next section. 
\begin{art_prop}[{Eventual Decrease}] 
Consider a transformation in the procedure specified by the algorithm given 
in \ref{5.1} 
(in the monomial case with $\tau = 1$) $P \in \mathrm{Sing}({\mathcal R}) \subset W \overset{\pi}\longleftarrow \widetilde{P} \in \pi^{-1}(P) \cap \mathrm{Sing}(\widetilde{\mathcal R}) \subset \widetilde{W}$ where the new invariant strictly increases, i.e., $\mathrm{inv}_{\mathrm{MON,\spadesuit}}(P) < \mathrm{inv}_{\mathrm{MON,\spadesuit}}(\widetilde{P})$ (while the invariant $\sigma$ stays the same and hence we remain in the monomial case with $\tau = 1$ after the transformation).  By Proposition \ref{invbehavior-prop} 
we observe that the transformation $\pi$ satisfies the following properties: 
\begin{itemize} 
\setlength{\parskip}{0pt} 
\setlength{\itemsep}{0pt} 
\item 
the transformation $\pi$ is necessarily esoteric, where the center of blow up is the point $P$, which is bad, and 
\item 
the point $P \in \mathrm{Sing}({\mathcal R}) \subset W$ is in configuration \textcircled{\footnotesize 4} or \textcircled{\footnotesize 5}, while the point 
$\widetilde{P} \in \mathrm{Sing}(\widetilde{\mathcal R}) \subset \widetilde{W}$ is in configuration \textcircled{\footnotesize 3}. 
\end{itemize} 
Now starting from $P$ in configuration \textcircled{\footnotesize 4} or \textcircled{\footnotesize 5}, followed by $\widetilde{P}$ and the points afterwards all in configuration \textcircled{\footnotesize 3}, we denote by $P^{\sharp}$ the first point to be in configuration \textcircled{\footnotesize 4} or \textcircled{\footnotesize 5}. 
 
Then we have $\mathrm{inv}_{\mathrm{MON,\spadesuit}}(P) > \mathrm{inv}_{\mathrm{MON,\spadesuit}}(P^{\sharp})$. 
\end{art_prop} 
\end{subsection} 
\begin{subsection}{Termination of the algorithm}\label{5.3} 
With the analysis of the behavior of the new invariant 
$\mathrm{inv}_{\mathrm{MON,\spadesuit}}$ 
in Proposition \ref{invbehavior-prop} 
and the statement of the eventual decrease in Proposition 5 at hand, we are now ready to show that the algorithm terminates after finitely many procedures. 
\begin{art_thm} Our algorithm for resolution of singularities in the monomial case with $\tau = 1$ in dimension 3, as described 
in \ref{5.1}, 
terminates after finitely many procedures.  More precisely, after finitely many procedures, we reach the situation where one of the following holds. 
 
Case 1. The singular locus is empty (over the fiber of the original point $P$). 
 
\noindent In this Case 1, we have already achieved (local) resolution of singularities. 
 
Case 2. The invariant $\sigma$ strictly decreases. 
 
\noindent In this Case 2, we achieve (local) resolution of singularities by induction on the invariant $\sigma$.  (More precisely, we go back to the reduction process 
{\rm \textbf{general case $\to$ monomial case}}, as described in Villamayor's philosophy 
\ref{1.5} (3), with the decreased value of the invariant $\sigma$.) 
 
Case 3. The invariant 
$\mathrm{inv}_{\mathrm{MON,\spadesuit}}$ becomes zero, 
i.e., we are in the tight monomial case.  In this Case 3, we achieve (local) resolution of singularities by following the procedure described in Proposition 2 (2). 
 
\smallskip 
 
In all of the cases above, therefore, we achieve (local) resolution of singularities. 
\end{art_thm} 
\begin{proof} Suppose we have a sequence of transformations for resolution of singularities in the monomial case with $\tau = 1$ in dimension 3, 
as described in \ref{5.1}. 
We may assume, throughout the sequence, that the singular locus locally at the point of reference (over the fiber of the original point) is never empty and that the invariant $\sigma$ stays the same (and hence we stay in the monomial case with $\tau = 1$).  Note that, otherwise, we are in Case 1 or Case 2 and we are done. 
 
\smallskip 
 
\noindent \underline{Observation 1.}  When the transformation $\pi$ has a 1-dimensional center (say, $C = V(z,x)$), we observe 
the following by Proposition \ref{invbehavior-prop}. 
\begin{itemize} 
\setlength{\parskip}{0pt} 
\setlength{\itemsep}{0pt} 
\item 
The new invariant $\mathrm{inv}_{\mathrm{MON,\spadesuit}}$ stays the same. 
\item 
The invariant $\mu(\xi_{H_x})$ strictly decreases (while the invariant $\mu(\xi_{H_y})$ remains the same). 
\end{itemize} 
Therefore, there is no infinite consecutive sequence of the transformations with 1-dimensional centers. 
 
\smallskip 
 
\noindent \underline{Observation 2.}  When the transformation $\pi$ has 
a point center $P$, we observe the following by Proposition \ref{invbehavior-prop}. 
\begin{itemize} 
\setlength{\parskip}{0pt} 
\setlength{\itemsep}{0pt} 
\item 
If the transformation $\pi$ is standard, then 
the new invariant $\mathrm{inv}_{\mathrm{MON,\spadesuit}}$ strictly 
decreases after the transformation. 
\item 
If the transformation $\pi$ is esoteric, then 
the transformation cannot be followed by infinite and consecutive points in configuration \textcircled{\footnotesize 3}. 
When we reach a point in configuration \textcircled{\footnotesize 4} or \textcircled{\footnotesize 5} (after finitely many points in configuration \textcircled{\footnotesize 3}), 
the new invariant $\mathrm{inv}_{\mathrm{MON,\spadesuit}}$ 
decreases to a value strictly lower than the original one by Proposition 5. 
\end{itemize} 
Note that the second item in Observation 2 is verified as follows. 
Assume that we have 
an infinite and consecutive sequence of points in configuration \textcircled{\footnotesize 3}.  Then, 
such a sequence would have to contain infinitely many transformations with point blow ups by Observation 1.  However, if the point is in 
configuration \textcircled{\footnotesize 3}, 
then the transformation with a point center is necessarily standard by Proposition \ref{invbehavior-prop}, and the new invariant 
$\mathrm{inv}_{\mathrm{MON,\spadesuit}}$ strictly decreases 
by the first item in Observation 2. 
Since the new invariant 
$\mathrm{inv}_{\mathrm{MON,\spadesuit}}$ 
stays the same under the transformation with 1-dimensional center by Observation 1, we conclude that the alleged sequence would give rise to 
an infinite and strictly decreasing sequence of the values of the new 
invariant $\mathrm{inv}_{\mathrm{MON,\spadesuit}}$, a contradiction ! 
 
\smallskip 
 
By Observations 1 and 2, we see that the resolution sequence gives rise to, by looking at some subsequence of the points, a strictly decreasing sequence of the values of the new invariant $\mathrm{inv}_{\mathrm{MON,\spadesuit}}$. 
Therefore, we finally conclude that after finitely many transformations, we reach the stage where the new invariant 
$\mathrm{inv}_{\mathrm{MON,\spadesuit}}$ 
becomes $0$ and hence we are in Case 3, i.e., in the tight monomial case. 
 
This completes the proof of Theorem 1. 
\end{proof} 
\begin{art_cor} Our algorithm for (local) resolution of singularities of an idealistic filtration in dimension 3 is complete, realizing the philosophy of Villamayor in the framework of the I.F.P. 
\end{art_cor} 
\begin{proof} This is an easy corollary to Theorem 1, once one understands the general mechanism of our algorithm explained 
in \ref{1.3}, \ref{1.4}, and \ref{1.5}. 
Note that this corollary has been already proved in our previous paper \cite{KM2}, using the old invariant $\mathrm{inv}_{\mathrm{MON}}$ 
in the monomial case.  Here we use the new invariant 
$\mathrm{inv}_{\mathrm{MON,\spadesuit}}$ 
in the monomial case.  The description of the algorithm in \cite{KM2} is different in the monomial case, but it turns out to provide exactly the same procedure as the one provided by the algorithm described in this paper 
(cf.~\ref{5.1}).  So the only difference lies in how we prove the termination in the monomial case. 
\end{proof} 
\end{subsection} 
\end{section} 
\begin{section}{Analysis of the jumping phenomenon and eventual decrease} 
\begin{subsection}{Detailed analysis of the ``Moh-Hauser jumping 
phenomenon''}\label{6.1} 
Here we give a more detailed analysis of the ``Moh-Hauser jumping 
phenomenon'' in dimension 3.  In short, the transformation that goes through the ``Moh-Hauser jumping 
phenomenon'' (more precisely, the transformation that is esoteric) is characterized by the initial form $\mathrm{In}_P(a_{p^e})$ of the last coefficient $a_{p^e}$ of the unique element $h$ in the L.G.S.. 
 
\smallskip 
 
\noindent \fbox{\rm \textbf{Situation}} 
 
\smallskip 
 
Consider a transformation in the procedure specified by the algorithm 
given in \ref{5.1} 
(in the monomial case with $\tau = 1$) $P \in \mathrm{Sing}({\mathcal R}) \subset W \overset{\pi}\longleftarrow \widetilde{P} \in \pi^{-1}(P) \cap \mathrm{Sing}(\widetilde{\mathcal R}) \subset \widetilde{W}$, where the new invariant strictly increases, i.e., $\mathrm{inv}_{\mathrm{MON,\spadesuit}}(P) < \mathrm{inv}_{\mathrm{MON,\spadesuit}}(\widetilde{P})$ (while the invariant $\sigma$ stays the same and hence we remain in the monomial case with $\tau = 1$ after the transformation).  By Proposition \ref{invbehavior-prop} we observe that the transformation $\pi$ satisfies the following properties: 
\begin{itemize} 
\setlength{\parskip}{0pt} 
\setlength{\itemsep}{0pt} 
\item 
the transformation $\pi$ is necessarily esoteric, where the center of blow up is the point $P$, which is bad, 
\item 
the point $P \in \mathrm{Sing}({\mathcal R}) \subset W$ is in configuration \textcircled{\footnotesize 4} or \textcircled{\footnotesize 5}, while the point 
$\widetilde{P} \in \mathrm{Sing}(\widetilde{\mathcal R}) \subset \widetilde{W}$ is in configuration \textcircled{\footnotesize 3}. 
\end{itemize} 
 
Observe 
\begin{align*} 
\mathrm{inv}_{\mathrm{MON,\heartsuit}}(P) &= \mathrm{H}(P) - \mathrm{ord}_P(M_{\mathrm{tight}}) \\ 
&\leq 
\begin{cases} 
w\text{-}\rho_x(P) 
&\text{if $P$ is in configuration \textcircled{\footnotesize 4}} \\ 
\mathrm{lex}\left(w\text{-}\rho_x(P), w\text{-}\rho_y(P)\right) 
&\text{if $P$ is in configuration \textcircled{\footnotesize 5}} 
\end{cases} 
\\ 
&< w\text{-}\rho_{\widetilde{x}}, 
\end{align*} 
where the first equality holds because the point $P$ is bad 
(cf.~Definitions 4 and 6), 
the second inequality holds in general (cf.~Remark 5 (2)), and the third inequality follows from the fact that the transformation $\pi$ is esoteric. 
Thus we have the inequality 
$$(\odot)\qquad \mathrm{inv}_{\mathrm{MON,\heartsuit}}(P) 
= \mathrm{H}(P) - \mathrm{ord}_P(M_{\mathrm{tight}}) 
< w\text{-}\rho_{\widetilde{x}}(\widetilde{P}).$$ 
 
Let $(h,p^e)$ be the unique element in the L.G.S. ${\mathbb H}$ of the idealistic filtration ${\mathcal R}$ at $P$.  Suppose that $h$ is in the Weierstrass form (cf.~\ref{4.1}) 
$$h = z^{p^e} + a_1z^{p^e-1} + a_2z^{p^e - 2} + \cdots + a_{p^e - 1}z + a_{p^e}$$ 
with $a_i \in k[[x,y]]$ and $\mathrm{ord}_P(a_i) > i 
\quad (i = 1, \ldots, p^e - 1, p^e)$ for a regular system of parameters $X = (z,x,y)$ of $\widehat{{\mathcal O}_{W,P}}$, and that the divisor $H_x = \{x = 0\}$ is the bad divisor while the divisor $H_y = \{y = 0\}$ is the good one in configuration \textcircled{\footnotesize 4} (resp. the divisors $H_x$ and $H_y$ are the two bad divisors in configuration \textcircled{\footnotesize 5}). 
Suppose further that $h$ is well-adapted with respect to $X$ 
at the closed point $P$ and at the generic points of the divisors $H_x$ and $H_y$ simultaneously. 
Since the point $\widetilde{P}$ is in configuration \textcircled{\footnotesize 3}, we may assume that it lies in the $x$-chart of the blow up and that it has a regular system of parameters $(\widetilde{z},\widetilde{x},\widetilde{y}) = (z/x,x,(y - cx)/x)$ for some $c \in k \setminus \{0\}$.  By replacing $x$ with $cx$, we may further assume that a regular system of parameters at $\widetilde{P}$ is given by $\widetilde{X} = (\widetilde{z},\widetilde{x},\widetilde{y}) = (z/x,x,(y - x)/x)$. 
 
Let $\mathrm{In}_P(a_{p^e}) = \Phi(x,y)$ be the initial form (i.e., the lowest degree homogeneous part) of the last coefficient $a_{p^e}$ of the element $h$. 
Note that, since $h$ is well-adapted with respect to $(z,x,y)$ at $P$ 
and since $P$ is a bad point, the initial form $\mathrm{In}_P(a_{p^e}) = \Phi(x,y)$ is 
\emph{not} a $p^e$-th power. 
 
\begin{art_prop}[{Analysis of the ``Moh-Hauser jumping phenomenon'' in dimension 3}]\label{MHjump-prop} 
Under the situation above, 
we assert that the homogeneous polynomial $\Phi(x,y)$ of degree $d = \deg \Phi(x,y) = \mathrm{ord}_P(a_{p^e})$ is, up to the multiplication by a nonzero constant, is in the following form 
$$\Phi(x,y) = x^r \cdot y^s \cdot (y -x)^t \cdot \psi(x,y)$$ 
satisfying the conditions below: 
\begin{enumerate} 
\setlength{\parskip}{0pt} 
\setlength{\itemsep}{0pt} 
\item[{\rm (i)}] 
$d = n \cdot p^e$ for some $n \in {\mathbb Z}_{> 1}$, 
\item[{\rm (ii)}] 
$r = \mathrm{ord}_{\xi_{H_x}}(a_{p^e}) = \mathrm{H}(\xi_{H_x}) \cdot p^e$ 
with $0 < r < p^e$, 
\item[{\rm (iii)}] 
$s = \lceil \mathrm{H}(\xi_{H_y}) \cdot p^e \rceil$ with $0 < s < p^e$, 
while we remark that 
\begin{itemize} 
\setlength{\parskip}{0pt} 
\setlength{\itemsep}{0pt} 
\item 
when $P$ is in configuration \textcircled{\footnotesize 4}, we have 
$s = \lceil \mathrm{H}(\xi_{H_y}) \cdot p^e \rceil \geq \mathrm{H}(\xi_{H_y}) \cdot p^e$, 
where the strict inequality may happen, and that 
\item 
when $P$ is in configuration \textcircled{\footnotesize 5}, we have $s = \mathrm{H}(\xi_{H_y}) \cdot p^e = \lceil \mathrm{H}(\xi_{H_y}) \cdot p^e \rceil$, 
\end{itemize} 
\item[{\rm (iv)}] 
$t = \mathrm{ord}_{\xi_{H_{y-x}}}\left(\Phi(x,y)\right) = l \cdot p^e$ for some $l \in {\mathbb Z}_{\geq 0}$ where $H_{y-x} = \{y - x = 0\}$, 
while we remark that $\psi(x,y)$ may be divisible by $x$ and/or $y$, but it is not divisible by $(y - x)$ according to the definition of $t$, and 
\item[{\rm (v)}] 
the homogeneous polynomial $\psi(x,y)$ of 
$\deg \psi(x,y) = u = d - (r + s + t)$ has 
the description below. 
\end{enumerate} 
 
\noindent \underline{\textbf{Description of $\psi(x,y)$}} 
 
\smallskip 
 
First we write $u = \alpha \cdot p^e + \beta$  with $\alpha, \beta \in {\mathbb Z}_{\geq 0} \text{ and } 0 \leq \beta < p^e$. 
Then there exist constants $\gamma_1, \gamma_2, \ldots, \gamma_{\alpha} \in k$ such that, writing the Taylor expansion 
$$\left(1 + \sum\nolimits_{i=1}^\alpha\gamma_i\widetilde{y}^{ip^e}\right)/(1 + \widetilde{y})^s = 1 + c_1\widetilde{y} + c_2\widetilde{y}^2 + \cdots + c_u\widetilde{y}^u + c_{u+1}\widetilde{y}^{u+1} \cdots,$$ 
we see that $\psi(x,y)$ is determined by the equality 
$$\psi(\widetilde{x},\widetilde{x} + \widetilde{x}\widetilde{y}) 
= \widetilde{x}^u \cdot (1 + c_1\widetilde{y} + c_2\widetilde{y}^2 + \cdots + c_u\widetilde{y}^u).$$ 
Since $\widetilde{y} = (y - x)/x$ and $\widetilde{x} = x$, this is equivalent to saying 
$$\psi(x,y) = x^u \cdot \left\{1 + c_1\left(\frac{y - x}{x}\right) + c_2\left(\frac{y - x}{x}\right)^2 + \cdots + c_u\left(\frac{y - x}{x}\right)^u \right\}.$$ 
 
Moreover, we see that the following inequality holds 
$$(\diamondsuit)\quad 0 \leq \beta < p^e - s.$$ 
We also have the following upper bounds for the increase 
\begin{align*} 
(\clubsuit 1) &\quad \{w\text{-}\rho_{H_{\widetilde{x}}}(\widetilde{P}) - \mathrm{inv}_{\mathrm{MON,\heartsuit}}(P)\} \cdot p^e \leq s < p^e, \\ 
(\clubsuit 2) &\quad \{w\text{-}\rho_{H_{\widetilde{x}}}(\widetilde{P}) - \mathrm{inv}_{\mathrm{MON,\heartsuit}}(P)\} \cdot p^e \leq p^{e - 1}. 
\end{align*} 
(We remark that the second inequality is called ``the Moh's inequality''.) 
\end{art_prop} 
\begin{art_rem} 
\item[(1)] 
Since $\psi(x,y)$ is not divisible by $(y - x)$, we conclude that $\psi(\widetilde{x},\widetilde{x} + \widetilde{x}\widetilde{y})/\widetilde{x}^u$ is not divisible by $\widetilde{y}$ and hence that it starts with the nonzero constant term $c_0$. 
We multiply $1/c_0$ so that this constant term becomes equal to $1$, i.e., $c_0 = 1$.  This is the meaning of ``$\Phi(x,y)$ is, up to the multiplication by a nonzero constant, in the following form'' 
as stated above. 
\item[(2)] 
What we actually do in the following proof is to present the characterization of the initial form $\mathrm{In}_P(a_{p^e}) = \Phi(x,y)$, which gives rise to the inequality $(\odot)$. 
\end{art_rem} 
\begin{proof}[Proof of Proposition \ref{MHjump-prop}]\ 
 
\noindent \textbf{Step 1.} Preliminary computation of the invariants after the transformation $\pi$. 
 
\smallskip 
 
Set $\widetilde{h} = \pi^*(h)/x^{p^e} = \widetilde{z}^{p^e} + \widetilde{a_1}\widetilde{z}^{p^e - 1} + \cdots + \widetilde{a_{p^e-1}}\widetilde{z} + \widetilde{a_{p^e}}$, where $\widetilde{a_i} = \pi^*(a_i)/x^i$ for $i = 1, \ldots, p^e -1, p^e$.  We write $\widetilde{a_{p^e}} = \widetilde{x}^{d - p^e} \cdot \left\{\pi^*(\Phi(x,y))/x^d + \widetilde{x} \cdot \omega_{\widetilde{x}}(\widetilde{x},\widetilde{y})\right\}$, where $0 \neq \pi^*(\Phi(x,y))/x^d \in k[\widetilde{y}]$ and $\omega_{\widetilde{x}}(\widetilde{x},\widetilde{y}) \in k[[\widetilde{x},\widetilde{y}]]$.  We observe that $\widetilde{h}$ is well-adapted at $\xi_{\widetilde{x}}$ with respect to $(\widetilde{z},\widetilde{x},\widetilde{y})$, since 
\begin{align*} 
\mathrm{Slope}_{\widetilde{h},(\widetilde{z},\widetilde{x},\widetilde{y})}(\xi_{H_{\widetilde{x}}}) 
&= \mathrm{ord}_{\xi_{H_{\widetilde{x}}}}(\widetilde{a_{p^e}})/p^e = (d - p^e)/p^e 
= \mathrm{ord}_P(a_{p^e})/p^e - 1 
\\& 
< \mu(P) - 1 = \mu(\xi_{H_{\widetilde{x}}}) 
\quad\quad(\text{since $P$ is bad}), 
\end{align*} 
and since $\mathrm{In}_{\xi_{H_{\widetilde{x}}}}(\widetilde{a_{p^e}}) = \widetilde{x}^{d-p^e} \cdot \pi^*(\Phi(x,y))/x^d = \pi^*(\Phi(x,y))/x^{p^e}$ is not a $p^e$-th power, a fact which follows easily from the fact that $\Phi(x,y)$ is not a $p^e$-th power. 
 
We remark that the requirement $\mathrm{ord}_{\widetilde{P}}(\widetilde{a_i}) > i \hskip.01in \text{ for }\hskip.01in i = 1, \ldots, p^e - 1$ follows from the conditions that $\widetilde{P} \in \mathrm{Sing}({\mathcal R})$ and that the invariant $\sigma$ stays the same.  The inequality $\mathrm{ord}_{\widetilde{P}}(\widetilde{a_{p^e}}) \geq p^e$ follows from the same conditions above.  However, the equality can happen when the term $x^{2p^e}$ shows up with a nonzero coefficient $\gamma$ in $a_{p^e}$, and then we have to replace $\widetilde{z}$ with $\widetilde{z}' = \widetilde{z} + \gamma^{1/p^e}\widetilde{x}$ in order to have the strict inequality.  However, it is easy to see that this replacement does not affect the computation of $w\text{-}\rho_{\widetilde{x}}(\widetilde{P})$, and hence we ignore this replacement. 
 
Therefore, we conclude that $\mathrm{H}(\xi_{H_{\widetilde{x}}}) = (d - p^e)/p^e$. 
 
We also compute 
\begin{align*} 
w\text{-}\rho_{\widetilde{x}}(\widetilde{P}) 
&= \mathrm{res\text{-}ord}^{(p^e)}_{\widetilde{P}}\left(\mathrm{In}_{\xi_{H_{\widetilde{x}}}}(\widetilde{a_{p^e}})\right)/p^e - \mathrm{ord}_{\widetilde{P}}\left(\widetilde{M}_{\mathrm{tight}}\right) \\ 
&= \mathrm{res\text{-}ord}^{(p^e)}_{\widetilde{P}}\left(\widetilde{x}^{d - p^e} \cdot \pi^*(\Phi(x,y))/x^d\right)/p^e - \mathrm{H}(\xi_{H_{\widetilde{x}}}) \\ 
&= \mathrm{res\text{-}ord}^{(p^e)}_{\widetilde{P}}\left(\widetilde{x}^{d - p^e} \cdot \pi^*(\Phi(x,y))/x^d\right)/p^e - (d - p^e)/p^e. 
\end{align*} 
\textbf{Step 2.} Check the conditions (i), (ii), (iii), (iv) on the numbers $d, r, s, t$. 
 
\smallskip 
 
Now write the homogeneous polynomial $\mathrm{In}_P(a_{p^e}) = \Phi(x,y)$ of degree $d$ in the following form $\Phi(x,y) = x^r \cdot y^s \cdot (y - x)^t \cdot \psi(x,y)$, where 
\begin{enumerate} 
\renewcommand{\labelenumi}{(\roman{enumi})} 
\setlength{\parskip}{0pt} 
\setlength{\itemsep}{0pt} 
\item 
$d = \deg \Phi(x,y) = \mathrm{ord}_P(a_{p^e})$, 
\item 
$r = \mathrm{ord}_{\xi_{H_x}}(a_{p^e}) = \mathrm{H}(\xi_{H_x}) \cdot p^e$, 
\item 
$s = \lceil \mathrm{H}(\xi_{H_y}) \cdot p^e \rceil$, 
\item 
$t = \mathrm{ord}_{\xi_{H_{y-x}}}\left(\Phi(x,y)\right)$ (which implies $\psi(x,y)$ is not divisible by $(y - x)$), and 
\item 
$\psi(x,y)$ is a homogeneous polynomial of $\deg \psi(x,y) = u = d - (r + s + t)$. 
\end{enumerate} 
We compute $\pi^*(\Phi(x,y))/x^d = \Phi(\widetilde{x},\widetilde{x} + \widetilde{x}\widetilde{y})/\widetilde{x}^d = \Phi(1,1 + \widetilde{y}) = 1^r \cdot (1 + \widetilde{y})^s \cdot  \widetilde{y}^t \cdot \psi(1,1 + \widetilde{y})$, where $\psi(1,1 + \widetilde{y})$ is not divisible by $\widetilde{y}$. 
\begin{claim} 
The numbers $d$ and $t$ are both multiples of $p^e$, i.e., 
$d = n\cdot p^e$ and $t = l\cdot p^e$ for some $n,l \in {\mathbb Z}_{\geq 0}$. 
\end{claim} 
\begin{proof}[Proof of the claim.] 
Suppose either that $d$ is not a multiple of $p^e$ or that $t$ is not a multiple of $p^e$. 
Then we compute 
\begin{align*} 
\lefteqn{ 
w\text{-}\rho_{\widetilde{x}}(\widetilde{P}) 
= \mathrm{res\text{-}ord}^{(p^e)}_{\widetilde{P}}\left(\widetilde{x}^{d - p^e} \cdot \pi^*(\Phi(x,y))/x^d\right) - (d - p^e)/p^e 
}\\ 
&\quad 
= \mathrm{res\text{-}ord}^{(p^e)}_{\widetilde{P}}\left(\widetilde{x}^{d - p^e} \cdot 1^r \cdot (1 + \widetilde{y})^s \cdot \widetilde{y}^t \cdot \psi(1,1 + \widetilde{y})\right)/p^e - (d - p^e)/p^e \\ 
&\quad 
= \mathrm{ord}_{\widetilde{P}}(\widetilde{x}^{d - p^e} \cdot \widetilde{y}^t)/p^e - (d - p^e)/p^e = ((d - p^e) + t)/p^e - (d - p^e)/p^e = t/p^e \\ 
&\quad 
\leq \left(d - (r + s)\right)/p^e \leq \mathrm{ord}_P(a_{p^e})/p^e - \left(\mathrm{H}(\xi_{H_x}) + \mathrm{H}(\xi_{H_y})\right) \\ 
&\quad 
= \mathrm{H}(P) - \mathrm{ord}_P(M_{\text{tight}}), 
\end{align*} 
contradicting the inequality $(\odot)$.  This completes the proof of the claim. 
\end{proof} 
 
Thanks to the claim above and the condition that 
$d = \mathrm{ord}_P(a_{p^e}) > p^e$, 
we have 
$d \in p^e\cdot{\mathbb Z}_{>1}$ and $t \in p^e\cdot{\mathbb Z}_{\geq0}$, 
confirming the conditions (i) and (iv). 
%\begin{enumerate} 
%\setlength{\parskip}{0pt} 
%\setlength{\itemsep}{0pt} 
%\item[(i)] 
%$d = n \cdot p^e$ for some $n \in {\mathbb Z}_{> 1}$, 
%since $p^e\mid d = \mathrm{ord}_P(a_{p^e}) > p^e.$ 
%\item[(iv)] 
%$t = \mathrm{ord}_{\xi_{H_{y-x}}}(\Phi(x,y)) = l \cdot p^e$ for some $l \in {\mathbb Z}_{\geq 0}$. 
%\end{enumerate} 
\begin{claim} 
The number $s$ is not a multiple of $p^e$. 
\end{claim} 
\begin{proof}[Proof of the claim.] 
Suppose that $s$ is a multiple of $p^e$.  Then since $d$ and $t$ are both multiples of $p^e$ and since $\pi^*(\Phi(x,y))/x^d = \Phi(1,1 + \widetilde{y}) = 1^r \cdot (1 + \widetilde{y})^s \cdot  \widetilde{y}^t \cdot \psi(1,1 + \widetilde{y})$ is not a $p^e$-th power because $\Phi(x,y)$ is not, we would conclude that $\psi(1,1 + \widetilde{y})$ is not a $p^e$-th power.  Then we compute 
\begin{align*} 
\lefteqn{ 
w\text{-}\rho_{\widetilde{x}}(\widetilde{P}) 
= \mathrm{res\text{-}ord}^{(p^e)}_{\widetilde{P}}\left(\widetilde{x}^{d - p^e} \cdot \pi^*(\Phi(x,y))/x^d\right) - (d - p^e)/p^e 
}\\ 
&\quad 
= \mathrm{res\text{-}ord}^{(p^e)}_{\widetilde{P}}\left(\widetilde{x}^{d - p^e} \cdot 1^r \cdot (1 + \widetilde{y})^s \cdot \widetilde{y}^t \cdot \psi(1,1 + \widetilde{y})\right)/p^e - (d - p^e)/p^e 
\\ 
&\quad 
= \mathrm{res\text{-}ord}^{(p^e)}_{\widetilde{P}}\left(\widetilde{x}^{d - p^e} \cdot \widetilde{y}^t \cdot \psi(1,1 + \widetilde{y})\right)/p^e - (d - p^e)/p^e 
\\ 
&\quad 
= \mathrm{res\text{-}ord}^{(p^e)}_{\widetilde{P}}\left(\widetilde{y}^t \cdot \psi(1,1 + \widetilde{y})\right)/p^e \leq \left(t +\deg \psi(x,y)\right)/p^e = \left(d - (r + s)\right)/p^e 
\\ 
&\quad 
\leq \mathrm{ord}_P(a_{p^e})/p^e - \left(\mathrm{H}(\xi_{H_x}) + \mathrm{H}(\xi_{H_y})\right) = \mathrm{H}(P) - \mathrm{ord}_P(M_{\mathrm{tight}}), 
\end{align*} 
contradicting the inequality $(\odot)$.  This completes the proof of the claim. 
\end{proof} 
 
Observe that $s = \lceil \mathrm{H}(\xi_{H_y}) \cdot p^e \rceil \leq p^e$.  In fact, if $s > p^e$, then $\mathrm{H}(\xi_{H_y}) \cdot p^e > p^e$ and hence $h_y = \mathrm{H}(\xi_{H_y}) > 1$.  But then we would choose the center of blow up $C = V(z,y)$ for the transformation $\pi$ 
(cf.~\ref{5.1} Description of the algorithm). 
This contradicts the description of our center to be the point $P$ 
as given in the situation. 
Combining the claim above with the inequality $s \leq p^e$, 
we have $0<s<p^e$, confirming the condition (iii). 
%\begin{enumerate} 
%\setlength{\parskip}{0pt} 
%\setlength{\itemsep}{0pt} 
%\item[(iii)] 
%$s = \lceil \mathrm{H}(\xi_{H_y}) \cdot p^e \rceil$ with $0 < s < p^e$. 
%\end{enumerate} 
 
Suppose $r = \mathrm{H}(\xi_{H_x}) \cdot p^e \geq p^e$.  But then again we would choose the center of blow up $C = V(z,x)$ for the transformation $\pi$ 
(cf.~\ref{5.1} Description of the algorithm). 
This contradicts the description of our center to be the point $P$ 
as given in the situation. 
Suppose $r = 0$.  Then, since $r + s + t + u = s + l \cdot p^e + \alpha \cdot p^e + \beta = d = n \dot p^e$ 
and since $0 < s < p^e$ by (iii) and $0 \leq \beta < p^e$ by definition, 
we would conclude $s + \beta = p^e$. 
This contradicts the inequality $(\diamondsuit)$. 
Therefore, we have $0<r<p^e$, 
confirming the condition (ii). 
%\begin{enumerate} 
%\setlength{\parskip}{0pt} 
%\setlength{\itemsep}{0pt} 
%\item[(ii)] 
%$r = \mathrm{H}(\xi_{H_x}) \cdot p^e$ with $0 < r < p^e$. 
%\end{enumerate} 
 
This finishes Step 2. 
 
\smallskip 
 
\noindent \textbf{Step 3.} Check the description of the polynomial $\psi(x,y)$. 
 
\smallskip 
 
First observe 
\begin{align*} 
\lefteqn{ 
\deg\left((y - x)^t \cdot \psi(x,y)\right) = d - (r + s) 
}\\ 
&\quad 
\leq d - (\mathrm{H}(\xi_{H_x}) + \mathrm{H}(\xi_{H_y})) \cdot p^e 
= \left(\mathrm{H}(P) - \mathrm{ord}_P(M_{\text{tight}})\right) \cdot p^e 
\\ 
&\quad 
< w\text{-}\rho_{\widetilde{x}}(\widetilde{P}) \cdot p^e = \mathrm{res\text{-}ord}^{(p^e)}_{\widetilde{P}}\left((1 + \widetilde{y})^s \cdot \widetilde{y}^t \cdot \psi(1,1 + \widetilde{y})\right). 
\end{align*} 
Since $t$ is a multiple of $p^e$, this implies 
$$u = \deg \psi(x,y) < \mathrm{res\text{-}ord}^{(p^e)}_{\widetilde{P}}\left((1 + \widetilde{y})^s \cdot \psi(1,1 + \widetilde{y})\right).$$ 
Therefore, writing $u = \alpha \cdot p^e + \beta$ with $\alpha, \beta \in {\mathbb Z}_{\geq 0}$ and $0 \leq \beta < p^e$, we conclude that there exist constants $\gamma_1, \gamma_2, \ldots, \gamma_{\alpha} \in k$ such that 
$$(1 + \widetilde{y})^s \cdot \psi(1,1 + \widetilde{y}) = 
1 + \sum\nolimits_{i=1}^\alpha\gamma_i\widetilde{y}^{ip^e} 
\text{ mod }(\widetilde{y}^{u + 1}).$$ 
(Note that we assume that the constant term of $\psi(1,1 + \widetilde{y})$, not divisible $\widetilde{y}$, is equal to $1$ (cf.~Remark 7 (1)).)  Now multiplying $1/(1 + \widetilde{y})^s$ to both sides of the equation above, and writing the Taylor expansion of the right hand side as 
$$\left(1 + \sum\nolimits_{i=1}^\alpha\gamma_i\widetilde{y}^{ip^e}\right)/(1 + \widetilde{y})^s = 1 + c_1\widetilde{y} + c_2\widetilde{y}^2 + \cdots + c_u\widetilde{y}^u + c_{u+1}\widetilde{y}^{u+1} \cdots,$$ 
we conclude that the left hand side is equal to the truncation of the Taylor expansion up to degree $u$ 
$$\psi(\widetilde{x},\widetilde{x}+\widetilde{x}\widetilde{y})/\widetilde{x}^u 
= \psi(1,1 + \widetilde{y}) = 1 + c_1\widetilde{y} + c_2\widetilde{y}^2 + \cdots + c_u\widetilde{y}^u,$$ 
knowing that $\psi(1,1 + \widetilde{y})$ is a polynomial of degree at most $u$ in $\widetilde{y}$.  That is to say, we have 
$$\psi(\widetilde{x},\widetilde{x}+\widetilde{x}\widetilde{y}) 
= \widetilde{x}^u \cdot \left(1 + c_1\widetilde{y} + c_2\widetilde{y}^2 + \cdots + c_u\widetilde{y}^u\right).$$ 
%%%%%%%%%%%%%%%%%%%%% 
Since $\widetilde{y} = (y - x)/x$ and $\widetilde{x} = x$, this is equivalent to saying 
$$\psi(x,y) = x^u \cdot \left\{1 + c_1\left(\frac{y - x}{x}\right) + c_2\left(\frac{y - x}{x}\right)^2 + \cdots + c_u\left(\frac{y - x}{x}\right)^u \right\}.$$ 
 
\medskip 
 
\noindent \textbf{Step 4.} Check the inequalities $(\diamondsuit), (\clubsuit 1), (\clubsuit 2)$. 
\begin{claim} 
The following inequality holds:\qquad 
$(\diamondsuit)\quad 0 \leq \beta < p^e - s$. 
\end{claim} 
\begin{proof}[Proof of the claim.] 
Assuming $p^e - s \leq \beta$, we would conclude that $\Phi(x,y)$ is a $p^e$-th power, which contradicts the condition 
described in the situation that 
$\mathrm{In}_P(a_{p^e}) = \Phi(x,y)$ is \emph{not} a $p^e$-th power. 
 
\smallskip 
 
First note that the Taylor expansion of $(1 + \widetilde{y})^{-p^e} = 1/(1 + \widetilde{y})^{p^e} = 1/(1 + \widetilde{y}^{p^e})$ involves only the powers of $\widetilde{y}^{p^e}$. 
Let us denote by $[f]_j$ the truncation of the Taylor expansion of $f$ 
up to order $j$. Then we compute 
\begin{align*} 
&\psi(1,1 + \widetilde{y}) 
= 1 + c_1\widetilde{y} + c_2\widetilde{y}^2 + \cdots + c_u\widetilde{y}^u 
= \left[ 
\left(1 + \sum\nolimits_{i=1}^\alpha\gamma_i\widetilde{y}^{ip^e}\right)/(1 + \widetilde{y})^s 
\right]_{\leq u} 
\\ 
&\quad 
= \left[ 
\left(1 + \sum\nolimits_{i=1}^\alpha\gamma_i\widetilde{y}^{ip^e}\right)(1 + \widetilde{y})^{- p^e}(1 + \widetilde{y})^{p^e - s} 
\right]_{\leq u} 
= \left[ 
\Gamma(\widetilde{y}^{p^e}) \cdot (1 + \widetilde{y})^{p^e - s} 
\right]_{\leq u} 
\end{align*} 
where 
$\Gamma(\widetilde{y}^{p^e})=\left[ 
\left(1 + \sum_{i=1}^\alpha\gamma_i\widetilde{y}^{ip^e}\right) 
(1 + \widetilde{y})^{- p^e} 
\right]_{\leq \alpha p^e}$ is a polynomial of $\widetilde{y}^{p^e}$. 
But if $p^e - s \leq \beta$ and hence $\alpha \cdot p^e + p^e - s \leq \alpha \cdot p^e + \beta = u$, then we do not have to take the truncation (up to degree $u$) in the last term, 
i.e., we have $\psi(1,1 + \widetilde{y}) = \Gamma(\widetilde{y}^{p^e}) \cdot (1 + \widetilde{y})^{p^e - s}$.  Therefore, we conclude that 
\begin{align*} 
\pi^*(\Phi(x,y))/x^d &= \Phi(1,1 + \widetilde{y}) = 1^r \cdot (1 + \widetilde{y})^s \cdot  \widetilde{y}^t \cdot \psi(1,1 + \widetilde{y}) \\ 
&= (1 + \widetilde{y})^s \cdot  \widetilde{y}^t \cdot \Gamma(\widetilde{y}^{p^e}) \cdot (1 + \widetilde{y})^{p^e - s} = (1 + \widetilde{y})^{p^e} \cdot \widetilde{y}^t \cdot \Gamma(\widetilde{y}^{p^e}) 
\end{align*} 
is a $p^e$-th power, since $t = l \cdot p^e$ for some $l \in {\mathbb Z}_{\geq 0}$ by (iv).  This implies that $\Phi(x,y)$ is a $p^e$-th power, since $d$ is a multiple of $p^e$ by (i). 
This contradicts the condition that $\Phi(x,y)$ is not a $p^e$-th power. 
 
This completes the proof of the inequality $(\diamondsuit)$. 
\end{proof} 
\begin{claim} 
The following inequality holds: 
$$(\clubsuit 1)  \quad 
\{w\text{-}\rho_{\widetilde{x}}(\widetilde{P}) - \mathrm{inv}_{\mathrm{MON,\heartsuit}}(P)\} \cdot p^e \leq s < p^e.$$ 
\end{claim} 
\begin{proof}[Proof of the claim.] 
Using the same notation as in Step 1, we compute 
\begin{align*} 
\lefteqn{ 
\{w\text{-}\rho_{\widetilde{x}}(\widetilde{P}) - \mathrm{inv}_{\mathrm{MON,\heartsuit}}(P)\} \cdot p^e 
}\\ 
&\quad 
= \mathrm{res\text{-}ord}^{(p^e)}_{\widetilde{P}}\left((1 + \widetilde{y})^s \cdot \widetilde{y}^t \cdot \psi(1,1 + \widetilde{y})\right) - \left(\mathrm{ord}_P(a_{p^e}) 
- \mathrm{ord}_P(M_{\text{tight}})\cdot p^e\right) 
\\ 
&\quad 
= \mathrm{res\text{-}ord}^{(p^e)}_{\widetilde{P}}\left((1 + \widetilde{y})^s \cdot \widetilde{y}^t \cdot \psi(1,1 + \widetilde{y})\right) - \left(d - (r + \mathrm{H}(\xi_y) \cdot p^e)\right) \\ 
&\quad 
\leq (s + t + u) - (d - (r+s)) = (s + t + u) - (t + u) = s, 
\end{align*} 
proving the inequality $(\clubsuit1)$. 
\end{proof} 
 
Set $v_o = \min \{v \geq u+1 \mid c_v \neq 0\}$, 
$w_o = \min\{w \geq \beta + 1\mid \binom{p^e - s}{w} \neq 0\}$ and 
$$ 
\left(1 + \sum\nolimits_{i=1}^\alpha\gamma_i\widetilde{y}^{ip^e}\right) 
\cdot (1 + \widetilde{y})^{- p^e} 
= 1 + \epsilon_1\widetilde{y}^{p^e} + \epsilon_2\widetilde{y}^{2p^e} 
+\dotsb+ \epsilon_{\alpha}\widetilde{y}^{\alpha p^e} + \epsilon_{\alpha+1}\widetilde{y}^{(\alpha + 1)p^e} + \cdots. 
$$ 
The following claim is used for the proof of 
the inequality $(\clubsuit2)$. 
\begin{claim} 
The following inequalities hold. 
\item[\rm (1)] $\epsilon_\alpha\neq0$. 
\item[\rm (2)] $w_o - \beta \leq p^{e-1}$. 
\item[\rm (3)] 
$\alpha \cdot p^e \leq \alpha \cdot p^e + \beta = u < v_o \leq \alpha 
 \cdot p^e + (p^e - s) < (\alpha + 1) \cdot p^e$. 
\item[\phantom{\rm (3)}] 
In particular, we have $v_o \not\equiv 0 \text{ mod }p^e$.  Moreover, we have $v_o - u \leq p^{e-1}$. 
\end{claim} 
\begin{proof}[Proof of the claim.] 
We use the same notation as in the proof of the inequality $(\diamondsuit)$. 
It follows from the definition of $c_i$ and $\epsilon_i$ that 
$$ 
1 + \sum\nolimits_{i=1}^\infty c_i\widetilde{y}^i 
= \left(1 + \sum\nolimits_{i=1}^\infty\epsilon_i\widetilde{y}^{ip^e}\right) 
(1 + \widetilde{y})^{p^e - s}. 
$$ 
\item[\rm (1)] 
Suppose $\epsilon_{\alpha} = 0$. 
Then we would have 
$$\sum\nolimits_{\alpha \cdot p^e \leq v < (\alpha + 1) \cdot p^e} 
c_v\widetilde{y}^v = \epsilon_{\alpha}\widetilde{y}^{\alpha p^e} 
(1 + \widetilde{y})^{p^e - s} = 0,$$ 
i.e., $c_v = 0$ for all $v$ with 
$\alpha \cdot p^e \leq v < (\alpha + 1) \cdot p^e$.  But this implies 
\begin{align*} 
&(1 + \widetilde{y})^s \cdot \psi(1,1 + \widetilde{y}) 
= (1 + \widetilde{y})^s \cdot (1 + c_1\widetilde{y} + c_2\widetilde{y}^2 + \cdots + c_u\widetilde{y}^u) \\ 
&\quad 
= (1 + \widetilde{y})^s \cdot \left\{\left(1 + \sum\nolimits_{i=1}^\alpha\gamma_i\widetilde{y}^{ip^e}\right)/(1 + \widetilde{y})^s - (c_{u+1}\widetilde{y}^{u+1} + \cdots)\right\}\\ 
&\quad 
= (1 + \widetilde{y})^s \cdot \left\{\left(1 + \sum\nolimits_{i=1}^\alpha\gamma_i\widetilde{y}^{ip^e}\right)/(1 + \widetilde{y})^s 
- (c_{(\alpha + 1) \cdot p^e}\widetilde{y}^{(\alpha + 1) \cdot p^e} 
+ \cdots)\right\}\\ 
&\quad 
= 1 + \sum\nolimits_{i=1}^\alpha\gamma_i\widetilde{y}^{ip^e} 
-  (1 + \widetilde{y})^s \cdot 
\sum\nolimits_{i=(\alpha + 1) \cdot p^e}^\infty 
c_{i}\widetilde{y}^{i}, 
\end{align*} 
where the equalities above are equalities as the Taylor expansions in $\widetilde{y}$. 
Since 
\begin{align*} 
\deg \left((1 + \widetilde{y})^s \cdot \psi(1,1 + \widetilde{y})\right) 
&\leq s + u = s + \alpha \cdot p^e + \beta \\ 
&< s + \alpha \cdot p^e + p^e - s = (\alpha + 1) \cdot p^e, 
\end{align*} 
where the second inequality follows from $(\diamondsuit)\ \beta < p^e - s$, we have the following equality as polynomials 
$$(1 + \widetilde{y})^s \cdot \psi(1,1 + \widetilde{y}) 
= 1 + \sum\nolimits_{i=1}^\alpha\gamma_i\widetilde{y}^{ip^e}$$ 
But then the equality above implies that 
$$ 
\pi^*(\Phi(x,y))/x^d 
%= \Phi(1,1 + \widetilde{y}) 
= 1^r \cdot (1 + \widetilde{y})^s \cdot  \widetilde{y}^t \cdot \psi(1,1 + \widetilde{y}) 
= \widetilde{y}^t \cdot \left(1 + \sum\nolimits_{i=1}^\alpha\gamma_i\widetilde{y}^{ip^e}\right) 
$$ 
is a $p^e$-th power (since $t$ is a multiple of $p^e$ by (iv)).  Since $d$ is a multiple of $p^e$ by (i), this in turn implies that $\Phi(x,y)$ is a $p^e$-th power, 
contradicting the condition described in the situation.  This completes the proof of (1). 
\item[\rm (2)] 
We draw the attention of the reader to the following elementary fact about the binomial coefficients in $\mathrm{char}(k) = p > 0$ (cf.~\cite{K}): 
\begin{itemize} 
\setlength{\parskip}{0pt} 
\setlength{\itemsep}{0pt} 
\item[$(\aleph)$\quad] 
Suppose that two nonnegative integers $\kappa, m$ have $p$-adic expansions 
$\kappa = \sum_l a_lp^l$ and $m = \sum_l b_lp^l$.  Then we have 
$$\binom{\kappa}{m} = \prod\nolimits_l \binom{a_l}{b_l}.$$ 
Especially, we have 
$$\binom{\kappa}{m} \neq 0 \Longleftrightarrow a_l \geq b_l, \quad \forall l.$$ 
\end{itemize} 
Now write the $p$-adic expansions of $p^e - s$ and $\beta (< p^e - s)$ as $p^e - s = \sum_{l = 0}^{e - 1}a_lp^l$ and $\beta = \sum_{l = 0}^{e - 1}\delta_lp^l$.  Observe that, for $\beta + 1 \leq w \leq p^e - s$ with $p$-adic expansion $w = \sum_{l = 0}^{e-1}b_lp^l$, using $(\aleph)$ we have 
$\binom{p^e - s}{w} \neq 0 \Longleftrightarrow a_l \geq b_l\quad 0 \leq \forall l \leq e - 1$. 
Define 
$$l_o = 
\begin{cases} 
-1 &\text{if } \hskip.1in a_l \geq \delta_l,\quad 0 \leq \forall l \leq e - 1,\\ 
\max\{0 \leq l \leq e - 1\mid a_l < \delta_l\} &\text{otherwise}. 
\end{cases} 
$$ 
Set $l_1 = \min\{l_o + 1 \leq l \leq e - 1\mid a_l > {\delta}_l\}$. 
Note that $a_l \geq {\delta}_l \hskip.1in l_o + 1 \leq \forall l \leq e - 1$ by the definition of $l_o$, 
and that $\{l_o + 1 \leq l \leq e - 1\mid a_l > {\delta}_l\} \neq \emptyset$ 
since $p^e - s > \beta$.  Now it is straightforward to see that $w_o = ({\delta}_{l_1} + 1)p^{l_1} + \sum_{l_1 + 1 \leq l \leq e - 1}{\delta}_lp^l$ and hence that $w_o - \beta \leq p^{l_1} \leq p^{e-1}$. 
This completes the proof of (2). 
\item[(3)] 
Note that by (1) we have 
$v_o = \alpha \cdot p^e + w_o$. 
This implies the inequalities 
$$\alpha \cdot p^e \leq \alpha \cdot p^e + \beta 
= u < v_o \leq \alpha \cdot p^e + (p^e - s) 
<(\alpha+1) \cdot p^e,$$ 
and hence $v_o \not\equiv 0 \text{ mod }p^e$. 
The last inequality in (3) follows from (2), since 
$$v_o - u = (\alpha \cdot p^e + w_o) - (\alpha \cdot p^e + \beta) 
= w_o - \beta \leq p^{e-1}.$$ 
This completes the proof of (3). 
\end{proof} 
\begin{claim} 
The following inequality holds: 
$$ 
(\clubsuit 2)  \quad 
 \{w\text{-}\rho_{\widetilde{x}}(\widetilde{P}) - \mathrm{inv}_{\mathrm{MON,\heartsuit}}(P)\} \cdot p^e \leq p^{e - 1}. 
$$ 
\end{claim} 
\begin{proof}[Proof of the claim.] 
First remember that 
\begin{align*} 
\lefteqn{ 
\pi^*(\Phi(x,y))/x^d 
%= \Phi(1,1 + \widetilde{y}) 
= 1^r \cdot (1 + \widetilde{y})^s \cdot  \widetilde{y}^t \cdot \psi(1,1 + \widetilde{y}) 
} 
\\ 
&\quad 
= \widetilde{y}^t \cdot (1 + \widetilde{y})^s \cdot (1 + c_1\widetilde{y} + \cdots + c_u\widetilde{y}^u) 
\\ 
&\quad 
= \widetilde{y}^t \cdot (1 + \widetilde{y})^s \cdot \left\{ 
\left(1 + \sum\nolimits_{i=1}^\alpha\gamma_i\widetilde{y}^{ip^e}\right) 
/(1 + \widetilde{y})^s - (c_{u+1}\widetilde{y}^{u+1} + \cdots)\right\} 
\\ 
&\quad 
= \widetilde{y}^t \cdot \left\{ 
\left(1 + \sum\nolimits_{i=1}^\alpha\gamma_i\widetilde{y}^{ip^e}\right) 
- (1 + \widetilde{y})^s(c_{u+1}\widetilde{y}^{u+1} + \cdots)\right\}. 
\end{align*} 
Therefore (3) of the claim above implies 
$$\mathrm{res\text{-}ord}^{(p^e)}_{\widetilde{P}}\left(\pi^*(\Phi(x,y))/x^d\right) = t + v_o.$$ 
Now, from the computation carried out in Step 1, we have 
\begin{align*} 
w\text{-}\rho_{\widetilde{x}}(\widetilde{P}) \cdot p^e 
&= \mathrm{res\text{-}ord}^{(p^e)}_{\widetilde{P}}\left(\widetilde{x}^{d - p^e} \cdot \pi^*(\Phi(x,y))/x^d\right) - (d - p^e) \\ 
&= \mathrm{res\text{-}ord}^{(p^e)}_{\widetilde{P}}\left(\pi^*(\Phi(x,y))/x^d\right), 
\end{align*} 
where the second equality holds because $d$ is a multiple of $p^e$ and because $\Phi(x,y)$ is not a $p^e$-th power. 
Therefore, we conclude that 
\begin{align*} 
&\left\{w\text{-}\rho_{\widetilde{x}}(\widetilde{P}) - \mathrm{inv}_{\mathrm{MON,\heartsuit}}(P)\right\} \cdot p^e 
\\ &\quad 
= \mathrm{res\text{-}ord}^{(p^e)}_{\widetilde{P}}\left(\pi^*(\Phi(x,y))/x^d\right) - \left(\mathrm{ord}_P(a_{p^e}) - \mathrm{ord}_P(M_{\mathrm{tight}}) \cdot p^e)\right) 
\\ &\quad 
= (t + v_o) - \left(d - (r + \mathrm{H}(\xi_{H_y}) \cdot p^e)\right) \leq (t + v_o) - (d - (r + s)) 
\\ &\quad 
= (t + v_o) - (t + u) = v_o - u \leq p^{e-1}, 
\end{align*} 
proving the inequality $(\clubsuit 2)$. 
\end{proof} 
This finishes the proof of Proposition \ref{MHjump-prop}. 
\end{proof} 
\end{subsection} 
\begin{subsection}{Eventual decrease}\label{6.2} 
From the view point that we would like to use 
the new invariant $\mathrm{inv}_{\mathrm{MON,\spadesuit}}$ as 
an effective measure to see what is improved under the transformations specified by the algorithm and that it should strictly decrease after each transformation, the occurrence of the ``Moh-Hauser jumping 
phenomenon'' is a bad news, since the new invariant strictly increases after some special transformations, which are necessarily esoteric by Proposition \ref{invbehavior-prop}. 
We gave a detailed analysis of the esoteric transformations 
in \ref{6.1}.  Here in \ref{6.2}, using the analysis done in \ref{6.1}, 
we show that, when $\mathrm{inv}_{\mathrm{MON,\spadesuit}}$ 
goes through the ``Moh-Hauser jumping phenomenon'', it strictly decreases to a value lower than the original one after some more transformations.  This 
\emph{eventual decrease} as presented in the following proposition 
is enough to show the termination of our algorithm 
(cf.~\ref{5.3}). 
 
\smallskip 
 
\noindent \fbox{\rm \textbf{Situation}} 
 
\smallskip 
 
Consider a transformation in the procedure specified by the algorithm 
given in \ref{5.1} 
(in the monomial case with $\tau = 1$) $P \in \mathrm{Sing}({\mathcal R}) \subset W \overset{\pi}\longleftarrow \widetilde{P} \in \pi^{-1}(P) \cap \mathrm{Sing}(\widetilde{\mathcal R}) \subset \widetilde{W}$ where the new invariant strictly increases, i.e., 
$\mathrm{inv}_{\mathrm{MON,\spadesuit}}(P) < \mathrm{inv}_{\mathrm{MON,\spadesuit}}(\widetilde{P})$ 
 (while the invariant $\sigma$ stays the same and hence we remain in the monomial case with $\tau = 1$ after the transformation).  By Proposition \ref{invbehavior-prop} 
we observe that the transformation $\pi$ satisfies the following properties: 
\begin{itemize} 
\setlength{\parskip}{0pt} 
\setlength{\itemsep}{0pt} 
\item 
the transformation $\pi$ is necessarily esoteric, 
where the center of blow up is the point $P$, which is bad, and 
\item 
the point $P \in \mathrm{Sing}({\mathcal R}) \subset W$ is 
in configuration \textcircled{\footnotesize 4} or 
\textcircled{\footnotesize 5}, while the point 
$\widetilde{P} \in \mathrm{Sing}(\widetilde{\mathcal R}) 
\subset \widetilde{W}$ is in configuration \textcircled{\footnotesize 3}. 
\end{itemize} 
Observe by Proposition \ref{invbehavior-prop} that, after each transformation going from configuration \textcircled{\footnotesize 3} to configuration \textcircled{\footnotesize 3}, 
the invariant $\mathrm{inv}_{\mathrm{MON,\spadesuit}}$ strictly decreases. 
Therefore, after the closed point $\widetilde{P}$ followed by finitely many points in configuration \textcircled{\footnotesize 3}, we reach the last closed point $P^{\triangle}$ in configuration \textcircled{\footnotesize 3}, which is followed by a closed point $P^{\sharp}$ in configuration \textcircled{\footnotesize 4} or \textcircled{\footnotesize 5} (unless we reach Case 1, 2, or 3 as described in Theorem 1 in the process): 
\begin{table}[hbtp] 
\centering 
\begin{tabular}{r|ccccccccc} 
closed points & $P$ & $\leftarrow$ & $\widetilde{P}$ & $\leftarrow$ 
& $\cdots$ & $\leftarrow$ & $P^{\flat}$ & $\leftarrow$ & $P^{\sharp}$ \\ 
\hline 
configuration & 
\textcircled{\footnotesize 4}  or \textcircled{\footnotesize 5} 
&& 
\textcircled{\footnotesize 3} 
&& all in \textcircled{\footnotesize 3} 
&& \textcircled{\footnotesize 3} 
&& \textcircled{\footnotesize 4} or \textcircled{\footnotesize 5}. 
\end{tabular} 
\end{table} 
 
\noindent 
By Proposition \ref{invbehavior-prop} again we have the following inequalities among the values of the new invariant 
$$\mathrm{inv}_{\mathrm{MON,\spadesuit}}(P) < \mathrm{inv}_{\mathrm{MON,\spadesuit}}(\widetilde{P}) > \cdots > \mathrm{inv}_{\mathrm{MON,\spadesuit}}(P^{\flat}) > \mathrm{inv}_{\mathrm{MON,\spadesuit}}(P^{\sharp}).$$ 
\begin{art_prop} 
Under the situation above, we have the following inequalities. 
$$\mathrm{inv}_{\mathrm{MON,\spadesuit}}(P) > \mathrm{inv}_{\mathrm{MON,\heartsuit}}(P) > \mathrm{inv}_{\mathrm{MON,\spadesuit}}(P^{\sharp}).$$ 
In particular, 
compared to the original value at the closed point $P$, 
the value of the new invariant $\mathrm{inv}_{\mathrm{MON,\spadesuit}}$ 
strictly decreases at the closed point $P^{\sharp}$, 
i.e., we have the \emph{eventual decrease}. 
\end{art_prop} 
\begin{proof} 
The first inequality 
$\mathrm{inv}_{\mathrm{MON,\spadesuit}}(P) > 
\mathrm{inv}_{\mathrm{MON,\heartsuit}}(P)$ 
follows easily from Remark 5 (2) and the definition of 
the lexicographical order (cf.~Definition \ref{invs-def}). 
 
We show the following 
inequalities 
$$ 
(\ast)\quad 
2m^\heartsuit>\widetilde{m}\geq 
m^\flat>w\text{-}\rho_{x^{\sharp}}(P^{\sharp}) + m^\natural, 
$$ 
where we set 
$m^{\heartsuit} = \mathrm{inv}_{\mathrm{MON,\heartsuit}}(P)$, 
$\widetilde{m} = w\text{-}\rho_{\widetilde{x}}(\widetilde{P})$, 
$m^{\flat} = w\text{-}\rho_{x^{\flat}}(P^{\flat})$ and 
$$m^{\natural} = 
\begin{cases} 
\mathrm{ord}_{P^{\sharp}}(M^{\sharp}_{\mathrm{usual}}) 
- \mathrm{ord}_{P^{\sharp}}(M^{\sharp}_{\mathrm{tight}}) 
& 
\text{when $P^{\sharp}$ is in configuration 
\textcircled{\footnotesize 4}},  \\ 
w\text{-}\rho_{y^{\sharp}}(P^{\sharp}) 
& 
\text{when $P^{\sharp}$ is in configuration 
\textcircled{\footnotesize 5}}. 
\end{cases} 
$$ 
Note that the the second inequality 
$m^{\heartsuit} = \mathrm{inv}_{\mathrm{MON,\heartsuit}}(P) > 
 \mathrm{inv}_{\mathrm{MON,\spadesuit}}(P^{\sharp})$ 
follows easily from the inequality 
$2m^{\heartsuit} > w\text{-}\rho_{x^{\sharp}}(P^{\sharp}) + m^\natural$ 
and the definitions. 
 
The remaining part of the proof is devoted to 
showing the inequalities $(\ast)$. 
\begin{proof}[Proof for $2 m^{\heartsuit}>\widetilde{m}$]\quad 
Assume $\widetilde{m} \geq 2 m^{\heartsuit}$.  Then we have 
$$ 
p^e > \left(\widetilde{m} - m^{\heartsuit}\right) \cdot p^e 
\geq m^{\heartsuit} \cdot p^e 
$$ 
by $(\clubsuit 1)$ or $(\clubsuit 2)$ in Proposition \ref{MHjump-prop} and 
by the assumption. 
Recall that, by Proposition \ref{MHjump-prop}, the initial form 
$\mathrm{In}_P(a_{p^e}) = \Phi(x,y)$ of the last coefficient 
$a_{p^e}$ of $h$ is a homogeneous polynomial of degree 
$d =np^e$ ($n\in{\mathbb Z}_{>1}$) of the following form 
$$\Phi(x,y) = x^r \cdot y^s \cdot (y - x)^t \cdot \psi(x,y)$$ 
where $\psi(x,y)$ is a homogeneous polynomial of degree $u$ and 
\begin{align*} 
r = \mathrm{ord}_{\xi_{H_x}}(a_{p^e}) = \mathrm{H}(\xi_{H_x}) \cdot p^e, 
\quad 
s = \lceil \mathrm{H}(\xi_{H_y}) \cdot p^e \rceil 
\geq \mathrm{H}(\xi_{H_y}) \cdot p^e, 
\quad 
t = l \cdot p^e 
\end{align*} 
with $l \in {\mathbb Z}_{\geq 0}$. 
Observe 
\begin{align*} 
m^{\heartsuit} \cdot p^e 
&= \mathrm{inv}_{\mathrm{MON,\heartsuit}}(P) \cdot p^e = \left(\mathrm{H}(P) - \mathrm{ord}_P(M_{\mathrm{tight}})\right) \cdot p^e \\ 
&= 
\mathrm{ord}_P(a_{p^e}) - (\mathrm{H}(\xi_{H_x}) \cdot p^e + \mathrm{H}(\xi_{H_y}) \cdot p^e) 
\qquad({\text{since }P\text{ is bad}}) 
\\ 
&\geq d - (r + s) = u + t = u + l \cdot p^e. 
\end{align*} 
Therefore, we have $p^e > m^{\heartsuit} \cdot p^e \geq u + l \cdot p^e$.  Writing $u = \alpha \cdot p^e + \beta$ with $\alpha, \beta \in {\mathbb Z}_{\geq 0}$ and $0 \leq \beta < p^e$, we conclude $l = 0,\alpha = 0$ and hence $u = \beta$.  On the other hand, by the inequality $(\diamondsuit)$ in 
Proposition \ref{MHjump-prop}, we have $u = \beta < p^e - s$ and hence $s + u < p^e$.  However, since $r < p^e$, this implies $d = \deg \Phi(x,y) = \deg (x^ry^s\psi(x,y)) = r + s + u < 2p^e$, which contradicts the description of $d$ above. 
This finishes the proof of the inequality $2 m^{\heartsuit}>\widetilde{m}$. 
\end{proof} 
\begin{proof}[Proof for 
$\widetilde{m}\geq m^{\flat}$]\quad 
Since the invariant $w\text{-}\rho$ never increases during a sequence of transformations with consecutive points in configuration \textcircled{\footnotesize 3}, we have the inequality $\widetilde{m} \geq m^{\flat}$. 
\end{proof} 
\begin{proof}[Proof for $m^{\flat} \geq 
w\text{-}\rho_{x^{\sharp}}(P^{\sharp}) + m^\natural$]\quad 
We first remark that 
\begin{center} 
$P^{\sharp}$ is in configuration 
\textcircled{\footnotesize 4} (resp.~\textcircled{\footnotesize 5}) 
$\Longleftrightarrow$ $P^{\flat}$ is good (resp.~bad). 
\end{center} 
We set 
$$ 
\mu^\natural= 
\begin{cases} 
\mu(P^\flat) &\text{when $P^\flat$ is good}, 
\\ 
d^\flat/p^e &\text{when $P^\flat$ is bad}. 
\end{cases} 
$$ 
We show the following relations (0), (1) and (2), 
which imply the inequality 
$m^{\flat} \geq w\text{-}\rho_{x^{\sharp}}(P^{\sharp}) + m^\natural$ 
immediately: 
\begin{itemize} 
\setlength{\parskip}{0pt} 
\item[(0)] 
${\begin{cases} 
\mathrm{H}(\xi_{H_{x^{\sharp}}}) = \mathrm{H}(\xi_{H_{x^{\flat}}}) = r^{\flat}/p^e, 
&\mathrm{H}(\xi_{H_{y^{\sharp}}}) = \mu^{\natural} - 1, \\ 
\mu(\xi_{H_{x^{\sharp}}}) = \mu(\xi_{H_{x^{\flat}}}) = \mu(P^{\flat}), 
&\mu(\xi_{H_{y^{\sharp}}}) = \mu(P^{\flat}) - 1. 
\end{cases}}$ 
\item[(1)] 
$m^\natural 
\leq \mu^{\natural} - r^{\flat}/p^e$ 
\quad(\text{the equality holds when $P^\flat$ is good 
}), 
\item[(2)] 
$w\text{-}\rho_{x^{\sharp}}(P^{\sharp}) = m^{\flat} 
- (\mu^{\natural} - r^{\flat}/p^e)$. 
\end{itemize} 
 
\smallskip 
 
The proof of the relations above 
in the case when $P^\flat$ is good 
is almost identical to the one in the case when $P^\flat$ is bad. 
(We draw the attention of the reader to the fact that (1) in the former case 
is an equality, while (1) in the latter case is an inequality.) 
Therefore, in the following, we only give a proof in the latter case when $P^\flat$ is bad. 
 
\smallskip 
 
The last two equalities about the invariant $\mu$ in (0) are obvious.  We will verify the remaining relations. 
Recall that we are in the monomial case with $\tau = 1$ 
at the closed point $P^{\flat}$. 
Let $(h^{\flat},p^e)$ be the unique element in the L.G.S. ${\mathbb H}^{\flat}$ of the idealistic filtration ${\mathcal R}^{\flat}$ at $P^{\flat}$.  Suppose that $h^{\flat}$ is in the following Weierstrass form $h^{\flat} = (z^{\flat})^{p^e} + a_1^{\flat}(z^{\flat})^{p^e-1} + \cdots + a_{p^e - 1}^{\flat}(z^{\flat}) + a_{p^e}^{\flat}$, where $a_i^{\flat} \in k[[x^{\flat},y^{\flat}]]$ with $\mathrm{ord}_{P^{\flat}}(a_i^{\flat}) > i \hskip.1in (i = 1, \ldots, p^e)$ 
for a regular system of parameters $(z^{\flat},x^{\flat},y^{\flat})$ of $\widehat{{\mathcal O}_{W^{\flat},P^{\flat}}}$, and that the divisor $H_{x^{\flat}} = \{x^{\flat} = 0\}$ is the bad divisor in configuration \textcircled{\footnotesize 3}.  Suppose further that $h^{\flat}$ is well-adapted at $P^{\flat}$ and at the generic point of the divisor $H_{x^{\flat}}$ simultaneously.  Write $a_{p^e}^{\flat} = (x^{\flat})^{r^{\flat}} \cdot \left\{g^{\flat}(y^{\flat}) + x^{\flat} \cdot \omega_{x^{\flat}}(x^{\flat},y^{\flat})\right\}$ where 
$r^{\flat} = \mathrm{H}(\xi_{H_{x^{\flat}}}) \cdot p^e, 0 \neq g^{\flat}(y^{\flat}) \in k[[y^{\flat}]]$, and $\omega_{x^{\flat}}(x^{\flat},y^{\flat}) \in k[[x^{\flat},y^{\flat}]]$.  Note that $\mathrm{In}_{\xi_{H_{x^{\flat}}}}(a_{p^e}^{\flat}) = (x^{\flat})^{r^{\flat}} \cdot g^{\flat}(y^{\flat})$ is not a $p^e$-th power, since $h^{\flat}$ is well-adapted at $\xi_{H_{x^{\flat}}}$ with respect to $(z^{\flat},x^{\flat},y^{\flat})$.  We may further 
assume (cf.~Proposition 1 (3)) that the equality $(*^{\flat})\ \mathrm{res\text{-}ord}^{(p^e)}_{P^{\flat}}\left((x^{\flat})^{r^{\flat}} \cdot g^{\flat}(y^{\flat})\right)/p^e = \mathrm{ord}_{P^{\flat}}\left((x^{\flat})^{r^{\flat}} \cdot g^{\flat}(y^{\flat})\right)/p^e$ holds.  Let $\mathrm{In}_{P^{\flat}}(a_{p^e}^{\flat}) = \Phi^{\flat}(x^{\flat},y^{\flat})$ be the initial form (i.e., the lowest degree homogeneous part) of the last coefficient $a_{p^e}^{\flat}$ of the element $h^{\flat}$.  Note that neither $\mathrm{In}_{P^{\flat}}(a_{p^e}^{\flat}) = \Phi^{\flat}(x^{\flat},y^{\flat})$ nor $\mathrm{In}_{\xi_{x^{\flat}}}(a_{p^e}^{\flat}) =  (x^{\flat})^{r^{\flat}} \cdot g^{\flat}(y^{\flat})$ is a $p^e$-th power, since $h^{\flat}$ is well-adapted at $P^{\flat}$ and $\xi_{H_{x^{\flat}}}$ simultaneously with respect to $(z^{\flat},x^{\flat},y^{\flat})$. 
 
When we take the transformation $P^{\flat} \in \mathrm{Sing}({\mathcal R}^{\flat}) \subset W^{\flat} \overset{\pi^{\sharp}}\longleftarrow P^{\sharp} \in {\pi^{\sharp}}^{-1}(P^{\flat}) \cap \mathrm{Sing}({\mathcal R}^{\sharp}) \subset W^{\sharp}$ with the center of blow up being the point $P^{\flat}$, the closed point $P^{\sharp}$ in configuration \textcircled{\footnotesize 5} has a regular system of parameters $(z^{\sharp},x^{\sharp},y^{\sharp}) = (z^{\flat}/y^{\flat},x^{\flat}/y^{\flat},y^{\flat})$, and we compute the transformation 
$$h^{\sharp} = {\pi^{\sharp}}^*(h^{\flat})/(y^{\flat})^{p^e} = (z^{\sharp})^{p^e} + a^{\sharp}_1(z^{\sharp})^{p^e-1} + \cdots + a^{\sharp}_{p^e-1}(z^{\sharp}) + a^{\sharp}_{p^e}$$ 
where $a^{\sharp}_i = {\pi^{\sharp}}^*(a_i^{\flat})/(y^{\flat})^i$ for $i = 1, \ldots, p^e - 1, p^e$, and where $\mathrm{ord}_{P^{\sharp}}(a^{\sharp}_i) > i$ for $i = 1, \ldots, p^e - 1, p^e$, since $P^{\sharp} \in \mathrm{Sing}({\mathcal R}^{\sharp})$ and the invariant $\sigma$ stays the same at $P^{\sharp}$ and since the equality $(*^{\flat})$ holds. 
 
\noindent Observe that 
\begin{itemize} 
\setlength{\parskip}{0pt} 
\setlength{\itemsep}{0pt} 
\item[$\circ$] 
the divisor $H_{y^{\sharp}} = \{y^{\sharp} = 0\}$ is bad, since $P^{\flat}$ is bad, and 
\item[$\circ$] 
the divisor $H_{x^{\sharp}} = \{x^{\sharp} = 0\}$ is bad, since it is the strict transform of the bad divisor $H_{x^{\flat}} = \{x^{\flat} = 0\}$. 
\end{itemize} 
We compute 
$$a^{\sharp}_{p^e} = (y^{\sharp})^{d^{\flat} - p^e} \cdot\left\{\Phi^{\flat}(x^{\sharp},1) + y^{\sharp} \cdot \omega_{y^{\sharp}}(x^{\sharp},y^{\sharp})\right\},$$ 
where $\mathrm{In}_{\xi_{H_{y^{\sharp}}}}\left(a^{\sharp}_{p^e}\right) = (y^{\sharp})^{d^{\flat} - p^e} \cdot \Phi^{\flat}(x^{\sharp},1) = \Phi^{\flat}(x^{\flat},y^{\flat})/(y^{\flat})^{p^e}$ is not a $p^e$-th power because $\mathrm{In}_{P^{\flat}}(a_{p^e}^{\flat}) = \Phi^{\flat}(x^{\flat},y^{\flat})$ is not, and where $\omega_{y^{\sharp}}(x^{\sharp},y^{\sharp}) \in k[[x^{\sharp},y^{\sharp}]]$. 
 
We also compute 
$$a^{\sharp}_{p^e} = (x^{\sharp})^{r^{\flat}} \cdot (y^{\sharp})^{r^{\flat} - p^e} \cdot \{g^{\flat}(y^{\sharp}) + x^{\sharp} \cdot y^{\sharp} \cdot \omega_{x^{\flat}}(x^{\sharp} \cdot y^{\sharp},y^{\sharp})\},$$ 
where $\mathrm{In}_{\xi_{H_{x^{\sharp}}}}\left(a^{\sharp}_{p^e}\right) = (x^{\sharp})^{r^{\flat}} \cdot (y^{\sharp})^{r^{\flat} - p^e} \cdot g^{\flat}(y^{\sharp}) = (x^{\flat})^{r^{\flat}} \cdot g^{\flat}(y^{\flat})/(y^{\flat})^{p^e}$ 
is not a $p^e$-th power because $\mathrm{In}_{\xi_{H_{x^{\flat}}}}(a_{p^e}^{\flat}) = (x^{\flat})^{r^{\flat}} \cdot g^{\flat}(y^{\flat})$ is not. 
 
We also have the inequalities 
$$\begin{cases} 
\ \mathrm{ord}_{\xi_{H_{x^{\sharp}}}}(a^{\sharp}_{p^e}) 
= r^{\flat} < 
\mu(\xi_{H_{x^{\flat}}}) \cdot p^e = \mu(\xi_{H_{y^{\sharp}}}) \cdot p^e 
&\text{(since $H_{x^{\flat}}$ is bad)} 
,\\ 
\ \mathrm{ord}_{\xi_{H_{y^{\sharp}}}}(a^{\sharp}_{p^e}) 
= d^{\flat} - p^e < 
\mu(P^{\flat}) \cdot p^e - p^e = \mu(\xi_{H_{y^{\sharp}}}) \cdot p^e 
&\text{(since $P^{\flat}$ is bad)}. 
\end{cases}$$ 
 
Therefore, $h^{\sharp}$ is well-adapted at $\xi_{H_{x^{\sharp}}}$ and $\xi_{H_{y^{\sharp}}}$ simultaneously with respect to $(z^{\sharp},x^{\sharp},y^{\sharp})$.  Therefore, we conclude that the first two equalities in (0) hold 
$$\mathrm{H}(\xi_{H_{x^{\sharp}}}) = \mathrm{H}(\xi_{H_{x^{\flat}}}) = r^{\flat}/p^e, 
\quad 
\mathrm{H}(\xi_{H_{y^{\sharp}}}) = (d^{\flat} - p^e)/p^e 
= d^{\flat}/p^e - 1= \mu^\natural - 1.$$ 
We also conclude that the inequality in (1) holds 
\begin{align*} 
w\text{-}\rho_{y^{\sharp}}(P^{\sharp}) 
&= \left[\mathrm{res\text{-}ord}^{(p^e)}_{P^{\sharp}}\left((y^{\sharp})^{d^{\flat} - p^e} \cdot \Phi^{\flat}(x^{\sharp},1)\right) - \mathrm{ord}_{P^{\sharp}}((x^{\sharp})^{r^{\flat}} \cdot (y^{\sharp})^{d^{\flat} - p^e})\right]/p^e \\ 
&\leq \left[(d^{\flat} - p^e) + d^{\flat}\} - \{r^{\flat} + (d^{\flat} - p^e)\right]/p^e 
= (d^{\flat} - r^{\flat})/p^e= \mu^\natural - r^{\flat}/p^e. 
\end{align*} 
We compute to check the equality in (2) 
\begin{align*} 
\lefteqn{w\text{-}\rho_{x^{\sharp}}(P^{\sharp})} 
\\ 
&\quad 
= \left[\mathrm{res\text{-}ord}^{(p^e)}_{P^{\sharp}}\left((x^{\sharp})^{r^{\flat}} \cdot (y^{\sharp})^{r^{\flat} - p^e} \cdot g^{\flat}(y^{\sharp})\right) - \mathrm{ord}_{P^{\sharp}}\left((x^{\sharp})^{r^{\flat}} \cdot (y^{\sharp})^{d^{\flat} - p^e}\right)\right]/p^e \\ 
&\quad 
= \left[\mathrm{ord}_{P^{\sharp}}\left((x^{\sharp})^{r^{\flat}} \cdot (y^{\sharp})^{r^{\flat} - p^e} \cdot g^{\flat}(y^{\sharp})\right) - \mathrm{ord}_{P^{\sharp}}\left((x^{\sharp})^{r^{\flat}} \cdot (y^{\sharp})^{d^{\flat} - p^e}\right)\right]/p^e \\ 
&\quad 
= \left[\{r^{\flat} + (r^{\flat} - p^e) + m^{\flat} \cdot p^e\} - \{r^{\flat} + (d^{\flat} - p^e)\}\right]/p^e 
\\ 
&\quad 
= m^{\flat} - (d^{\flat} - r^{\flat})/p^e 
= m^{\flat} - (\mu^{\natural} - r^{\flat}/p^e). 
\end{align*} 
In order to obtain the second and third equalities above, we use the equalities 
$(\ast^{\flat})$ and 
$m^{\flat} = w\text{-}\rho_{x^{\flat}}(P^{\flat}) = \mathrm{ord}_{P^{\flat}}g^{\flat}(y^{\flat})$. 
This completes the proof of the relations (0), (1), (2), and hence 
the proof of the inequality 
$m^{\flat} \geq w\text{-}\rho_{x^{\sharp}}(P^{\sharp}) + m^\natural$. 
\end{proof} 
This completes the proof of Proposition 7. 
\end{proof} 
\end{subsection} 
\end{section} 
\begin{section}{Comparison of the old invariant with the new invariant} 
\begin{subsection}{Comparison Table} 
We present the table comparing the descriptions of 
the old invariant $\mathrm{inv}_{\mathrm{MON}}$ in \cite{KM2} and 
the new invariant $\mathrm{inv}_{\mathrm{MON,\spadesuit}}$ in this paper. 
\begin{table}[hbtp] 
%\caption{Comparison of $\mathrm{inv}_{\mathrm{MON}}$ and 
%$\mathrm{inv}_{\mathrm{MON,\spadesuit}}$} 
\label{comparison} 
\centering 
\begin{tabular}{c|c||cc|c} 
\multicolumn{2}{c||}{Configuration} 
&\multicolumn{3}{|c}{Description} 
\\ 
\hline 
\multicolumn{5}{l}{$\mathrm{inv}_{\mathrm{MON}}\phantom{\biggl\{}$} 
\\ \cline{2-4}\cline{2-4} 
& \textcircled{\footnotesize 1} & \multicolumn{2}{c|}{$(0,0,\mu_x)$ 
}&\\ \cline{2-4} 
& \textcircled{\footnotesize 2} & \multicolumn{2}{c|}{ 
$(0,0,\min\{\mu_x,\mu_y\}, \max\{\mu_x,\mu_y\})$ 
}&\\ \cline{2-4} 
& \textcircled{\footnotesize 3} & \multicolumn{2}{c|}{ 
$(\rho_x,0,\mu_x)$ 
}&\\ \cline{2-4} 
& \textcircled{\footnotesize 4} & \multicolumn{2}{c|}{ 
$(\min\{\rho_x,\mu_x\}, \max\{\rho_x,\mu_x\})$ 
}&\\ \cline{2-4} 
& \textcircled{\footnotesize 5} & \multicolumn{2}{c|}{ 
$(\min\{\rho_x,\rho_y\}, \max\{\rho_x,\rho_y\})$ 
}&\\ 
\cline{2-4} 
\multicolumn{5}{l}{$\mathrm{inv}_{\mathrm{MON,\spadesuit}}\phantom{\biggl\{}$ 
 } 
\\ 
\cline{2-4}\cline{2-4} 
& \textcircled{\footnotesize 1} & \multicolumn{2}{c|}{0}&\\ \cline{2-4} 
& \textcircled{\footnotesize 2} & \multicolumn{2}{c|}{0}&\\ \cline{2-4} 
& \textcircled{\footnotesize 3} & \multicolumn{2}{l|}{ 
$ 
\begin{array}{l} 
\mathrm{lex}\left\{w\text{-}\rho_x, \mathrm{ord}_P(M_{\text{usual}}) 
- \mathrm{ord}_P(M_{\text{tight}})\right\} 
\\ \quad 
= \mathrm{lex}\left\{w\text{-}\rho_x,w\text{-}\mu_x\right\} 
\end{array} 
$ 
}&\\ \cline{2-4} 
& \textcircled{\footnotesize 4} & \multicolumn{2}{l|}{ 
$ 
\begin{array}{l} 
\mathrm{lex}\left\{w\text{-}\rho_x, 
\mathrm{ord}_P(M_{\text{usual}}) - \mathrm{ord}_P(M_{\text{tight}})\right\} 
\\ \quad 
= \mathrm{lex}\left\{w\text{-}\rho_x, w\text{-}\mu_x\right\} 
\\ \quad 
= (\min\{w\text{-}\rho_x,w\text{-}\mu_x\}, 
\max\{w\text{-}\rho_x,w\text{-}\mu_x\}) 
\end{array} 
$ 
}&\\ \cline{2-4} 
& \textcircled{\footnotesize 5} & \multicolumn{2}{l|}{ 
$ 
\begin{array}{l} 
\mathrm{lex}\left\{\mathrm{lex}\left(w\text{-}\rho_x,w\text{-}\rho_y\right), 
\mathrm{ord}_P(M_{\text{usual}}) - \mathrm{ord}_P(M_{\text{tight}})\right\} 
\\ \quad 
 = \mathrm{lex}\left\{ 
\begin{array}{l} 
(\min\{w\text{-}\rho_x,w\text{-}\rho_y\}, \max\{w\text{-}\rho_x,w\text{-}\rho_y\}), 
\\ 
\mathrm{ord}_P(M_{\text{usual}}) -\mathrm{ord}_P(M_{\text{tight}}) 
\end{array} 
\right\} 
\end{array} 
$ 
}&\\ 
\cline{2-4} 
\end{tabular} 
\end{table} 
 
\smallskip 
 
\noindent \textbf{Observation} 
 
\smallskip 
 
\noindent 1. We hope that the reader can sense the analogies and similarities between the old invariant $\mathrm{inv}_{\mathrm{MON}}$ and 
the new invariant $\mathrm{inv}_{\mathrm{MON,\spadesuit}}$ from the table. 
The major differences are: 
 
\begin{itemize} 
\setlength{\parskip}{0pt} 
\item In the new invariant, we use the ``weak'' version of the invariants, 
indicated by the prefix ``$w\text{-}$'', where we subtract the appropriate amount of contributions coming from the boundary divisors. 
 
\item In the new invariant, the number 
$\mathrm{ord}_P(M_{\text{usual}}) - \mathrm{ord}_P(M_{\text{tight}})$, 
which measures the difference between the usual monomial and the tight monomial, shows up always and more explicitly, where in the old invariant, the number is disguised as 
$\mu_x$ (in configuration \textcircled{\footnotesize 3} 
or \textcircled{\footnotesize 4}) or does not show up 
(in configuration \textcircled{\footnotesize 5}). 
\end{itemize} 
 
\noindent 2. The original argument of Benito-Villamayor \cite{BV3} (for showing the termination of the algorithm for resolution of singularities in the monomial case in dimension 3) was quite complicated.  The old invariant 
$\mathrm{inv}_{\mathrm{MON}}$ is the result of our search 
for a simple invariant which measures effectively what is improved in the process of the algorithm.  Even when the first author finally came up with the description of 
$\mathrm{inv}_{\mathrm{MON}}$ in December of 2012 in Austria, 
it remained something of a mystery to us why it works.  For example, we were wondering why we look at the invariant $\mu_x$ in configuration \textcircled{\footnotesize 4} associated to the good divisor $H_y = \{y = 0\}$, while we look at the invariant $\rho_y$ in configuration \textcircled{\footnotesize 5} associated to the bad divisor $H_y = \{y = 0\}$.  The interpretation of the old invariant 
$\mathrm{inv}_{\mathrm{MON}}$ in terms of the new invariant 
$\mathrm{inv}_{\mathrm{MON,\spadesuit}}$, where the latter is 
based upon the original philosophy of Villamayor, demystifies the mechanism 
to some extent, e.g. demystifies the use of $\mu_x$ as the disguise of the number 
$\mathrm{ord}_P(M_{\text{usual}}) - \mathrm{ord}_P(M_{\text{tight}})$. 
\end{subsection} 
\begin{subsection}{Why ``$w\text{-}$'', i.e.,  the ``weak'' version of the invariant?} 
Of course this demystification comes with a seemingly hefty price: The old invariant strictly decreases after each transformation, while the new one strictly increases from time to time.  We were happy that, using the old invariant, we did not encounter the ``Moh-Hauser jumping 
phenomenon'', one of the well-known obstacles toward establishing an algorithm for resolution of singularities in positive characteristic.  So why are we looking at the new invariant, using ``$w\text{-}$'', which inevitably brings back the ``Moh-Hauser jumping 
phenomenon'' ?  The motivation lies in our attempt to go into higher dimensions. 
 
\smallskip 
 
\noindent 
1. \textit{``weak order'' vs ``usual order''}: 
 
\smallskip 
 
It is well-known and easy to observe that 
the (usual) order of an ideal ${\mathcal I}$ on a nonsingular variety $W$ may strictly increase after blowing up along a smooth center $C$ (even when $C$ is taken within the locus of highest order) if $\dim W > 1$, while the weak order does not increase.  This is the very reason, when we face the problem of resolution of singularities of the triplet $(W,({\mathcal I},a),E)$ in characteristic zero, why we use the weak order and not 
the usual order (cf.~\ref{1.2}). 
 
\smallskip 
 
\noindent 2. \textit{Why the old invariant does not increase after blow up in dimension 3, even though we are not using the weak version?} 
 
\smallskip 
 
Consider the definition of the invariant $\rho_x$.  After the process of cleaning, we look at the last coefficient $a_{p^e} = x^r \cdot \{g_x(y) + x \cdot \omega_x(x,y)\}$ and the invariant $\rho_x$ is computed using the usual order $\rho_x = \mathrm{ord}_P(g_x(y))/p^e$. 
The reason why the old invariant does not increase after blow up, roughly speaking, is that 
the invariant $\rho_x$ is computed using the order on $V = \{z = x = 0\}$ of $\dim V = 1$, the only dimension where the usual order does not increase after blow up !  It is this beneficial peculiarity, unique to the low dimension $\dim W = 3$, that guarantees the old invariant does not increase. 
 
\smallskip 
 
\noindent 3. \textit{What happens if we try to go into higher dimensions?} 
 
\smallskip 
 
Let's consider what happens when $\dim W = 4$.  Then the last coefficient is of the form, after cleaning, $a_{p^e} = x^r \cdot \{g_x(y,z) + x \cdot \omega(x,y,z)\}$ with respect to a regular system of parameters $(w,x,y,z)$, where $h$ is in the Weierstrass form $h = w^{p^e} + a_1w^{p^e - 1} + \cdots + a_{p^e}$ with $a_i \in k[[x,y,z]]$ and $\mathrm{ord}_P(a_i) > i$ for $i = 1, \ldots, p^e$.  If we define the invariant $\rho_x$, without using the weak version, by the formula $\rho_x = \mathrm{ord}_P(g_x(y,z))/p^e$, then it could easily increase strictly after a point blow up, as the usual order on the surface $V = \{w = x = 0\}$ of $\dim V = 2$ could increase strictly after a point blow up. 
This is why we want to use the weak version $w\text{-}\rho_x$, subtracting the appropriate amount of contributions coming from the boundary divisors, leading to the use of 
the weak version of the invariants in higher dimensions. 
Actually the new invariant in dimension 3 arose from our attempt to generalize our method in \cite{KM2} to the case in dimension 4.  Our hope is that we can carry out an analysis of the ``Moh-Hauser jumping 
phenomenon'' leading to the eventual decrease in dimension 4, similar to the one in dimension 3, and hence that we can show the termination of the algorithm for resolution of singularities in the monomial case, whose choice of the center, nevertheless, is much more involved in dimension 4 than in dimension 3.  Details will be discussed elsewhere. 
\end{subsection} 
\end{section} 
%%%%%%%%%%%%%%%%%%%%%%%%%%%%%%%%% 
% Bibliography 
%%%%%%%%%%%%%%%%%%%%%%%%%%%%%%%%% 

%%%%%%%%%%%%%%%%%%%%%%%%%%%%%%%%% 
% Projects and funding 
% Uncomment and fill in with the information if needed 
%%%%%%%%%%%%%%%%%%%%%%%%%%%%%%%%% 
 
%\projects{\noindent } 
 
\end{document}